\documentclass[a4paper,reqno,12pt]{amsart}
\usepackage[left=1 in, right=1 in,top=1 in, bottom=1 in]{geometry}
\usepackage{amsfonts}
\usepackage{amssymb}
\usepackage{amsthm}
\usepackage{amsmath}
\usepackage{mathrsfs}
\usepackage{color}
\usepackage{enumerate}
\usepackage{hyperref}
\hypersetup{colorlinks=true,linkcolor=black}
\usepackage[numbers,sort&compress]{natbib}
\usepackage{pdfsync}
\usepackage{esint}
\usepackage{graphicx}
\usepackage{float}
\usepackage{caption}
\allowdisplaybreaks

\makeatother

\numberwithin{equation}{section}

\newtheorem{theorem}{Theorem}[section]
\newtheorem{definition}[theorem]{Definition}

\newtheorem{lemma}[theorem]{Lemma}
\newtheorem{remark}[theorem]{Remark}

\newtheorem{proposition}[theorem]{Proposition}
\newtheorem{corollary}[theorem]{Corollary}
\numberwithin{equation}{section}

\newcommand*{\supp}{\ensuremath{\mathrm{supp\,}}}

\newcommand{\sig}{\sigma}
\newcommand{\al}{\alpha}

\newcommand{\pt}{\partial_t}
\newcommand{\ptau}{\partial_\tau}
\renewcommand{\P}{\mathbb{P}}
\newcommand{\R}{\mathbb{R}}
\newcommand{\F}{\mathscr{F}}
\newcommand{\dif}{{\mathrm{d}}}
\renewcommand{\div}{{\mathrm{div}}}
\newcommand{\curl}{{\mathrm{curl}}}
\newcommand{\lss}{{\boldsymbol{L}_{\operatorname{ss}}}}
\newcommand{\lssa}{{\boldsymbol{L}_{\operatorname{ss}}^\al}}

\newcommand{\norm}[1]{\lVert#1\rVert}

\newcommand{\loc}{\operatorname{loc}}
\newcommand{\na}{\nabla}
\newcommand{\lc}{\left(}
\newcommand{\rc}{\right)}
\newcommand{\ul}{U^{\operatorname{lin}}}
\newcommand{\up}{U^{\operatorname{per}}}
\newcommand{\les}{\lesssim}

\newcommand{\D}{\boldsymbol{D}}
\newcommand{\M}{\boldsymbol{M}}
\newcommand{\K}{\boldsymbol{K}}
\newcommand{\T}{\boldsymbol{T}}
\newcommand{\aps}{\operatorname{aps}}
\newcommand{\bs}{\mathrm{BS}_{3d}}

\newcommand{\wt}{\widetilde}
\newcommand{\wh}{\widehat}
\newcommand{\ol}{\overline}
\newcommand{\bbe}{\mathbb{E}}
\newcommand{\bbf}{\mathbb{F}}
\newcommand{\bbn}{\mathbb{N}}
\newcommand{\bbr}{\mathbb{R}}
\newcommand{\bbx}{\mathbb{X}}
\newcommand{\bbp}{\mathbb{P}}
\newcommand{\ve}{\varepsilon}
\def\<{\langle}
\def\>{\rangle}
\newcommand{\calc}{\mathcal{C}}
\newcommand{\cald}{\mathcal{D}}
\newcommand{\calf}{\mathcal{F}}
\newcommand{\calh}{\mathcal{H}}
\newcommand{\cals}{\mathcal{S}}
\newcommand{\calt}{\mathcal{T}}
\newcommand{\calu}{\mathcal{U}}
\newcommand{\calv}{\mathcal{V}}
\newcommand{\calw}{\mathcal{W}}

\newcommand{\calz}{\mathcal{Z}}

\newcommand{\scrx}{\mathscr{X}}
\newcommand{\scry}{\mathscr{Y}}
\def\<{\langle}
\def\>{\rangle}
\newcommand{\vf}{\varphi}
\newcommand{\caln}{\mathcal{N}}

\begin{document}
	
\title[] {Non-uniqueness in law of Leray solutions
to 3D forced stochastic Navier-Stokes equations}

\author{Elia Bru\'{e}}
\address{School of Mathematics, Institute for Advanced Study, USA.}
\email[Elia Bru\'{e}]{elia.brue@ias.edu}
\thanks{}

\author{Rui Jin}
\address{School of Mathematical Sciences, Shanghai Jiao Tong University, China.}
\email[Rui Jin]{jinrui@sjtu.edu.cn}
\thanks{}

\author{Yachun Li}
\address{School of Mathematical Sciences, CMA-Shanghai, MOE-LSC, and SHL-MAC,  Shanghai Jiao Tong University, China.}
\email[Yachun Li]{ycli@sjtu.edu.cn}
\thanks{}

\author{Deng Zhang}
\address{School of Mathematical Sciences, CMA-Shanghai, Shanghai Jiao Tong University, China.}
\email[Deng Zhang]{dzhang@sjtu.edu.cn}
\thanks{}



\begin{abstract}
  This paper concerns the forced stochastic Navier-Stokes equation
  driven by additive noise in the three dimensional Euclidean space.
  By constructing an appropriate forcing term,
  we prove that
  there exist distinct Leray solutions
  in the probabilistically weak sense.
  In particular,
  the joint  uniqueness in law fails in the  Leray class.
  The non-uniqueness also displays in the probabilistically strong sense
in the local time regime, up to stopping times.
Furthermore,
we discuss the optimality from two different perspectives:
sharpness of the hyper-viscous exponent and size of the external force. These results in particular yield that the Lions exponent
is the sharp viscosity threshold for the uniqueness/non-uniqueness in law
of Leray solutions.
  Our proof utilizes the self-similarity and instability programme
  developed by Jia-\v{S}ver\'{a}k \cite{JS14,JS15}
  and  Albritton-Bru\'{e}-Colombo \cite{ABC21},
  together with the theory of martingale solutions
  including stability for non-metric spaces
  and gluing procedure.
\end{abstract}

\maketitle

{
\tableofcontents
}

\section{Introduction and main results} \label{Sec-Intro}

\subsection{Introduction}
In the  pioneering work \cite{Leray1934},
Leray constructed weak solutions
to deterministic Navier-Stokes equations (NSE for short),
which live in the energy class $L_t^{\infty} L_x^2 \cap L_t^2 \dot{H}_x^1$
and obey energy inequalities.
This class of solutions
is also called  {\it solution turbulent} in \cite{Leray1934},
and now is referred to as Leray-Hopf weak solutions
due to the important contributions by Hopf \cite{hopf1951}
in bounded domains.
Since then, a famous open problem is the uniqueness of Leray solutions to NSE.
This question is largely open in the stochastic setting.
Additive noise was invoked as the most natural candidate for the regularization of noise,
and has attracted significant interests in the literature of stochastic NSE.
As pointed out by Flandoli \cite{F10},
uniqueness of weak solutions to stochastic NSE in dimension $3$ remains an open problem.
The problem whether additive noise ``regularizes'' 3D NSE also remains open, see \cite{FL21}.

\bigskip
The aim of this paper is to
make progresses towards the uniqueness problem of Leray solutions
to forced stochastic NSE driven by additive noise.
More specifically, we are concerned with the stochastic Navier-Stokes equations
with an external stochastic force and additive noise,
\begin{align}\label{SNSE}
	\dif v-\Delta v\dif t+\div\lc v\otimes v \rc\dif t  + \na p  \dif t = f\dif t +\dif w,
\end{align}
under the incompressible condition
\begin{align*}
	\div\ v  = 0.
\end{align*}
Here
$v = v(t,x)$ and $p=p(t,x)$ are the velocity field
and the pressure term, respectively.
$f$ represents the (stochastic) external force,
and $w$ is a divengence-free $\mathcal{Q}$-Wiener process on a filtrated probability space
$(\Omega, (\mathcal{F}_t), \bbp)$, of the form
\begin{align} \label{w-BM}
	w(t) = \sum\limits_{j=1}^\infty \sqrt{\lambda_j} \beta_j(t) \epsilon_j, \ \ t\geq 0,
\end{align}
where $(\epsilon_j)$ is an orthonormal basis of $\calw:= H^N \cap L^2_\sigma$
consisting of eigenvectors of $\mathcal{Q}$ with
corresponding eigenvalues $(\lambda_j) \subseteq l^1$,
$N>5/2$, and  $L^2_\sig$ denotes the space of divergence-free $L^2$-integrable vector fields.

Building upon the recently discovered non-uniqueness of deterministic Leray solutions \cite{ABC21}, we establish non-uniqueness in law for the 3D forced stochastic NSE driven by additive noise.


\begin{theorem} [Non-uniqueness in law of Leray solutions] \label{Thm-Nonuniq}
	Consider the stochastic NSE \eqref{SNSE} with zero initial condition.
	Then, there exists an $(\mathcal{F}_t)$-adapted forcing
    $f \in L^2(\Omega; L^2_{\loc}$ ${H}^{-1}_x) \cap L^2(\Omega; L^1_{\loc}L^2_x)$,
	such that
	there exist distinct Leray solutions in the probabilistically weak sense.
\end{theorem}
In particular,
the joint uniqueness in law of Leray solutions fails.

Our proof strategy allows us to provide more precise information regarding non-uniqueness. Specifically, we demonstrate pathwise non-uniqueness in the probabilistic strong sense, up to the occurrence of a stopping time. This result is presented in Theorem \ref{Thm-Nonuniq-Local} below.

It also would be possible to weaken the noise regularity by using the stochastic convolution
as in \cite{HZZ19}, thanks to the regularization effect of heat semi-flows.

Furthermore, our approach extends to cover the hyper-viscous stochastic Navier-Stokes equations below Lions' exponent $\alpha<5/4$. More precisely, we consider the problem
\begin{align}
	\label{Hyper-SNSE}
	& \dif v+(-\Delta)^\al v\dif t+\div\lc v\otimes v \rc\dif t  + \na p  \dif t = f\dif t +\dif w,
\end{align}
under the incompressible condition $\div v =0$,
where  $\alpha\geq 1$
and $(-\Delta)^\al$ is the hyper-vicosity
defined via the Fourier transform
\begin{equation*}
	\mathcal{F}((-\Delta)^\al v)(\eta)=|\eta|^{2 \al} \mathcal{F}(v)(\eta), \quad \eta \in \R^3.
\end{equation*}

\begin{theorem} [Non-uniqueness in law: below the Lions exponent] \label{Thm-Nonuniq-HyperSNSE}
	Consider the hyper-viscous, forced stochastic NSE \eqref{Hyper-SNSE} with
	$1\leq \alpha<5/4$.
	Then,
	there exists an $(\mathcal{F}_t)$-adapted forcing term $f \in L^2(\Omega; L^2_{\loc} {H}^{-\alpha}_x) \cap L^2(\Omega; L^1_{\loc}L^2_x)$
	such that
	the joint uniqueness in law of Leray solutions
	fails.
\end{theorem}


\subsection{Optimality}  \label{Subsec-Sharp}

We briefly discuss the optimality of our non-uniqueness result from two different perspectives: in relation to the hyper-viscous parameter and the size of the stochastic force in appropriate critical spaces.

\subsubsection{Sharpness of Lions exponent}

In the high-viscosity regime above the Lions exponent,
well-posedness holds even in the probabilistically strong sense for Leray solutions.

\begin{theorem} [Probabilistically strong well-posedness: above the Lions exponent] \label{Thm-GWP-HyperSNSE}
	Consider the hyper-viscous, forced stochastic NSE \eqref{Hyper-SNSE} with
	$\alpha \geq 5/4$.
	Then, for any $T\in(0,\infty)$, $v_0\in L^2(\Omega; L^2_\sigma)$
	and for any $f\in L^2(\Omega; L^2_{\loc}H^{-\alpha}_x)$,
	there exists a unique solution $(v(t))_{t\in [0,T]}$ to \eqref{Hyper-SNSE},
	such that $v(0) = v_0$,
    and it satisfies the energy bound
	\begin{align*}
		\bbe \|v\|_{C([0,T]; L^2_\sigma)}^2
		+ \bbe \| v\|_{L^2(0,T; H^\alpha)}^2 < \infty,
	\end{align*}
    and the energy inequality for $t\in [0,T]$,
    \begin{align*}
    		\bbe \|v(t)\|_{L^2}^2
    		+ 2 \bbe \int_0^t \|(-\Delta)^{\frac \alpha 2} v(s)\|_{L^2}^2 \dif s
    		\leq \bbe \|v(0)\|_{L^2}^2 +
    		2 \bbe \int_0^t (f(s), v(s)) \dif s
    		+ \sum_{j=1}^{\infty}\lambda_j\|\epsilon_j\|^2_{L^2} t.
    \end{align*}
\end{theorem}

Since well-posedness in the probabilistically strong sense implies joint uniqueness in law,
Theorems \ref{Thm-Nonuniq-HyperSNSE} and \ref{Thm-GWP-HyperSNSE} together
lead to the sharpness of Lions exponent for
the uniqueness/non-uniqueness in law of Leray solutions.

\begin{corollary} [Sharpness of Lions exponent]  \label{Coro-Sharp-Lions}
	Consider the hyper-viscous, forced stochastic Navier-Stokes equations \eqref{Hyper-SNSE}.
	Then, the following holds:
	\begin{enumerate}
		\item[$(i)$] In the regime where $\alpha\in [1,5/4)$,
		the joint uniqueness in law of Leray solutions fails
		for some $L^2_\sigma$ initial condition
		and for some $(\mathcal{F}_t)$-adapted stochastic forcing term
		in $L(\Omega; L^2_{\loc} H^{-\alpha}_x)$.
		
		\item[$(ii)$] In the regime where $\alpha \in [5/4, \infty)$,
		the joint uniqueness in law of Leary solutions holds
		for any $L^2_\sigma$ initial data and
		for any $L^2(\Omega; L^2_{\loc} H^{-\alpha}_x)$-integrable stochastic force.
	\end{enumerate}
\end{corollary}

\subsubsection{The role of the external body force}

A close inspection of our proof strategy for the instability
of velocity oepratros (see Theorem \ref{Thm-Lvel-eign} below)
reveals that the external body force is significantly more regular t
han just $f \in L^2_{\loc}H^{-1}_x$, almost surely with respect to $\mathbb{P}$.
In fact, it belongs to more regular critical spaces, see Section \ref{sec:strategy of proof}.

To be more precise, we express the body force in similarity coordinates:
\begin{equation*}
	f(x,t) = \frac{1}{t^{3/2}} F(\xi,\tau)\, ,
	\quad \xi = \frac{x}{\sqrt{t}}\, , \tau = \log t \, .
\end{equation*}
It turns out that the self-similar profile $F$ satisfies
certain spatial estimates that are uniform with respect to the time variable $\tau\in (-\infty,1)$.
For instance, it is not difficult to verify that:
\begin{equation}\label{eq:scal inv bound}
	\bbe[\| F \|_{L^\infty((-\infty,1); L^2_\xi)}] < \infty\, .
\end{equation}
The crucial point of this section is that the forced stochastic NSE is well-posed with high probability
when we reduce the magnitude of the external body force in critical norms as \eqref{eq:scal inv bound}.

\bigskip

\noindent
We consider the following problem:

\begin{align}\label{SNSEdelta}
	\begin{cases}
		\dif v-\Delta v\dif t+\div\lc v\otimes v \rc\dif t  + \na p \dif t = \delta f\dif t +\dif w\, ,
		\\
		\div\, v = 0\, ,
		\\
		v(0) = 0\, ,
	\end{cases}	
\end{align}
where $\delta>0$ is a positive parameter.

\begin{theorem} (Well-posedness with high probability: small external force) \label{thm:well-posedness small force}
	Assume that
	\begin{equation*}
		\bbe[\| F \|_{L^\infty((-\infty,1); L^2_\xi)}] = C_1 < \infty\, .
	\end{equation*}
	For every $\epsilon>0$, if $\delta\le \delta(\epsilon, C_1)$ sufficiently small,
then there exists $E\subseteq \Omega$ and a deterministic time $T>0$ such that
	\begin{itemize}
		\item[(i)] $\bbp(\Omega \setminus E)\le \epsilon$.

		\item[(ii)] There exists a unique strong solution to \eqref{SNSEdelta} up to time $T$, for every $\omega \in E$.
	\end{itemize}
\end{theorem}
In particular, we conclude that Theorem \ref{Thm-Nonuniq} is sharp in the sense that if we replace $f$ with $\delta f$, for some sufficiently small $\delta>0$, the problem becomes strongly well-posed with high probability.

A completely analogous conclusion holds for the hyper-viscous Navier-Stokes equations, but for the sake of readability we consider only the case $\alpha=1$.

\subsection{Comparison with existing results}

The uniqueness question for Leray solutions in the deterministic setting is quite old, and significant progress has been made recently. One notable result is that Buckmaster and Vicol \cite{BV19} established the non-uniqueness of distributional solutions to the Navier-Stokes equations in the class $C_tL^2_x$, which can conform to any prescribed ``kinetic energy'' profiles.
The proof in \cite{BV19} is based on the convex integration method,
which was originally introduced by De Lellis and Sz\'{e}kelyhidi \cite{dls09, dls10} in the context of the 3D Euler equations. It has since been applied to various hydrodynamic models. See, e.g.,  \cite{bdsv19,I18} for the resolution of the famed Onsager conjecture, and \cite{cl20.2,cl21.2,lqzz21} for the sharp non-uniqueness near two endpoints of Lady\v{z}enskaja-Prodi-Serrin criteria.
We also would like to refer to \cite{LT20,BCV22} for
the sharpness of Lions exponent for the $C_tL^2_x$ well-posedness of NSE,
and \cite{CDD18,DR19}
for the ill-posedness of Leray solutions to hypo-dissipative NSE,
via the convex integration method.

\bigskip

Another distinct approach towards establishing non-uniqueness of Leray solutions was developed
by Jia, \v{S}ver\'{a}k, and Guillod \cite{JS14,JS15,GS15},
focusing on the investigation of self-similar solutions and spectral instability.
Very recent progress has been achieved by Albritton, Bru\'{e}, and Colombo \cite{ABC21},
where they prove the non-uniqueness of Leray solutions for the forced NSE.
This method has also been successfully applied to cases involving bounded domains \cite{ABC22} and the 2D forced hypo-viscous NSE \cite{AC22}.

The most important novelty in \cite{ABC21} revolves around the creation of a smooth, unstable vortex ring in three dimensions, see Appendix \ref{App-Vel-Profile}. This achievement relies on the unstable vortex previously developed by Vishik \cite{Vishik2018I,Vishik2018}, see also the lecture notes \cite{ABC+21}. An impressive application of Vishik's construction lies in its proof of non-uniqueness of the forced two-dimensional Euler equations in the context of $L^p$ vorticities, $p<\infty$.

\bigskip

Equations of fluid mechanics
driven by stochastic force have been studied extensively in literature. Since the work \cite{BT73} by Bensoussan-Temam in the early 70's, many aspects of stochastic NSE have been developed, including the existence of martingale solutions \cite{CP01,Brz13,FG95,Roz2005}, ergodicity \cite{HM06,KNS20} and stochastic representation \cite{CI08,HRZ23}.

	Many attempts have been made to investigate noise regularization effects on
	the uniqueness problem of stochastic NSE.
	Additive noise was invoked as the most natural candidate for the regularization of noise.
	One main reason is that it succeeds in improving well-posedness
	of finite-dimensional stochastic differential equations with irregular drifts
	\cite{D07,KR05,HRZ23,V81}.
	This phenomena also exhibit for certain infinite-dimensional SPDEs
	\cite{BM19,DaF10,DaFPR13,DaFRV16}.
	Thus, stochastic NSE driven by additive noise has attracted significant interest
	in literature.
	Among others,
	we would like to mention \cite{FR02} for the development of
	stochastic partial regularity theory of Caffarelli-Kohn-Nirenberg,
	\cite{DaD03} for infinite dimensional Kolmogorov equations,
     and  \cite{FR08} for Markov selections with the strong Feller property.
	See also the very recent result \cite{GL23}.

    Recently, convex integration method also applies to stochastic settings.
    We refer to \cite{HZZ19} for the non-uniqueness in law of
    $C_tL^2_x$ weak solutions to stochastic NSE on the torus.
    See \cite{HZZ23.1,HZZ23.2} for the non-uniqueness of
     probabilistically strong and analytically weak solutions.
    See also \cite{MS23,Y22,Y23} for   various other equations.

\bigskip
While we were finalizing the current paper, Hofmanov\'a-Zhu-Zhu uploaded a work on arXiv \cite{HZZ23}
that explores a similar problem.
We now briefly discuss the similarities and distinctions between the two contributions.

In \cite{HZZ23} the authors
consider the stochastic NSE
with a multiplicative one dimensional noise that does not depend on the space variable. More precisely, the equations are
\begin{equation*}
	\dif v-\Delta v\dif t+\div\lc v\otimes v \rc\dif t  + \na p  \dif t = f\dif t + v \dif W \, ,
\end{equation*}
where $W$ is the standard one-dimensional Brownian motion.

In our context, the noise is  additive, depending on space and time variables. It excites all Fourier modes, thus exhibiting a strong spatial dependence.

An close similarity between our work and the study by \cite{HZZ23} lies in the proof strategy. In both cases, the central tool is a local-in-time strong non-uniqueness stemming from the unstable vortex ring constructed in \cite{ABC21}.
Additionally, in both works a gluing step to establish global non-unique solutions has been employed. However, the reduction to the deterministic setting is executed in distinct manners due to the different nature of the noise. Consequently, this leads to forcing terms $f$,  with differing structures.
See \eqref{f-force-def} below for our choice of stochastic force in the current additive case.

Regarding the extension to global martingale solutions,
the defect of the compactness of domain causes several technical problems
and requires new topologies different from the torus case.
To solve these issues, in the present work we introduce
a non-metric space
and prove new  stability result,
which is the key ingredient to implement the gluing procedure
of martingale solutions.
It turns out that the stochastic force,
with $L^2(\Omega; L^2_{\loc} H^{-\alpha}_x) \cap L^2(\Omega; L^1_{\loc} L^2_x)$ integrability,
cooperates  well with the martingale theory.
The proof is quite delicate and we give the detailed arguments
in this work.

\section{Strategy of proof}
\label{sec:strategy of proof}


Our proof proceeds in three main steps: non-uniqueness of local-in-time Leray solutions
in the probabilistically strong sense, existence of global-in-time martingale solutions
for any $L^2_\sigma$ initial data, gluing
of martingale solutions.
The first step follows the ideas of \cite{ABC21}.

\medskip

{\bf Step 1. Local pathwise nonuniqueness.}
In this step,
we fix the given filtrated  probability space
$(\Omega, (\mathcal{F}_t), \mathbb{P})$ and
the $\mathcal{Q}$-Wiener process $w$.
We aim to construct
an appropriate stochastic external force
and two distinct local Leray solutions to \eqref{Hyper-SNSE}.
The solutions here are considered
in the probabilistically strong
and analytically weak sense defined below.

\begin{definition} [Probabilistically strong solutions up to stopping times] \label{def-Prob-strong}
	Given a filtrated probability space
	$(\Omega, (\mathcal{F}_t), \bbp)$
	and a $\mathcal{Q}$-Wiener process $w$.
	For any divergence-free initial datum $v_0 \in L^2(\Omega; L^2)$,
	$v$ is called a Leray solution in the probabilistically strong sense
	to \eqref{Hyper-SNSE}
	up to a stopping time $\tau(>0)$,
	if $v$ is an $(\mathcal{F}_t)_{t \geq 0}$-adapted,
	$L^2$-valued continuous process and satisfies the following:
	\begin{itemize}
		\item For all $t \in[0, \tau]$, $v(t, \cdot)$ is divergence free in the sense of distributions.
		\item Equation \eqref{Hyper-SNSE} holds in the sense of distributions, i.e.,
		$\P$-a.s. for any divergence-free test functions $\varphi \in C_0^{\infty}(\R^3)$
		\begin{align}
			(v(t), \vf)
			= \int_0^t (-(-\Delta)^\alpha v - \div\lc v\otimes v \rc, \vf) \dif s
			+ \int_0^t (f(s), \vf) \dif s + (w(t), \vf).
		\end{align}
		\item Energy inequality:
		for any $t\geq 0$,
		\begin{align*}
				& \bbe \|v(t\wedge \tau)\|_{L^2}^2
				+ 2 \bbe \int_0^{t\wedge \tau} \|(-\Delta)^{\frac \alpha 2} v(s)\|_{L^2}^2 \dif s \\
				\leq& \bbe \|v_0\|_{L^2}^2
				+  2 \bbe \int_0^{t\wedge \tau} (f(s), v(s)) \dif s
                + \sum_{j=1}^{\infty}\lambda_j\|\epsilon_j\|^2_{L^2} \bbe (\tau \wedge t).
		\end{align*}
		
		Moreover, the energy bound holds:
		\begin{align*}
			\bbe^{\bbp} (\sup_{0\leq s\leq t\wedge \tau}\|v(s)\|_{L^2}^2
			+ \int_{0}^{t\wedge \tau} \| v(s)\|_{H^\alpha}^2 \dif s)
			\leq C_t (1+ \|v_0\|_{L^2}^2),
		\end{align*}
		where $t\mapsto C_t$ is a positive increasing constant.
	\end{itemize}
\end{definition}

\paragraph{
	{\bf Self-similarity reformulation of stochastic NSE} }
The deterministic hyper-vscous NSE posses a scaling symmetry
\begin{align*}
	v(t,x) \mapsto \lambda^{2\al-1} v (\lambda^{2\al}t, \lambda x),\ \
	p \mapsto \lambda^{4\al-2} p (\lambda^{2\al}t, \lambda x), \ \
	f(t,x) \mapsto \lambda^{4\al-1} f(\lambda^{2\al} t, \lambda x) \, ,
\end{align*}
which allows to identify subcritical, critical and supercritical spaces.
A general philosophy is that non-linear equations are well-posed in the subcritical spaces,
while solutions may exhibit ill-posedness in the supercritical spaces. Comprehensive discussions can be found in the works \cite{K00,K17} by Klainerman.

A particular class of solutions living on the borderline of well-posedness theory is known as {\it self-similar solutions}. Recently, they have played a central role in the study of the uniqueness problem for Leray solutions \cite{ABC21,JS14,JS15}.

\medskip
\noindent

Following \cite{ABC21}, we aim to construct local-in-time non-unique solutions for the stochastic NSE
starting from an {\it unstable} self-similar solution. To achieve this, we begin by removing the noise component by seeking solutions in the form of
\begin{align}  \label{v-z-u}
	v = u + w,
\end{align}
where $u$ solves
\begin{align}\label{rand-eq}
	\begin{cases}
		\pt u+(-\Delta)^\al
		u +\div\lc(u+w)\otimes(u+w)\rc +(-\Delta)^\al w +\na p  =f
		\\
		\div\ u  =0\, .
	\end{cases}
\end{align}
We then consider {\it similarity variables}
\begin{align*}
	\xi= {t^{-\frac{1}{2\al}}}{x}, \quad \tau=\log t.
\end{align*}
and define self-similar profiles:
\begin{align} \label{sim-va}
	u(x, t)= {t^{\frac{1}{2\al}-1}} U(\xi, \tau),
	\quad w(x, t)={t^{\frac{1}{2\al}-1}} W(\xi, \tau)\, ,
\end{align}
\begin{align}  \label{sim-va-pf}
	p(x,t) =t^{\frac{1}{\alpha}-2} P(\xi, \tau),
	\quad f(x, t)={t^{\frac{1}{2\al}-2}} F(\xi, \tau)\, .
\end{align}
Here, as in \cite{ABC21,JS15}, we use the convention that lower (resp. upper) case functions denote functions in physical variables
(resp. similarity variables).

We have the self-similar reformulation of equation \eqref{rand-eq}:
\begin{align}  \label{SNSE-Similarity}
	& \ptau U-(1-\frac{1}{2\al})U-\frac{\xi}{2\al} \cdot \na U +(-\Delta)^\al U +U\cdot\na U \notag \\
	&\ \ +W\cdot\na U +U\cdot\na W+W\cdot\na W+(-\Delta)^\al W+\na P  = F
\end{align}
coupled with the incompressibility condition $\div U =0$, for $\xi\in \mathbb{R}^3$ and $\tau\in (-\infty,T)$.

\paragraph{{\bf Vortex ring and Linearized operator.} }
Let $\bar U \in C^\infty_c(\mathbb{R}^3)$ be the unstable vortex ring built in \cite{ABC21},
see Appendix \ref{App-Vel-Profile} below for more explanations on its structure.
We think of $\bar U$ as a stationary solution to \eqref{SNSE-Similarity} with stochastic body force
\begin{align}\label{F-def}
	F:= &-(1-\frac{1}{2\al})\bar U-\frac{\xi}{2\al} \cdot \na \bar U +(-\Delta)^\al \bar U + \bar U\cdot\na \bar U \notag \\
	&\ \ +W\cdot\na \bar U + \bar U\cdot\na W+W\cdot\na W+(-\Delta)^\al W+\na P\, .
\end{align}
The corresponding force in physical variables is
\begin{align}  \label{f-force-def}
	f(x, w(t)) :=& t^{\frac{1}{2\al}-2} F(\xi, \tau)   \\
	=& \pt\bar{u}+(-\Delta)^\al \bar{u}+ \bar{u}\cdot\na \bar{u}+
	\div\lc\bar{u}\otimes w+w\otimes\bar{u}\rc +
	\div(w\otimes w)+(-\Delta)^\al w\, , \notag
\end{align}
and consists of a deterministic and a random part. The latter depends on the Wiener process $w$, hence $f$ is $(\mathcal{F}_t)$-adapted.

It is clear from \eqref{SNSE-Similarity} and \eqref{F-def} that
\begin{align*}
	U_1 := \bar{U}
\end{align*}
solves \eqref{SNSE-Similarity}
with the stochastic forcing term $F$.
However,
\begin{equation*}
	\bar u_1(x,t) = \frac{1}{\sqrt{t}} U_1(\xi) \, ,
\end{equation*}
is a Leray solution to NSE only when $\alpha<5/4$, since
\begin{equation*}
	\norm{\bar{u}(\cdot,t)}_{L^2_x}=t^{\frac{5}{4\al}-1}\norm{\bar{U}(\cdot)}_{L^2_\xi}\stackrel{t\rightarrow 0}{\longrightarrow} 0 \quad\text{for}\; \al<\frac{5}{4} \, .
\end{equation*}

We proceed by linearizing the equations \eqref{SNSE-Similarity} with respect to the steady state $\bar U$, resulting in the linear operator:
\begin{equation*}
	\begin{split}
		-\boldsymbol{L}U = (\frac{1}{2\al}-1-\frac{\xi}{2\al} \cdot \nabla) &  U+(-\Delta)^\al U+ P_H(\bar{U} \cdot \nabla U+U \cdot \nabla \bar{U})
		\\
		&+ W \cdot \nabla U + U \cdot \nabla W\, .
	\end{split}
\end{equation*}
However, the quick time-decay as $\tau \to -\infty$
\begin{equation}\label{Z-est}
	\norm{W(\cdot,\tau)}_{\dot{H}_\xi^N}\les e^{(1-\frac{5}{4\al}+\kappa+\frac{N}{2\alpha})\tau}
	\norm{w}_{C^{\kappa}([0,t];\dot{H}_x^N)}\, ,
\end{equation}
allows to treat the linear terms depending on $W$ perturbatively.
\begin{remark}
	The $C^{\frac 14 +}$-H\"older continuity of the Weiner noise is crucial in the construction of local Leray solutions. It compensates the contribution coming form the linear stochastic part in \eqref{up-equation}, which is of big size $\sim e^{\tau (1-\frac{5}{4\alpha})}$.
\end{remark}
Consequently, we can concentrate our attention solely on the operator.
\begin{equation*}
	\begin{aligned}
		-\lssa U &:= (\frac{1}{2\al}-1-\frac{\xi}{2\al} \cdot \nabla) U+(-\Delta)^\al U+ P_H(\bar{U} \cdot \nabla U+U \cdot \nabla \bar{U})
		\\
		\lssa &: \cald_{vel} \subseteq L_\sig^2 \rightarrow L_\sig^2,
	\end{aligned}
\end{equation*}
where
\begin{equation*}
	\cald_{vel}:= \{U \in L_\sig^2: |||U||| := \|U\|_{H^{2\al}}+ \|\xi \cdot \nabla_{\xi} U\|_{L^2} <\infty \},
\end{equation*}
and $P_H$ is the Helmholtz-Leray projection.

The highly non-trivial aspect is that $\lssa$ possesses a maximal unstable eigenvalue $\lambda'$, as demonstrated in \cite{ABC21} for the case $\alpha=1$. Extending this result to the case where $\alpha\neq 1$ is not challenging since, as noted in \cite{ABC21}, the term $P_H(\bar{U} \cdot \nabla U+U \cdot \nabla \bar{U})$ is primarily responsible for spectral instability. Therefore, the (hyper-viscous) component can be treated perturbatively. See also the very recent work \cite{KMS23}.

\medskip

The linear operator $\lss$ is mainly studied in Subsections \ref{Subsec-Vor-Oper} and \ref{Subsec-Vel-Oper} by means of the theory of semigroups and operators, based on the Hille-Yosida theorem and stability theory of semigroups \cite{En2000,Kato80}. This approach to the problem is novel and partially deviates from the analysis in \cite{KMS23}.

%
%
%

\paragraph{{\bf The second solution.}}

In this section, we outline the specific ansatz for constructing the second solution to the NSE with the body force $f$ as defined in \eqref{f-force-def}.

Let $\eta \in \cald_{vel}$ be the eigenfunction associated to the maximal unstable eigenvalue $\lambda'$ of $\lssa$,
i.e. $\lssa \eta = \lambda' \eta$,
and $a:= {\rm Re} \lambda'>0$ is maximal among all the eigenvalues of $\lssa$. We set
\begin{equation*}
	\ul(\tau): = e^{\tau \lssa} \eta =\operatorname{Re} (e^{\tau\lambda'} \eta) \, ,
\end{equation*}
Note that
$\ul$ satisfies the equation
\begin{align}  \label{equa-Ulin}
	\partial_\tau \ul = \lssa \ul.
\end{align}
Then, the second solution to NSE will be given by
\begin{equation*}
	U_2 := \bar{U} + \ul + \up \, ,
\end{equation*}
where the remainder $\up$ needs to solve the equation
\begin{align}\label{up-equation}
	& \ptau \up- \lssa \up \notag  \\
	=&-
	P_H [(\up \cdot \nabla) \up+(\ul \cdot \nabla) \up
	+(\up \cdot \nabla) \ul+(\ul \cdot \nabla) \ul] \notag \\
	&- P_H [(\ul\cdot\na)W+(W\cdot\na)\ul+(\up\cdot\na)W+(W\cdot\na)\up]
\end{align}
and decays with rate $o(e^{\tau {\rm Re}\lambda'})$ as $\tau \to -\infty$.
Since $U^{lin}$ and $U^{per}$ have different decay rates at $-\infty$, it follows that $U_1$ and $U_2$ are distinct solutions to
to \eqref{SNSE-Similarity}.

\medskip
\noindent

To solve \eqref{up-equation} we adopt the strategy presented in \cite{ABC21}.
The key ingredients are the parabolic estimates for the semigroup $e^{\lssa}$
(see Proposition \ref{Prop-Semigroup} below),
along with the H\"older continuity in time of the noise. The latter is essential for handling linear terms dependent on $W$. However, the presence of these terms prevents us from constructing global solutions. Consequently, we infer local pathwise non-uniqueness for stochastic Leray solutions.

\begin{theorem} [Local pathwise non-uniqueness]   \label{Thm-Nonuniq-Local}
	Fix any integer $N>\frac 52$,
	$\kappa \in (\frac 14, \frac 12)$,
	and $\ve_0>0$ small enough such that
	$1+\kappa - \frac{5}{4\alpha} - \ve_0>0$.
	Let
\begin{equation}\label{sigma-stop-def}
		\sigma:= (20C(N,\kappa))^{\frac{-1}{a-\ve_0}}\wedge (20C(N,\kappa))^{\frac{-1}{1+\kappa-\frac{5}{4\al}-\ve_0}} \wedge \inf\{t>0:\norm{w}_{C^{\kappa}([0,t]; H_{x}^{N})} \geq 1 \},
	\end{equation}
	where $a(>1)$ is the spectral bound of $\lssa$,
	$C(N,\kappa)$ is a large deterministic constant given by \eqref{TU-esti} and \eqref{TU1-TU2-esti} below,
	depending on $N$ and $\kappa$.
	Then, there exist two distinct Leray solutions
	\begin{align}  \label{v1-def}
		v_1:= \bar{u} + w,
	\end{align}
	and
	\begin{align}   \label{v2-def}
		v_2:= \bar{u} + u^{\operatorname{lin}} + u^{\operatorname{per}} + w,
	\end{align}
	to equation \eqref{Hyper-SNSE}
	with the stochastic force $f$ given by \eqref{f-force-def},
	up to the stopping time $\sigma$
	in the probabilistically strong sense of Definition \ref{def-Prob-strong}.
\end{theorem}

\medskip

{\bf Step 2. Global existence of martingale solutions.}

So far, we have built pathwise non-uniqueness up to a stopping time.
To make it global, we need to prove a global existence result for martingale solutions to
the hyper-viscous stochastic NSE \eqref{Hyper-SNSE}.
This is the main goal of this section.

Differently from the previous stage,
the proof here relies crucially on
tightness criteria and stability
for non-metric spaces.

Because of the defect of the compact Sobolev embedding $H^\al \hookrightarrow L^2$,
the usual energy space $C_tL^2_x \cap L^2_t \dot{H}^1_x$
used in the bounded domain case
is no longer applicable
when constructing martingale solutions in the whole space.
Inspired by \cite{Roz2005} and \cite{Brz13},
we use the space
\begin{align*}
		\wh{\calz}_{[0,T]} := C([0, T]; \calu') \cap L^2(0,T; L^2_{\loc}) \cap C([0,T]; L^2_{w}),
\end{align*}
where $\calu'$ is the dual space of a Hilbert space $\calu$,
which is dense in $\calw$
and the embedding $\calu\hookrightarrow \calw$ is compact.
The existence of space $\calu$ follows from Lemma C.1 in \cite{Brz13}.
For the precise definitions and (weak) topologies for the above three spaces, see Subsection \ref{Subsec-Mart-Setup}.

Another delicate problem,
due to the stochastic forcing term here,
is that it is no longer possible to formulate martingale solutions
to \eqref{Hyper-SNSE}
in the usual way \cite{Roz2005,Brz13} by  requiring that
\begin{align}  \label{Mt-def}
	M_t:= v(t) + \int_0^t  P_H[ (-\Delta)^\al v +\div\lc v\otimes v \rc -f] \dif s
\end{align}
is a continuous  martingale with the quadratic variation $\sum_{j=1}^{\infty}\lambda_j\|\epsilon_j\|^2_{L^2} t$
(hence $(M_t)$ is a $\mathcal{Q}$-Wiener process by the martingale representation theorem).
The reason is that the stochastic force $f$ given by \eqref{f-force-def}
itself depends on
the  Wiener process $(M_t)$.

To fix this problem,
we formulate martingale solutions to \eqref{Hyper-SNSE} in the product canonical space
$\wh{\Omega}$,
including the trajectories of both solutions and Wiener noise,
that is
\begin{align*}
	\wh{\Omega}:= \scrx_{[0,\infty), \loc} \times \scry_{[0,\infty), \loc},
\end{align*}
where
	$\scrx_{[0,\infty),\loc}$ is the global version of the following separable space
	with the locally uniform topology
	\begin{align*}
		\scrx_{[0,T]} := C([0, T]; \calu') \cap L^2(0,T; \calh_{\loc}),
\end{align*}
and $\scry_{[0,\infty),\loc}$ is the trajectory space of Wiener process
\begin{align*}
	\scry_{[0,\infty),\loc}:=C_{\loc}([0,\infty);\calw)
\end{align*}
with the locally uniform topology.

Let $\mathscr{P}(\wh \Omega)$ denote the set of all probability measures
on $(\wh \Omega, \mathcal{B}(\wh \Omega))$
with $\mathcal{B}(\wh \Omega)$ being the Borel $\sig$-algebra coming
from the topology of locally uniform convergence on $\wh \Omega$.

\begin{definition} [Martingale solutions] \label{def-mart-1}
	Let $s \geq 0$, $N> \frac 52$, and $z_0:=(x_0, y_0)\in  L^2_\sigma \times \calw $.
	A probability measure $P_{s_0, z_0} \in \mathscr{P} (\wh \Omega)$
	is said to be a martingale solution to \eqref{Hyper-SNSE}
	with the initial value $z_0$ at time $s_0$ provided
	\begin{enumerate}
		\item[$(M1)$] $P_{s_0, z_0} (z(t)=z_0, 0 \leq t \leq s_0 )=1$ and $P_{s_0, z_0}$-a.s.
		for any $t\geq s_0$,
		\begin{equation*}
			x(t) - x_0 = \int_0^t
			P_H [  - (-\Delta)^\alpha x(s) - \div (x(s)\otimes x(s))   + f(s, y(s)) ]\dif s + y(t) - y_0,\ \ in\ \mathscr{D}'.
		\end{equation*}
		
		\item[$(M2)$]
		Energy inequality:
		for any $t\geq s_0$,
		\begin{align*}
				& \mathbb{E}^{P_{s_0, z_0}} \|v(t)\|_{L^2}^2
				+ 2 \mathbb{E}^{P_{s_0, z_0}} \int_{s_0}^{t} \|(-\Delta)^{\frac \alpha 2} v(s)\|_{L^2}^2 \dif s \\
				\leq&   \|x_0\|_{L^2}^2
				+  2 \mathbb{E}^{P_{s_0, z_0}} \int_{s_0}^{t} (f(s, y(s)), v(s)) \dif s + \sum_{j=1}^{\infty}\lambda_j\|\epsilon_j\|^2_{L^2} (t-s_0).
		\end{align*}
		
		Moreover, the energy bound holds:
		\begin{align*}
			\mathbb{E}^{P_{s_0, z_0}}  (\sup_{s_0\leq s\leq t}\|v(s)\|_{L^2}^2
			+ \int_{s_0}^{t} \| v(s)\|_{H^\alpha}^2 \dif s)
			\leq C_t (\|x_0\|_{L^2}^2 +1),
		\end{align*}
		where $t\mapsto C_t$ is a positive increasing constant.
		
		\item[$(M3)$] $(y(t))$ is a $\mathcal{Q}$-Wiener process starting from $y_0$ at time $s_0$ under $P_{s_0, z_0}$.
	\end{enumerate}
	Let $\mathscr{M}(s_0, z_0, C_{\cdot})$ denote the set of all martingale solutions satisfying $(M1)$-$(M3)$.
\end{definition}

The point of Definition \ref{def-mart-1} is to view martingale solutions
as probability measures on canonical spaces.
Hence,
gluing two martingale solutions reduces to gluing two probability measures on canonical spaces,
as in the finite dimensional case in \cite{SV79}.
It is worth noting that
the above formulation of martingale solutions
has similarity with the multiplicative noise case in \cite{HZZ19,HZZ23},
while differs from the additive noise case in \cite{HZZ19}.
In the infinite dimensional case,
one meets the right-continuous filtration problem, rather than the natural filtration in \cite{SV79},
which has been solved in \cite{HZZ19}.
We also refer to \cite{HZZ19} for the nice
treatment of stochastic convolutions on canonical spaces.

In this work, delicate analysis has been dedicated to the
corresponding tightness criterion
and stability results on the non-metric space  $\wh{\mathcal{Z}}_{[0,T]}$, which are important to glue martingale solutions.
The resulting tightness processes are treated  by
means of Jakubowski's version of Skorokhod representation theorem
for non-metric spaces (cf. \cite{Brz13}),
rather than the martingale representation theorem,
due to the aforementioned problem hidden in \eqref{Mt-def}.

The existence  of global-in-time martingale solutions
is formulated below.

\begin{proposition} [Global existence of martingale solutions] \label{Prop-GWP-Mart}
	For any $s_0\geq 0$ and $z_0:=(x_0, y_0)\in L^2_\sigma \times H^N$
	with $N>\frac 52$,
	we have
	\begin{align*}
		\mathscr{M}(s_0, z_0, C_{\cdot}) \not = \emptyset,
	\end{align*}
	that is,
	there exists at least one martingale solution $P_{s_0, z_0} \in \mathscr{M}(s_0, z_0, C_{\cdot})$ to \eqref{Hyper-SNSE}
	starting from $z_0$ at time $s_0$,
	where $t\mapsto C_t$ is a positive increasing function.
\end{proposition}

\medskip

{\bf Step 3. Gluing procedure.}
We are now in position to glue together the  local-in-time
strong solutions
and global-in-time martingale solutions
in the previous two steps.

The gluing procedure relies on the
Stoock-Varadhan theory \cite{SV79}.
The key ingredient is the stability of martingale solutions
with respect to initial data.
Again, due to the defect of the compactness in the whole space,
the proof of stability result requires more delicate
tightness criteria for
the solution space $\wh{\mathcal{Z}}_{[0,T]}$.
The gluing result is contained in Proposition \ref{Prop-Glue} below.

Note that,
the gluing procedure is performed in the product canonical spaces
for both the stochastic solutions and Wiener processes.

Finally,
we can glue together the local solutions in Theorem \ref{Thm-Nonuniq-Local}
and global martingale solutions in Proposition \ref{Prop-GWP-Mart},
to obtain distinct Leray solutions in the probabilistically weak sense,
which consequently yields the desirable non-uniqueness result in law.

\section{Instability of vorticity and velocity operators}

This section is devoted to the key instability of vorticity and velocity operators,
which plays the crucial role in the construction of non-unique local pathwise Leray solutions.

Throughout this section,
operators are denoted by boldface letters.
To ease notations,
the superscript of $\al$ is omitted.

\subsection{Vorticity operators}   \label{Subsec-Vor-Oper}


Let $\bar{U}$ denote
the axisymmetric-no-swirl vortex-ring built in \cite{ABC21}, see Appendix \ref{App-Vel-Profile}. The corresponding vorticity is given by
$\bar{\Omega} : =\operatorname{curl} \bar{U}$.

We employ standard cylindrical coordinates $(r\cos \theta, r\sin \theta, z)\in \R^3$, and denote as
$L_{\mathrm{aps}}^2$ the space of $L^2-$integrable
``axisymmetric pure-swirl'', i.e. vector fields of the form $\Omega=\Omega^\theta(r, z) e_\theta$.

Given $\Omega \in L_{\mathrm{aps}}^2$, the associated velocity field $U :=\bs(\Omega)$
is an ``axisymmetric without swirl'' vector field, i.e. it is expressed as $U^r(r,z) e_r + U^z(r,z) e_z$.
Recall that $U =\bs(\Omega)$ denotes the $3D$ Biot-Savart law:
$$
U(x)=-\frac{1}{4 \pi} \int_{\mathbb{R}^3}\left(\nabla_y \frac{1}{|x-y|}\right) \times \Omega(y) \dif y\, .
$$

We define the vorticity operators
$\boldsymbol{L}_{\mathrm{vor}}^{(\beta)}$,
$\beta>0$, by
\begin{equation*}
    \begin{aligned}
    \boldsymbol{L}_{\mathrm{vor}}^{(\beta)}:
     \ &\cald_{vor} \subseteq L_{\mathrm{aps}}^2 \rightarrow L_{\mathrm{aps}}^2, \\
    -\boldsymbol{L}_{\mathrm{vor}}^{(\beta)} \Omega
    &:=(-1-\frac{\xi}{2\al} \cdot \nabla) \Omega
     +(-\Delta)^\al \Omega+\beta[\bar{U}, \Omega]+\beta[U, \bar{\Omega}]
    \end{aligned}
\end{equation*}
with $U :=\bs(\Omega)$,
$[\bar{U}, \Omega] = \bar U \cdot \na \Omega - \Omega \cdot \na \bar U$,
$[{U}, \bar \Omega] =  U \cdot \na \bar \Omega - \bar \Omega \cdot \na U$
 and the domain
\begin{equation}  \label{Dvor-def}
    \cald_{vor} :=\{\Omega \in L_{\mathrm{aps}}^2:  |||\Omega |||:= \|\Omega\|_{H^{2\al}} + \|\xi \cdot \nabla_{\xi} \Omega\|_{L^2} <\infty \}.
\end{equation}
As  domain is essential for unbounded operators,
we will also denote the vorticity operator by a pair
$(\boldsymbol{L}_{\mathrm{vor}}^{(\beta)}, \cald_{vor})$.

In order to study the instability of the vorticity operators,
as in \cite{ABC21},
we rewrite
\begin{align*}
    \boldsymbol{L}_{\mathrm{vor}}^{(\beta)} = \beta \T_\beta +1- \frac{3}{4\alpha} ,
\end{align*}
where  the operator
\begin{equation}   \label{Tbeta-def}
    \T_\beta  =  \frac{1}{\beta} (\boldsymbol{L}_{\mathrm{vor}}^{(\beta)} + \frac{3}{4\alpha} -1)
              =  \frac{1}{\beta} \D+\M+\boldsymbol{S}+\K
\end{equation}
with
\begin{align*}
-\D \Omega &:= (-\frac{3}{4\al} -\frac{\xi}{2\al} \cdot \nabla) \Omega+(-\Delta)^\al \Omega, \\
-\M \Omega &:=\bar{U} \cdot \nabla \Omega, \\
-\boldsymbol{S} \Omega &:=-\bar{\Omega} \cdot \nabla U-\Omega \cdot \nabla \bar{U}, \\
-\K \Omega &:= U \cdot \nabla  \bar{\Omega},
\end{align*}
where $U :=\bs(\Omega)$.
The formal limit $\T_\infty$,
as $\beta\rightarrow\infty$,
is
\begin{equation}  \label{Tinfty-def}
    \T_\infty =  \M+\boldsymbol{S}+\K,
\end{equation}
with the domain
\begin{align*}
      D(\T_{\infty})=\left\{\Omega \in L_{\mathrm{aps}}^2: \bar{U} \cdot \nabla \Omega \in L^2 \right\}.
\end{align*}

As we shall see below that $\D$ is a generator of contraction semigroup,
$\M$ is skew-adjoint and compact relative to $\D$.
By \cite{ABC21},
$\boldsymbol{S}$ is a bounded operator,
and the operator norm is bounded by
$$\mu:=\|\boldsymbol{S}\|_{L_{\mathrm{aps}}^2 \rightarrow L^2}\les l^{-\frac{2}{3}},$$
where $l$ is a parameter in the construction of $\bar{U}$ (see Appendix \ref{App-Vel-Profile}).

We now check that $\K$ is $L_{\mathrm{aps}}^2\to L^2$ compact.
Take any sequence $\{\Omega_n\} \subseteq L_{\mathrm{aps}}^2$ with
\begin{equation*}
    \sup\limits_{n\in \bbn}
    \|\Omega_n\|_{L^2} \leq  C <\infty.
\end{equation*}
By the Biot-Savart law,
the corresponding velocity sequence $\{U_n\}$
is uniformly bounded in $\dot{H}^{1}\cap L^6$,
and therefore are uniformly bounded in $H^1_{\supp{\bar{\Omega}}}$.
The compact embedding $H^1_{\supp{\bar{\Omega}}} \hookrightarrow L^2_{\supp\bar{\Omega}}$
then yields that
there exists a subsequence $\{U_{n_k}\}$ being a Cauchy sequence in $L^2_{\supp\bar{\Omega}}$.
This ensures that $(\K \Omega_{n_k})$
is a Cauchy sequence in $L^2$ by noticing the smoothness and compactness of $\bar{\Omega}$,
thereby proving the compactness of $\K$.

\bigskip
By \cite[Proposition 2.6]{ABC21}, $\T_\infty$ has a unstable eigenvalue $\lambda_\infty$
with positive real part.
In particular, for $l$ large enough,
one has
\begin{equation}   \label{Relbb-mu}
    \operatorname{Re}\lambda_\infty > \mu.
\end{equation}
Here and after, we fix $l$ and $\bar{U}$ such that \eqref{Relbb-mu} holds.

\begin{lemma}  [Generator characterization of $\D$] \label{Lem-D-Oper}
For the operator $(\D, \cald_{vor})$ we have
\begin{enumerate}
  \item[$(i)$] The Schwartz class $\cals$ is dense in $\cald_{vor}$
  with respect to the norm $|||\cdot |||$ in \eqref{Dvor-def}.

  \item[$(ii)$] Resolvent set:
   \begin{align*}
     \rho (\D) \supseteq \{\lambda \in \mathbb{C}: {\rm Re} \lambda>0\}.
   \end{align*}

  \item[$(iii)$] $\D$ is dissipative, i.e.,
     \begin{align}  \label{D-dissip}
        \|(\lambda-\D)\Omega\|_{L^2} \geq \lambda \|\Omega \|_{L^2}, \ \ \forall \lambda >0.
     \end{align}

\end{enumerate}
In particular, the operator
$(\D, \cald_{vor})$
is densely defined, closed
and generates a strongly continuous contraction semigroup on $L^2_{aps}$.
\end{lemma}

\begin{proof}
Let $\lambda \in \mathbb{C}$ with ${\rm Re} \lambda >0$.
We first prove that for any $Z\in L^2_{aps}$,
\begin{equation}\label{equation}
   (\lambda-\D) \Omega = (-\Delta)^\alpha \Omega - \frac{\xi}{2\al} \cdot \nabla_{\xi} \Omega
     +(\lambda-\frac{3}{4\alpha}) \Omega=Z
\end{equation}
is uniquely solvable in  $\cald_{vor}$
and
\begin{equation}\label{parabolic_regu1}
   |||\Omega||| := \|\Omega\|_{H^{2\al}}+\left\|\xi \cdot \nabla_{\xi} \Omega\right\|_{L^2} \lesssim\|Z\|_{L^2}.
\end{equation}

To this end, we argue as in \cite[Lemma 2.1]{JS15}.
Let
\begin{align*}
   h(x, t) := t^{\lambda-\frac{3}{4\alpha}} \Omega(t^{-\frac{1}{2\al}} x), \ \
   g(x, t) :=  t^{\lambda -\frac{3}{4\alpha} -1} Z(t^{-\frac{1}{2\al}} x).
\end{align*}
Note that
\begin{align*}
   \partial_t h(x,t)
   = (-\frac{3}{4\alpha} + \lambda) t^{ \lambda -\frac{3}{4\alpha} -1} \Omega(t^{-\frac{1}{2\alpha}} x)
     + (-\frac{1}{2\alpha}) t^{\lambda -\frac{3}{4\alpha} -1} (t^{-\frac{1}{2\alpha}} x)
       (\na \Omega) (t^{-\frac{1}{2\alpha}} x).
\end{align*}
Moreover, by the Fourier transform,
\begin{align*}
   [(-\Delta)^\alpha h(\cdot, t)]^{\wedge}(\eta)
   = |\eta|^{2 \alpha} t^{\lambda+ \frac{3}{4\alpha}} \wh{\Omega}(t^{\frac{1}{2\alpha}} \eta)
   = t^{ \lambda  - \frac{3}{4\alpha} -1} [((-\Delta)^{\alpha} \Omega)(t^{-\frac{1}{2\alpha}} \cdot)]^{\wedge}(\eta),
\end{align*}
which yields that
\begin{align*}
   (-\Delta)^\alpha h (x,t) = t^{\lambda - \frac{3}{4\alpha} -1} ((-\Delta)^{\alpha} \Omega)(t^{-\frac{1}{2\alpha}} x).
\end{align*}

Then, by equation \eqref{equation},
\begin{equation}\label{fractional_heat}
    \pt h+(-\Delta)^\al h=g, \quad h(0) \equiv 0,
\end{equation}
where the zero initial condition
is due to
$\|h(t)\|_{L^2}=t^{{\rm Re} \lambda}\|\Omega\|_{L^2}
\to 0$, as $t\to 0$.

Since
\begin{align}   \label{g-Z}
   \|g(t)\|_{L^2} = t^{{\rm Re \lambda}-1} \|Z\|_{L^2}
\end{align}
is locally integrable in $(0,1)$,
we can write
\begin{align*}
   h(\cdot, t) = \int_0^t e^{-(t-s)(-\Delta)^\alpha} g(s) \dif s,
\end{align*}
which yields that
\begin{align*}
    \|h(\frac 14)\|_{L^2} \leq \int_0^{\frac 14} \|g(s)\|_{L^2} \dif s
    = \int_0^{\frac 14} s^{-1+{\rm Re} \lambda} \dif s \|Z\|_{L^2}
    \lesssim \|Z\|_{L^2}.
\end{align*}

Then, for any $t\geq \frac 14$,
\begin{align*}
   h(t) = h(\frac 14) + \int_{\frac 14}^t e^{-(t-s)(-\Delta)^\alpha} g(s) \dif s.
\end{align*}
Since $g(t,x)$ is no longer singular for $t\in [1/4,1]$,
energy estimate gives
\begin{align*}
   \|h\|_{C([\frac 14, 1]; L^2)} + \|h\|_{L^2(\frac 14, 1; H^\alpha)}
   \lesssim \|h(\frac 14)\|_{L^2} + \|g\|_{C([\frac 14, 1]; L^2)}
   \lesssim \|Z\|_{L^2}.
\end{align*}
In particular, $h(t)\in H^\alpha$ for a.e. $t\in (1/4,1)$.
Applying energy estimate again but with $h(t)\in H^\alpha$ as initial condition
we get for a.e. $t\in (\frac 14,1)$
\begin{align*}
   \|h\|^2_{C([t, 1]; H^\alpha)} + \|h\|^2_{L^2(t, 1; H^{2\alpha})}
   \lesssim \|h(t)\|^2_{H^\alpha} + \|g\|^2_{C([t, 1]; H^\alpha)}
   + \|Z\|_{L^2},
\end{align*}
where the implicit constant is independent of $t$.
Then, integrating over $(1/2, 3/4)$ yields
\begin{align*}
   \|h\|_{C([\frac 34, 1]; H^{\alpha})} + \|h\|_{L^2(\frac 34, 1; H^{2\alpha})}
   \lesssim \|h\|_{L^2(\frac 12, \frac 34; H^\alpha)} + \|g\|_{C([\frac 12, 1]; H^\alpha)}
   +  \|Z\|_{L^2}
   \lesssim \|Z\|_{L^2}.
\end{align*}
Applying the above strategy again we obtain
\begin{align*}
   \|h\|_{C([\frac 78, 1]; H^{2\alpha})} \lesssim \|Z\|_{L^2}.
\end{align*}
This, via equation \eqref{fractional_heat} and \eqref{g-Z}, also yields
\begin{align*}
   \| \partial_t h\|_{C([\frac 78, 1]; L^2)} \lesssim \|Z\|_{L^2}.
\end{align*}

Thus, taking into account
\begin{align*}
   h(1)=\Omega,\ \ \pt h(1)=(-\frac{3}{4\alpha} + \lambda)\Omega -\frac{\xi}{2\al} \cdot \nabla_{\xi} \Omega,
\end{align*}
we obtain  \eqref{parabolic_regu1}, as claimed.

Now,
the first statement follows from estimate \eqref{parabolic_regu1}.
Actually, for any $\Omega \in \cald_{vor}$,
let $Z :=(\lambda - \D) \Omega \in L^2_{aps}$.
Then, take a sequence of Schwartz functions $Z_n$
such that $Z_n \to Z$ in $L^2_{aps}$,
and let $\Omega_n\in \cald_{vor}$ be such that
$(\lambda - \D) \Omega_n = Z_n$.
We have that $\Omega_n\in \cals$,
and by estimate \eqref{parabolic_regu1},
$|||\Omega_n - \Omega||| \lesssim \|Z_n - Z\|_{L^2} \to 0$ as $n\to \infty$.

It also follows from \eqref{parabolic_regu1} that
the second statement $(ii)$ holds.

Regarding the dissipativity of $\D$,
it suffices to prove that
\begin{align}  \label{D-dissp-2}
     (\D \Omega, \Omega) \leq 0,\ \ \forall \Omega \in \cald_{vor}.
\end{align}
Actually, by the integration-by-parts formula,
\begin{align*}
   (\xi \cdot \na \Omega, \Omega)
   = - \frac{3}{2} \|\Omega\|^2_{L^2}.
\end{align*}
This yields that
\begin{align*}
   (\D\Omega, \Omega)
   = - \|\Omega\|_{\dot{H}^\alpha}^2 + \frac{3}{4\alpha} \|\Omega\|_{L^2}^2
      + \frac{1}{2\alpha} (\xi \cdot \na \Omega, \Omega)
   = - \|\Omega\|_{\dot{H}^\alpha}^2 \leq 0.
\end{align*}
Thus, we obtain \eqref{D-dissp-2} and so \eqref{D-dissip}.

Finally, by virtue of the Hille-Yosida theorem (cf. \cite[Theorem 2.6]{EK86}),
it follows that $(\D, \cald_{vor})$ is the generator of a strongly continuous contraction semigroup
on $L^2_{aps}$.
In particular, as a generator,
$(\D, \cald_{vor})$ is closed.
The proof is complete.
\end{proof}

\begin{lemma}  [Relative compactness of $\M$] \label{Lem-M-Oper}
For the operator $(\M, \cald_{vor})$ we have
\begin{enumerate}
  \item[$(i)$] $\M$ is skew-adjoint, dissipative
\begin{align*}
        \|(\lambda-\M)\Omega\|_{L^2_{aps}} \geq \lambda \|\Omega \|_{L^2_{aps}},\ \ \forall \Omega\in L^2_{aps}, \ \ \forall \lambda >0,
     \end{align*}
   and
   \begin{align*}
    \rho(\M) \supseteq \{\lambda \in \mathbb{C}: {\rm Re} \lambda >0\}.
   \end{align*}
   In particular, $(\M, \cald_{vor})$ generates a contraction semigroup on $L_{aps}^2$

  \item[$(ii)$] $\M$ is $\D$-compact, i.e.,
  $\M: (\cald_{vor}, |||\cdot|||_{G}) \to L^2_{aps}$ is compact,
  where $|||\cdot\||_{G}$ denotes the graph norm
  $|||\Omega\||_{G}:= \|\Omega\|_{L^2_{aps}} + \|\D \Omega\|_{L^2_{aps}}$, $\forall \Omega \in \cald_{vor}$.

  \item[$(iii)$] $\M$ is $\D$-bounded, and for any $\ve >0$
  there exists $C_\ve >0$ such that
  \begin{align*}
      \|\M \Omega\|_{L^2_{aps}} \leq \ve \|\D \Omega\|_{L^2_{aps}} + C_\ve \|\Omega\|_{L^2_{aps}},
      \ \ \forall\ \Omega \in \cald_{vor}.
  \end{align*}
\end{enumerate}
In particular,
$(\beta^{-1} \D + \M, \cald_{vor})$ generates a contraction semigroup on $L^2_{aps}$.
\end{lemma}

\begin{proof}
$(i)$ An application of the integration-by-parts formula
gives the skew-adjointness of $\M$,
which in turn yields dissipativity.
The skew-adjointness also yields that the spectrum of $\M$ belongs to the imaginary axis,
and thus
\begin{align*}
    \rho(\M) \supseteq \{\lambda \in \mathbb{C}: {\rm Re} \lambda >0\}.
\end{align*}
Hence,
in view of the Hille-Yosida theorem,
$(\M,\cald_{vor})$ generates a contraction semigroup on $L^2_{aps}$.

$(ii)$
Regarding the $\cald_{vor}$-compactness,
we take any sequence $\{\Omega_n\} \subseteq  \cald_{vor}$ with
\begin{equation*}
    \sup\limits_{n\in \bbn}
    |||\Omega_n|||_{G} \leq  C <\infty.
\end{equation*}
By estimate \eqref{parabolic_regu1},
$\{\Omega_n\}$ are uniformly bounded in $H^{2 \alpha}$.
Then the compact embedding
$H^{2 \al}_{\supp\bar{U}} \hookrightarrow H^1_{\supp\bar{U}}$
yields that
there exists a subsequence $\{n_k\}$
such that $\{\Omega_{n_k}\}$ is a Cauchy sequence in $H^1_{\supp\bar{U}}$.
This suffices to yield that $\{\M \Omega_{n_k}\}$
is a Cauchy sequence in $L^2_{aps}$ by noticing $\bar{U}$ is smooth and compactly supported,
thereby proving the $\D$-compactness of $\M$.

$(iii)$
Since Lemma \ref{Lem-D-Oper},
$(\D, \cald_{vor})$ is a generator.
Using the $\D$-compactness
and \cite[Chapter III, Lemma 2.16]{En2000}
we thus obtain the $\D$-boundedness of $\M$.

Finally,
since $(\D, \cald_{vor})$ is a generator,
by the statements $(i)$, $(iii)$
and \cite[Chapter III, Theorem 2.7]{En2000},
$(\beta^{-1} \D + \M, \cald_{vor})$ generates a contraction semigroup on $L^2_{aps}$.
\end{proof}

As a consequence of Lemmas \ref{Lem-D-Oper}, \ref{Lem-M-Oper}
and the boundedness of operator $\boldsymbol{S}$
we have

\begin{proposition} (Approximate operators)  \label{Prop-DMS-Oper}
For any $\beta \in (0,\infty)$,
we have
\begin{enumerate}
  \item[$(i)$] Generator:
  the operator $(\beta^{-1} \D + \M + \boldsymbol{S}, \cald_{vor})$
  is densely defined, closed and generates a semigroup on $L^2_{aps}$.

  \item[$(ii)$] Resolvent set:
  \begin{align*}
        \rho (\beta^{-1} \D + \M + \boldsymbol{S}) \supseteq \{\lambda \in \mathbb{C}: {\rm Re} \lambda >\mu\}.
  \end{align*}

  \item[$(iii)$] Uniform resolvent bound:
  for any $\lambda \in \mathbb{C}$ with ${\rm Re} \lambda >\mu$,
  \begin{align}  \label{resol-DMS-bdd}
     R(\lambda, \beta^{-1} \D + \M + \boldsymbol{S}) \leq ({\rm Re} \lambda -\mu)^{-1}.
  \end{align}
\end{enumerate}
\end{proposition}

\begin{proof}
The first statement follows from the stability of semigroups
with respect to bounded perturbations,
see \cite[p158, Theroem 1.3]{En2000}.
The second and third statements
are the standard consequence of boundedness perturbations of resolvents,
see \cite[p158, Theroem 1.3]{En2000}.
\end{proof}

Concerning the inviscid operator $\M + \boldsymbol{S}$,
the following result holds.

\begin{proposition} (Inviscid operator) \label{Prop-MS-Invis}
For the operator $(\M + \boldsymbol{S}, \cald_{vor})$ we have

\begin{enumerate}
  \item[$(i)$] Generator:
  the operator $(\M + \boldsymbol{S}, \cald_{vor})$
  is densely defined, closed and generates a semigroup on $L^2_{aps}$.

  \item[$(ii)$]  Resolvent set:
  \begin{align*}
     \rho(\M +  \boldsymbol{S}) \supseteq \{\lambda \in \mathbb{C}: {\rm Re} \lambda >\mu\},
  \end{align*}
  and for any $\lambda \in \mathbb{C}$ with ${\rm Re} \lambda >\mu$,
\begin{align}\label{resolvent_bd2}
     \left\|R(\lambda,  \M+\boldsymbol{S})\right\|_{L_{\mathrm{aps}}^2 \rightarrow L_{\mathrm{aps}}^2}
        \leq(\mathrm{Re} \lambda-\mu)^{-1}.
\end{align}

  \item[$(iii)$]
  Locally uniform convergence:
For any $\Omega\in L_{\mathrm{aps}}^2$,
\begin{equation}\label{convergence_noK}
    R(\lambda, \beta^{-1} \D+\M+\boldsymbol{S}) \Omega
    \to  R(\lambda, \M+\boldsymbol{S}) \Omega\ \  {\rm  in }\  L^2, \ \
    {\rm as} \ \beta\to \infty,
\end{equation}
locally uniformly in $\{\lambda \in \mathbb{C}: \mathrm{Re} \lambda>\mu\}$.

\end{enumerate}

\end{proposition}

\begin{proof}
Since $(\M, \cald_{vor})$ is a generator
and $ \boldsymbol{S}$ is a bounded operator,
the first two assertions follows similarly
from the stability of semigroup and resolvents
with respect to the bounded perturbations, as
in the proof of Proposition \ref{Prop-DMS-Oper}.

Regarding assertion $(iii)$,
let us first consider the initial datum $\Omega_0 \in L_{\aps}^2\cap \cals$.
In view of Proposition \ref{Prop-DMS-Oper} and Proposition \ref{Prop-MS-Invis}(i)-(ii),
$\beta^{-1}\D+\M+\boldsymbol{S}$ and $\M+\boldsymbol{S}$ generate semigroups on $L^2_{aps}$, which are bounded by
\begin{align}  \label{Semi-DMS-bdd}
    \norm{e^{\tau(\beta^{-1} \D+\M+\boldsymbol{S})}}\leq e^{\mu\tau},\ \ \norm{e^{\tau(\M+\boldsymbol{S})}}\leq e^{\mu\tau}.
\end{align}
Moreover, set $\Omega^\beta:=e^{\tau(\beta^{-1}\D+\M+\boldsymbol{S})}\Omega_0$,
$\Omega:=e^{\tau(\M+\boldsymbol{S})}\Omega_0$,
since $\Omega_0 \subseteq \cald_{vor}$,
$\Omega^\beta$ satisfies the hyper-viscous advection-diffusion equation
$$\partial_\tau \Omega^\beta = (\beta^{-1}\D+\M+\boldsymbol{S}) \Omega,$$
i.e.,
\begin{equation*}
    \ptau \Omega^\beta+\frac{1}{\beta}((-\Delta)^\al-\frac{3}{4\al}-\frac{\xi}{2\al} \cdot \nabla) \Omega^\beta+\bar{U} \cdot \nabla \Omega^\beta-\boldsymbol{S} \Omega^\beta=0,
\end{equation*}
and, similarly, $\Omega$ satisfies the inviscid equation
\begin{equation*}
    \ptau \Omega +\bar{U} \cdot \nabla \Omega -\boldsymbol{S} \Omega =0.
\end{equation*}
It follows that
\begin{align*}
& (\Omega^\beta-\Omega)+\frac{1}{\beta} ((-\Delta)^\al-\frac{3}{4\al}-\frac{\xi}{2\al} \cdot \nabla)
 (\Omega^\beta-\Omega)+\bar{U} \cdot \nabla (\Omega^\beta-\Omega)-\boldsymbol{S}
 (\Omega^\beta-\Omega)  \nonumber \\
 =&\frac{1}{\beta} (-(-\Delta)^\al+\frac{3}{4\al}+\frac{\xi}{2\al} \cdot \nabla) \Omega
 =: \frac{1}{\beta} H.
\end{align*}
Note that,
since $\Omega_0\in \cald_{vor}$,
we have $\Omega \in C([0,T];\cald_{vor})$,
and so
\begin{equation}\label{F_control}
    \|H\|_{L^2(0,T;  H_{\xi}^{-\alpha})}
    = \|((-\Delta)^\al-\frac{3}{4\al}-\frac{\xi}{2\al} \cdot \nabla_{\xi}) \Omega\|_{L^2(0,T;  H_{\xi}^{-\alpha})}
    \lesssim \|\Omega\|_{C([0,T];\cald_{vor})}<\infty.
\end{equation}

By the skew-adjointness of the operators
$-\frac{3}{4\al}-\frac{\xi}{2\al}\cdot \nabla$ and $\bar{U} \cdot \nabla$,
we compute
\begin{equation*}
    \begin{aligned}
  &\frac{1}{2} \frac{d}{d \tau} \|\Omega^\beta-\Omega \|_{L^2}^2
      +\frac{1}{\beta}  \|(-\Delta)^{\frac \alpha 2}  (\Omega^\beta-\Omega ) \|_{L^2}^2 \\
  \leq& \mu \|\Omega^\beta-\Omega \|_{L^2}^2+\frac{1}{\beta}  ( H, \Omega^\beta-\Omega ) \\
  \leq& \mu \|\Omega^\beta-\Omega \|_{L^2}^2+\frac{C}{\beta}\|H\|_{H^{-\alpha}}^2+\frac{1}{2 \beta} ( \|\Omega^\beta-\Omega \|_{L^2}^2+ \|(-\Delta)^{\frac \alpha 2}   (\Omega^\beta-\Omega )\|_{L^2}^2 ) ,
   \end{aligned}
\end{equation*}
which yields that
\begin{equation*}
     \frac{d}{d\tau} \|\Omega^\beta-\Omega \|_{L^2}^2
     \leq (2\mu+\frac{1}{\beta} )
       \|\Omega^\beta-\Omega \|_{L^2}^2+\frac{2C}{\beta}\|H\|_{{H}^{-\alpha}}^2.
\end{equation*}
Then, an application of Gronwall's inequality
yields that for any $\tau\in [0,T]$, $T<\infty$,
\begin{equation*}
     \| (\Omega^\beta-\Omega )(\tau) \|^2_{L^2}
     \leq\frac{C}{\beta} e^{(2\mu+\frac{1}{\beta})T}\int_0^T e^{-(2\mu+\frac{1}{\beta})\tau}\|H(\tau)\|_{{H}^{-\alpha}}^2 d\tau,
\end{equation*}
which along with \eqref{F_control} yields
\begin{equation*}
     \|\Omega^\beta-\Omega \|^2_{L^{\infty}(0,T;  L^2)}
    \leq   \frac{C(T)}{\beta} \|\Omega\|_{C([0,T];\cald_{vor})}
     \to  0,\ \ as\ \beta \to \infty.
\end{equation*}

Thus,
using the Laplace transform
(cf. \cite[p55, Theroem 1.10]{En2000}) and \eqref{Semi-DMS-bdd}
we derive that for any $T\in (0,\infty)$,
\begin{align}\label{convergence_est}
& \| [R (\lambda, \beta^{-1} \D+\M+\boldsymbol{S} )-R(\lambda, \M+\boldsymbol{S}) ] \Omega_0 \|_{L^2} \nonumber \\
\leq& \int_0^T e^{-s {\lambda} }\|(\Omega^\beta-\Omega)(s)\|_{L^2} \dif  s
  +\int_T^{+\infty} e^{-s {\lambda} }\|\Omega^\beta(s) \|_{L^2} \dif s
 +\int_T^{+\infty} e^{-s {\lambda} }\|\Omega(s) \|_{L^2} \dif s  \nonumber \\
\leq& \frac{C'(T)}{\beta}+ \frac{2}{\operatorname{Re}\lambda - \mu} e^{-T(\operatorname{Re} \lambda-\mu)} \|\Omega_0 \|_{L^2}.
\end{align}
For any $\lambda_0 \in\{\mathrm{Re} \lambda>\mu\}$,
$r>0$ such that $B_r(\lambda_0) \subseteq \{\rm Re \lambda >\mu\}$,
$\ve>0$,
taking  $T= T(\lambda_0,r,\ve)$ in \eqref{convergence_est} large enough
such that
\[
\frac{2}{\operatorname{Re}\lambda - \mu} e^{-T(\operatorname{Re} \lambda-\mu)} \|\Omega_0 \|_{L^2} < \frac{\ve}{2},
    \ \  \forall \;\lambda\in B_r(\lambda_0) \subseteq \{\mathrm{Re} \lambda>\mu\},
\]
and then taking $\beta$ large enough such that
$\frac{C'(T)}{\beta}<\frac{\ve}{2}$
we obtain the locally uniform convergence \eqref{convergence_noK}
in the case where $\Omega_0 \in  L_{\mathrm{aps}}^2 \cap \cals $.

Regarding the general case where $\Omega_0 \in$ $L_{\text {aps }}^2$,
we can use the uniform boundedness \eqref{resol-DMS-bdd}, \eqref{resolvent_bd2}
and density arguments.

More precisely,
take any sequence
$(\Omega_0^{(n)})_{n\geq 1}\subseteq  L_{\mathrm{aps}}^2 \cap \cals $  such that
\begin{equation}\label{C_infty}
    \Omega_0^{(n)} \to \Omega_0 \text { in } L^2,\ \ as\ n\to \infty.
\end{equation}
Then, by \eqref{resol-DMS-bdd}  and \eqref{resolvent_bd2},
\begin{align*}
       & \|[R (\lambda, \beta^{-1} \D+\M+\boldsymbol{S} )-R(\lambda, \M+\boldsymbol{S}) ] \Omega_0 \|_{L^2}\\
       \leq&\|R (\lambda, \beta^{-1} \D+\M+\boldsymbol{S} ) ( \Omega_0- \Omega_0^{(n)}) \|_{L^2}
             + \|R (\lambda,\M+\boldsymbol{S} ) (\Omega_0^{(n)}- \Omega_0) \|_{L^2} \\
       &+ \|R(\lambda, \beta^{-1} \D+\M+\boldsymbol{S}) \Omega_0^{(n)}-R(\lambda, \M+\boldsymbol{S}) \Omega_0^{(n)} \|_{L^2} \\
       \leq&  2 ({\rm Re} \lambda - \mu)^{-1} \|\Omega_0^{(n)}- \Omega_0\|_{L^2}
       + \|R(\lambda, \beta^{-1} \D+\M+\boldsymbol{S}) \Omega_0^{(n)}-R(\lambda, \M+\boldsymbol{S}) \Omega_0^{(n)} \|_{L^2}.
\end{align*}
Thus,
using  \eqref{C_infty}, \eqref{convergence_noK} with $\Omega_0^{(k)}$ replacing $\Omega_0$
we obtain \eqref{convergence_noK}
for general $\Omega \in L^2_{aps}$.
\end{proof}

Now, let us come back to the operator $\boldsymbol{T}_\beta$ and $\boldsymbol{T}_\infty$
defined in \eqref{Tbeta-def} and \eqref{Tinfty-def}, respectively.
By virtue of Propositions \ref{Prop-DMS-Oper} and \ref{Prop-MS-Invis}
and the compactness of operator $\K$,
one has the invertibility and convergence results below.
See \cite{ABC21,KMS23} for relevant arguments.

\begin{lemma} [Resolvents of $ \boldsymbol{T}_\beta$]  \label{lemma3.4}
For any compact subset $K$ of $\{\operatorname{Re} \lambda>\mu\} \backslash \sig(\boldsymbol{T}_{\infty})$,
there exists $\beta_0>0$
such that
for all $\beta \geq \beta_0$, the operator $\lambda-\boldsymbol{T}_\beta$ is invertible
for all $\lambda \in K$.

Moreover,
for any $\Omega \in L_{\mathrm{aps}}^2$,
\begin{equation*}
    R(\lambda, \boldsymbol{T}_\beta) \Omega \to  R(\lambda, \boldsymbol{T}_{\infty}) \Omega\ \  {\rm in }\ L^2,
    \ \  {\rm as}\ \beta\to \infty,
\end{equation*}
uniformly in $\lambda$ on $K$.
\end{lemma}

As a consequence,
in view of the Cauchy Theorem and uniform convergence in Proposition \ref{Prop-MS-Invis},
one has the key instability of  the vorticity operator $\boldsymbol{L}_{\mathrm{vor}}^{(\beta)}$.
See also \cite{KMS23}.

\begin{theorem} [Instability of vorticity operators] \label{theorem_vorticity}
Let $\lambda_\infty$ be an unstable eigenvalue of $\T_\infty$ with $\mathrm{Re}\lambda_\infty>\mu$.
Then, for any $\ve \in (0, \operatorname{Re} \lambda_{\infty}-\mu)$,
there exists $\beta_0>0$ such that
for any $\beta \geq \beta_0$, $\T_\beta$ has an unstable eigenvalue
$\lambda_\beta$ satisfying $\left|\lambda_\beta-\lambda_{\infty}\right|<\ve$.

In particular,
$\boldsymbol{L}_{\mathrm{vor}}^{(\beta)}$ has the  unstable eigenvalue
\begin{equation}\label{expression-lam}
    \widetilde{\lambda}_\beta=\beta\lambda_\beta+ (1-\frac{3}{4\al}).
\end{equation}
\end{theorem}

\subsection{Velocity operators}   \label{Subsec-Vel-Oper}

We now turn to the velocity operators
corresponding to the linearized operator of Navier-Stokes equation,
defined by
\begin{equation}  \label{Lvel-def}
    \begin{aligned}
\boldsymbol{L}_{\mathrm{vel}}^{(\beta)} &:
   \cald_{vel} \subseteq L_\sig^2 \rightarrow L_\sig^2, \\
-\boldsymbol{L}_{\mathrm{vel}}^{(\beta)} U
     &:=(\frac{1}{2\al}-1-\frac{\xi}{2\al} \cdot \nabla ) U+(-\Delta)^\al U
       +\beta P_H (\bar{U} \cdot \nabla U+U \cdot \nabla \bar{U}).
\end{aligned}
\end{equation}

Let us rewrite $\boldsymbol{L}_{\mathrm{vel}}^{(\beta)}$ as follows
\begin{align*}
   - \boldsymbol{L}_{\mathrm{vel}}^{(\beta)}
   := (\frac{5}{4\alpha}-1) -  \wt \D - \beta \wt \M- \beta \wt{\boldsymbol{S}}
\end{align*}
with
\begin{equation*}
\begin{aligned}
& - \wt \D U := (-\frac{3}{4\alpha} - \frac{\xi}{2\al} \cdot \nabla) U + (-\Delta)^\al U,   \\
& - \wt \M U:= P_H(\bar{U} \cdot \nabla U),\\
& - \wt{\boldsymbol{S}} U:= P_H(U \cdot \nabla \bar{U}).
\end{aligned}
\end{equation*}

Similarly to Lemmas \ref{Lem-D-Oper} and \ref{Lem-M-Oper},
we have

\begin{lemma}  [Generator characterization of $\wt \D$] \label{Lem-wtD-Oper}
For the operator $\wt \D$ we have
\begin{enumerate}
  \item[$(i)$] The Schwartz class is dense in $\cald_{vel}$
  with respect to the norm $|||\cdot |||$ as in \eqref{Dvor-def}.

  \item[$(ii)$] Resolvent set:
   \begin{align*}
     \rho (\wt \D) \supseteq \{\lambda \in \mathbb{C}: {\rm Re} \lambda>0 \}.
   \end{align*}

  \item[$(iii)$] $\wt \D$ is dissipative, i.e.,
     \begin{align*}
        |(\lambda-\wt \D)\Omega\|_{L^2_{\sigma}} \geq \lambda \|\Omega \|_{L^2_{\sigma}}, \ \ \forall \lambda >0.
     \end{align*}

\end{enumerate}
In particular, the operator
$(\wt \D, \cald_{vel})$
is densely defined, closed and generates a contraction semigroup in $L^2_{aps}$.
\end{lemma}

\begin{lemma} [Relative compactness of $\wt \M$] \label{Lem-wtM-Comp}
The operator $(\wt \M, \cald_{vel})$ is dissipative and $\wt \D$-compact, and moreover, $\wt \D$-bounded.
In particular, $(\wt\D+\beta\wt\M,  \cald_{vel})$ generates a contraction semigroup.
\end{lemma}

Note that $\bar{U}$ is smooth and compactly supported, $\wt{\boldsymbol{S}}$ is $L_{aps}^2\to L^2$ bounded.
As a consequence of Lemma \ref{Lem-wtM-Comp} and the stability theory of semigroups
with respect to bounded perturbations (cf. \cite[p158, Theroem 1.3]{En2000}),
we derive the following.

\begin{proposition} [Generator characterization of $L_{vel}^{(\beta)}$]  \label{Prop-PL-Oper}
For any $\beta \in (0,\infty)$,
the operator $L_{vel}^{(\beta)}$
is densely defined, closed and generates a semigroup $\{e^{\tau  L_{vel}^{(\beta)}} \}$
on $L^2_{\sigma}$.
\end{proposition}

Theorem \ref{Thm-Lvel-eign} below provides the crucial instability of velocity operators.
See also \cite{KMS23}.
Below we give the more detailed proof.

\begin{theorem}  [Instability of velocity operators]    \label{Thm-Lvel-eign}
There exists $\beta_0>0$
such that for any $\beta\geq\beta_0$,
the unstable eigenvalue $\widetilde{\lambda}_\beta$ given by \eqref{expression-lam}
is also an unstable eigenvalue of $\boldsymbol{L}_{\text{vel}}^{(\beta)}$.
\end{theorem}

\begin{proof}
In view of Theorem \ref{theorem_vorticity},
we may let $\widetilde{\lambda}_\beta>2$ by
taking $\beta$  large enough.
Let
$\Omega\in \cald_{vor}$
be the corresponding eigenfunction of $\boldsymbol{L}_{\mathrm{vor}}^{(\beta)}$,
i.e.,
$\boldsymbol{L}_{\mathrm{vor}}^{(\beta)} \Omega = \wt \lambda_\beta \Omega$,
and set $U:=\bs(\Omega)$.

Let us first show the algebraic identity relating the
velocity and vorticity operators
\begin{align}\label{algebra}
    \curl \boldsymbol{L}_{\text {vel }}^{(\beta)} U=\boldsymbol{L}_{\text {vor }}^{(\beta)}\Omega
\end{align}
when $U\in  \cald_{vel}$. Indeed, for the Eulerian part of $\boldsymbol{L}_{\mathrm{vel}}^{(\beta)}$,
we compute
\begin{align*}
    &\curl [P_H(\bar{U}\cdot\nabla U+U\cdot\nabla \bar{U})]\\
    =&\curl (\nabla(\bar{U}\cdot U)-\bar{U}\times\Omega-U\times\bar{\Omega})\\
    =&-\bar{U}\div \Omega+\Omega\div\bar{U}-\Omega\cdot\nabla\bar{U}+\bar{U}\cdot\nabla\bar{\Omega}\\
     &-U\div\bar{\Omega}+\bar{\Omega}\div U-\bar{\Omega}\cdot\nabla U+U\cdot\nabla\bar{\Omega}= [\bar{U}, \Omega]+[U, \bar{\Omega}].
\end{align*}
For the remaining terms in $\boldsymbol{L}_{\mathrm{vel}}^{(\beta)}$, we have
\begin{align*}
    &\curl  (\frac{1}{2\al}-1-\frac{\xi}{2\al}\cdot\nabla+(-\Delta)^\al)U\\
    =& (\frac{1}{2\al}-1+(-\Delta)^\al )\Omega+ \curl
        [U\cdot\nabla\frac{\xi}{2\al} + \frac{\xi}{2\al}\times\Omega + U\times\curl\frac{\xi}{2\al} - \nabla(\frac{\xi}{2\al}\cdot U)]\\
    =&(\frac{1}{\al}-1+(-\Delta)^\al)\Omega + \curl (\frac{\xi}{2\al}\times \Omega)\\
    =&(\frac{1}{\al}-1+(-\Delta)^\al)\Omega + [  \frac{\xi}{2\al} \div \Omega - \Omega \div \frac{\xi}{2\al} + \Omega\cdot\nabla \frac{\xi}{2\al}-\frac{\xi}{2\al}\cdot\nabla \Omega]\\
    =& (-1-\frac{\xi}{2\al}\cdot\nabla +(-\Delta)^\al )\Omega.
\end{align*}
Combining the above identities together
we obtain \eqref{algebra}, as claimed.

Below we  show that
$U  \in \cald_{vel}$, which along with \eqref{algebra} implies that
$U$ is exactly the eigenfunction of $\boldsymbol{L}_{\text {vel }}^{(\beta)}$,
i.e.,
$\boldsymbol{L}_{\text {vel }}^{(\beta)} U = \wt \lambda_\beta U$.

For this purpose,
we first claim that $\Omega \in L^1$.
This together with $\Omega \in L^2$
yield that $\Omega \in L^{\frac{6}{5}}$ by interpolation,
which,
via the Biot-Savart law (cf. \cite[Theorem 12.2]{RRS16}),
gives
$U \in L^2$.

In order to prove that $\Omega\in L^1$,
we note that
$\Omega$ satisfies the equation
\begin{equation}\label{eigenfunction_1}
    \widetilde{\lambda}_\beta \Omega
     -(1+\frac{\xi}{2\al} \cdot \nabla) \Omega+(-\Delta)^\al \Omega
     = - \beta[\bar{U}, \Omega]
       - \beta[U, \bar{\Omega}]
    =:G.
\end{equation}
Since $ \bar{\Omega}$
and $\bar{U}$ are compactly supported,
it follows that $G\in L^1$.
Undoing the similarity variables
\begin{equation}  \label{h-Omeg-g-G-sim}
    h(x, t) :=t^{\widetilde{\lambda}_\beta-1} \Omega ({t^{-\frac{1}{2\al}}}x),
    \quad g(x, t): =t^{\widetilde{\lambda}_\beta-2} G ({t^{-\frac{1}{2\al}}}x),
\end{equation}
we derive from \eqref{eigenfunction_1}
the  hyper-viscous heat equation
\begin{equation}  \label{h-g-eqa}
    \pt h+(-\Delta)^\al h=g, \quad h(0) \equiv 0,
\end{equation}
where the zero initial condition is due to the fact that
$\|h(t)\|_{L^2}=t^{\operatorname{Re} \widetilde{\lambda}_\beta+\frac{3}{4\al}-1}\|\Omega\|_{L^2}$ $\rightarrow 0$ as $t \rightarrow 0^{+}$.
Note that, for the inhomogeneous part,
\begin{equation*}
    \|g(t)\|_{L^1}=t^{\operatorname{Re} \widetilde{\lambda}_\beta+\frac{3}{2\al}-2}\|G\|_{L^1} \in C([0,1]).
\end{equation*}
Then, by \eqref{h-Omeg-g-G-sim} and \eqref{h-g-eqa},
\begin{equation*}
    \Omega =h(1)= \int_0^1 e^{-(1-s)(-\Delta)^\al} g(s) \dif s,
\end{equation*}
which along with the boundedness
$e^{-(1-s)(-\Delta)^\al} \in \mathcal{L} (L^1, L^1)$ yields
that $\Omega\in L^1$,
as claimed.

It remains to prove that
$U\in H^{2\al}$  and
$\xi\cdot\nabla U\in L^2$.

To this end,
since $U\in L^2$ and
$\norm{\nabla U}_{H^{2\al-1}}\lesssim \norm{\Omega}_{H^{2\al-1}}<\infty$
(cf. \cite[Lemma 1.17]{BV22}),
it follows immediately that $U\in H^{2\alpha}$.

Regarding the regularity $\xi\cdot\nabla U\in L^2$,
since by the above arguments $\Omega\in L^1$ and $G\in L^1$,
we infer from \eqref{eigenfunction_1}  that
$(\frac{3}{4\al}-1-\D)\Omega\in L^1$.
Taking into account $\Omega \in \cald_{vor}$,
$(\frac{3}{4\al}-1-\D)\Omega\in L^2$,
we use the interpolation to derive that
$(\frac{3}{4\al}-1-\D)\Omega \in L^{\frac{6}{5}}$.
Then,
applying the $L^p$-estimate of Biot-Savart law (see \cite[Theorem 12.2]{RRS16})
to the following identity
\begin{align*}
    \bs\circ (\frac{3}{4\alpha} -1 - \D ) \Omega
    = (\frac{5}{4\alpha} - 1 - \wt \D ) U
\end{align*}
we get $(\frac{5}{4\alpha} - 1 - \wt \D ) U\in L^2$,
which, via  $U\in H^{2\alpha}$,
yields that $\xi\cdot\nabla U\in L^2$,
thereby finishing the proof.
\end{proof}

In the following
we fix $\beta$ large enough
such that Theorem \ref{Thm-Lvel-eign} holds.
We abuse the notation, write $\bar{U}$ for $\beta \bar{U}$ in \eqref{Lvel-def}
and set
\begin{equation*}
    \begin{aligned}
   \lss &:  \cald_{vel}  \subseteq L_\sig^2 \rightarrow L_\sig^2, \\
   -\lss U & := (\frac{1}{2\al}-1-\frac{\xi}{2\al} \cdot \nabla) U+(-\Delta)^\al U+ P_H(\bar{U} \cdot \nabla U+U \cdot \nabla \bar{U})
\end{aligned}
\end{equation*}
with $U=\bs(\Omega)$.

Then, in view of Proposition \ref{Prop-PL-Oper} and Theorem \ref{Thm-Lvel-eign},
$\lss$  has an unstable eigenvalue,
which is greater than $1$ for $\beta$ large enough,
and generates a semigroup $\{e^{\tau \lss}\}$ on $L^2_{\sigma}$.

\subsection{Semigroups}
This subsection collects several key estimates for the semigroup $(e^{\tau \lss})$ generated by $\lss$.
Let us  begin with the small time estimate of the semigroup.

\begin{lemma} [Small time growth] \label{Lem-Semigroup-Small}
For every $k_2 \geq k_1\geq 0$,
$k_1, k_2 \in \mathbb{N}$,
we have
\begin{align}  \label{Semi-smallt-esti}
    \tau^{\frac{k_2 - k_1}{2\al}}\norm{ e^{\tau \lss} U_0}_{H^{k_2}}
      \leq C \|U_0\|_{H^{k_1}},\ \ \forall\ U_0\in L^2_{\sig}\cap H^{k_1}, \ \ \forall\ \tau\in [0,2].
\end{align}
In particular, by interpolation,
\begin{align}\label{Semi-smallt-esti-L2}
    \tau\norm{(-\Delta)^\al e^{\tau \lss} U_0}_{L^2}
    \leq& C \norm{U_0}_{L^2},\ \ \forall\ U_0\in L^2_{\sig}, \ \ \forall\ \tau\in [0,2].
    \end{align}
\end{lemma}

\begin{proof}
Let us consider the Schwartz initial condition $U_0\in \mathcal{S}$
as the general case can be derived by density arguments.
Let $U(\tau):= e^{\tau \lss} U_0$, $\tau \in [0,2]$.
Since $U_0\in \mathcal{S} \subseteq \cald_{vel}$,
$U$ satisfies the equation
\begin{align}\label{lss-equ}
\ptau U  = \lss U = (1-\frac{1}{2\al}+\frac{\xi}{2\al} \cdot \nabla ) U-(-\Delta)^\al U-P_H (\bar{U} \cdot \nabla U+U \cdot \nabla \bar{U})
\end{align}
with $U(0) =U_0 \in \mathcal{S} $.

In order to analyze equation \eqref{lss-equ},
we undo the similarity transform, for $t\in (0,e^2-1)$,
\begin{equation*}
    h(x,t) := {(t+1)^{\frac{1}{2\al}-1}} U(\frac{x}{(t+1)^{\frac{1}{2\al}}}, \log(t+1)),
    \ \ \bar{h} := {(t+1)^{\frac{1}{2\al}-1}} \bar{U}(\frac{x}{(t+1)^{\frac{1}{2\al}}}),
\end{equation*}
and derive from \eqref{lss-equ} that
\begin{align*}
\pt h+(-\Delta)^\al h & =-P_H(\bar{h} \cdot \nabla h+h \cdot \nabla \bar{h})
\ \ with \  h(0) =U_0.
\end{align*}

Since $\|U(\tau)\|_{H^k} \simeq \|h(t)\|_{H^k}$,
where $\tau := \log(t+1)$,
and $\tau \simeq t$ for $\tau \in (0,1)$,
it suffices to prove that  for any $t\in (0,e^2-1)$,
\begin{equation*}
     t^{\frac{k_2 - k_1}{2\al}}
      \|h(t)\|_{H^{k_2}}
    \leq C \norm{h_0}_{H^{k_1}}.
\end{equation*}

For this purpose,
we rewrite $h$ in the Duhamel form
\begin{align*}
    h(t) = e^{-t(-\Delta)^\alpha} h_0 - \int_0^t e^{-(t-s)(-\Delta)^\alpha} P_H\div  (\bar{h} \otimes h+  h\otimes \bar{h}) \dif s.
\end{align*}
Then,
\begin{align*}
   \<\na\>^{k_2} h(t)
   =& \<\na\>^{k_2 - k_1} e^{-t(-\Delta)^\alpha} (\<\na\>^{k_1} h_0) \\
    &  - \int_0^{\frac t2} \<\na\>^{k_2 -k_1} P_H\div e^{-(t-s)(-\Delta)^\alpha}
        (\<\na\>^{k_1} (\bar{h} \otimes h+  h\otimes \bar{h}) )\dif s \\
    &  -\int_{\frac t2}^t   P_H\div e^{-(t-s)(-\Delta)^\alpha}
        (\<\na\>^{k_2} (\bar{h} \otimes h+  h\otimes \bar{h}) )\dif s,
\end{align*}
where $\<\na\>$ is the operator corresponding to the symbol $\<\xi\>:=(1+|\xi|^2)^{1/2}$.

By the estimate of the heat semigroup $(e^{-t(-\Delta)^\alpha})_{t\geq 0}$,
we derive
\begin{align*}
   \|h(t)\|_{H^{k_2}}
   \lesssim& t^{-\frac{k_2 - k_1}{2\alpha}} \|h_0\|_{H^{k_1}}
           + \int_0^{\frac t2}  (t-s)^{-\frac{k_2 -k_1+1}{2\alpha}}
             \|h\|_{H^{k_1}} \dif s
     +  \int_{\frac t 2}^t   (t-s)^{-\frac{1}{2\alpha}}   \|h\|_{H^{k_2}} \dif s \\
   =:& t^{-\frac{k_2 - k_1}{2\alpha}} \|h_0\|_{H^{k_1}}
       + H_1 + H_2.
\end{align*}
Note that for $s\in (0,t/2)$,
since $t-s \simeq t$
and  by the energy estimate
\begin{align*}
   \|h\|_{C([0,e^2-1]; H^{k_1})} \lesssim \|h_0\|_{L^2},
\end{align*}
we derive
\begin{align*}
   H_1 \simeq t^{-\frac{k_2 - k_1+1}{2\alpha}} \|h\|_{C([0,e^2-1];H^{k_1})}
   \lesssim t^{-\frac{k_2 - k_1 +1}{2\alpha}} \|h_0\|_{H^{k_1}}.
\end{align*}
Moreover,
for $s\in (t/2, t)$,
since $s\simeq t$ we have
\begin{align*}
   H_2 \simeq& \int_{\frac t2}^t  (t-s)^{-\frac{1}{2\alpha}}
              s^{-\frac{k_2 - k_1}{2\alpha}} \dif s
               \sup\limits_{s\in [0,t]} (s^{\frac{k_2 - k_1}{2\alpha}} \|h(s)\|_{H^{k_2}})  \notag \\
       \lesssim& t^{-\frac{k_2 - k_1}{2\alpha} +1 - \frac{1}{2\alpha}}
                \sup\limits_{s\in [0,t]} (s^{\frac{k_2 - k_1}{2\alpha}} \|h(s)\|_{H^{k_2}}) .
\end{align*}
Combining the above estimates together
and,
since the above estimates hold for any $t\in [0,e^2-1]$,
taking the supremum over $[0,t]$
we obtain that for any $t\in [0,e^2-1]$,
\begin{align*}
   \sup\limits_{s\in [0,t]} (s^{\frac{k_2 - k_1}{2\alpha}} \|h(s)\|_{H^{k_2}})
   \lesssim  \|h_0\|_{H^{k_1}}
             + t^{1-\frac{1}{2\alpha}}
              \sup\limits_{s\in [0,t]} (s^{\frac{k_2 - k_1}{2\alpha}} \|h(s)\|_{H^{k_2}}).
\end{align*}
Thus, it follows that for $t \leq \delta_*$ with $\delta_*$ small enough,
\begin{align*}
   \sup\limits_{s\in [0,t]} (s^{\frac{k_2 - k_1}{2\alpha}} \|h(s)\|_{H^{k_2}})
   \lesssim  \|h_0\|_{H^{k_1}},\ \ \forall\ t\in [0,\delta_*].
\end{align*}
Since $\delta_*$ is independent of the initial condition $h_0$,
we can extend the above estimate to the expected interval $[0,e^2-1]$ by finite iterations.
\end{proof}

Next, we recall some preliminary results
	for general semigroups $\{e^{\tau A}\}$ with $A: D(A)\subset L_\sig^2\rightarrow L_\sig^2$ being the corresponding generator.
Define the spectral bound of $A$ by
	$$s(A):=\sup \{\operatorname{Re} \lambda: \lambda \in \sig(A)\}$$ and the essential growth bound by
	$$\omega_{\mathrm{ess}} (A):=\inf _{\tau>0} \frac{1}{\tau} \log
	\big(\inf \big\{\|e^{\tau A}-K\|: K \text { is compact}\big\} \big),$$
	both of which are bounded by the growth bound
	$$\omega_0:=\inf \left\{\omega \in R:  \exists M_\omega \geq 1
	\text { such that }\norm{e^{\tau A}} \leq M_\omega e^{\tau\omega },
	\forall\ \tau\geq0\right\}.  $$
    For any bounded linear operator $T:L_\sig^2\rightarrow L_\sig^2$,
     define   the essential spectrum by
	$$\sig_{\mathrm{ess }}(T):=\{\lambda \in \mathbb{C}: T-\lambda I \text { is not Fredholm }\}, $$
      the essential norm by
	$$
	\|T\|_{\mathrm{ess}}:=\inf \big\{\|T-K\|: K \text { is compact}\big\},
	$$
	and the essential spectral radius by
	$$r_{\mathrm{ess}}(T):=\sup \left\{|\lambda|: \lambda \in \sig_{e s s}(T)\right\}.$$

\begin{lemma}{\cite[Ch.4]{En2000}} \label{Lem-etA-A}
It holds that
$$\omega_0(A)=\inf_{\tau>0} \frac{1}{\tau} \log\|e^{\tau A}\|=\lim_{\tau\to \infty} \frac{1}{\tau} \log\|e^{\tau A}\|=\mathrm{max}\{s(A),\omega_\mathrm{ess}(\mathrm{A})\},
$$
and for any $\tau_0>0$
$$
\omega_\mathrm{ess}(A)=\inf_{\tau>0} \frac{1}{\tau} \log
\|e^{\tau A} \|_\mathrm{ess}=\lim _{\tau \rightarrow \infty} \frac{1}{\tau} \log
     \|e^{\tau A} \|_\mathrm{ess}=\frac{1}{\tau_0} \log r_\mathrm{ess} (e^{\tau_0 A} ).
$$
Moreover, for every $\omega>\omega_\mathrm{ess}(A)$, the set $\sig(A) \cap\{\operatorname{Re} \lambda>w\}$
is finite, and the corresponding spectral projection has finite rank.
\end{lemma}

The main result of this subsection contained in Proposition \ref{Prop-Semigroup} below.

\begin{proposition}  \label{Prop-Semigroup}
Let $\{e^{\tau \lss}\}$ denote the semigroup generated by the velocity operator $(\lss, \cald_{vel})$
and
$$a:=s(\lss) = \sup \left\{\operatorname{Re} \lambda: \lambda \in \sig(\lss)\right\}$$
denote the spectral bound of $\lss$.
Then, we have

\begin{enumerate}
  \item[$(i)$] Maximal instability:
   $0<a<\infty$,
   and there exist $\lambda \in \sig(\lss)$ with ${\rm Re} \lambda = a$
   and $\eta \in D(\lss)$
such that $\lss \eta=\lambda \eta$.
  Moreover,
  for any $\delta>0$,
  \begin{align}   \label{semigroup-L2-growth}
      \|e^{\tau \lss} U_0\|_{L^2}
      \leq M(\delta) e^{\tau(a+\delta)} \|U_0\|_{L^2},\ \ \forall\ U_0\in L^2_\sigma.
  \end{align}

  \item[$(ii)$]  Parabolic regularity estimates:
  For any $k_2 \geq k_1 \geq 0$ and any $\delta>0$,
\begin{equation}\label{parabolic_est}
    \quad\|e^{\tau \lss} U_0 \|_{H^{k_2}} \leq \frac{M (k_1, k_2, \delta )}{\tau^{ (k_2-k_1 )/2\al}}
     e^{\tau(a+\delta)} \|U_0 \|_{H^{k_1}}, \ \ \forall \ U_0 \in L_\sig^2 \cap H^{k_1}, \ \ \forall\ \tau>0.
\end{equation}
\end{enumerate}
\end{proposition}

\begin{proof}
$(i)$.
In the previous subsection we have already seen that $\lss$ has a positive eigenvalue,
an so the spectral bound $a>0$.
Below we claim that
	\begin{equation}  \label{ess-Lss-wtD}
		\omega_\mathrm{ess}(\lss) =\omega_\mathrm{ess}(\wt\D+1-\frac{5}{4\al})\leq \omega_0(\wt\D+1-\frac{5}{4\al})\leq 1-\frac{5}{4\al}.
\end{equation}
Then, in view of Lemma \ref{Lem-etA-A} and \eqref{ess-Lss-wtD}, $\sigma(\lss) \cap \{{\rm Re} \lambda >0 > 1- 5/(4\alpha) \geq \omega_{ess}(\lss)\}$ is finite,
and so $a= \omega_0(\lss) < \infty$.
In particular,
there exists $\lambda \in \mathbb{C}$
with ${\rm Re} \lambda = a$
and $\eta\in \cald_{vel}$
such that
$\lss \eta = \lambda \eta$.
By the definition of $\omega_0$,
\eqref{semigroup-L2-growth} follows immediately.

Since the last inequality of \eqref{ess-Lss-wtD} holds due to
Lemma \ref{Lem-etA-A} and the fact that $\{e^{\tau \wt \D}\}$ is a contraction semigroup,
it remains to prove the first equality of \eqref{ess-Lss-wtD},
i.e.,
\begin{equation*}
    \omega_\mathrm{ess}(\lss) =\omega_\mathrm{ess}(\wt\D+1-\frac{5}{4\al}).
\end{equation*}
In view of Lemma \ref{Lem-etA-A} $(ii)$,
we only need to prove that for some $s>0$,
\begin{equation*}
    r_\mathrm{ess}(e^{s \lss}) =r_\mathrm{ess}(e^{s (\wt\D+1-\frac{5}{4\al})}).
\end{equation*}
Since the essential spectrum is invariant under compact perturbations, and so is the essential spectral radius,
it reduces to verifying that
the difference $e^{s \lss} - e^{s (\wt\D+1-\frac{5}{4\al})}$ is a compact operator from $L_\sig^2$ to $L_\sig^2$.

For this purpose, take any bounded sequence
$(U_n) \subseteq L_\sig^2$,
$\sup_{n\in \bbn} \norm{U_n}_{L^2}\leq C<\infty$.
For  $s>0$ close to $0$,
by \eqref{Semi-smallt-esti-L2},
$(e^{s\lss}U_n)_n$ is bounded in $H^{2\alpha}$.
Then,
since $\bar{U}$ is compactly supported,
the compact embedding $H^{2\alpha}_{\supp{\bar{U}}} \hookrightarrow H^1_{\supp{\bar{U}}}$
yields that for a subsequence $(n_k)$,
\[
\norm{e^{s \lss}U_{n_k}-e^{s \lss}U}_{H^1_{\supp{\bar{U}}}}\rightarrow 0, \;\text{as} \;k\to\infty,
\]
for some $U\in H^1_{\supp{\bar{U}}}$.
Hence,
\begin{equation*}
    \norm{(\wt\M+\wt{\boldsymbol{S}}) e^{s\lss} (U_{n_k}-U)}_{L^2}\rightarrow 0.
\end{equation*}
Moreover, using \eqref{Semi-smallt-esti} again we have
\begin{equation*}
    \norm{(\wt\M+\wt{\boldsymbol{S}}) e^{s\lss} (U_{n_k}-U)}_{L^2}\le C s^{-\frac{1}{2\al}},
\end{equation*}
which is integrable on $[0,1]$.
Thus, using the representation
\begin{equation*}
    (e^{s\lss}-e^{s (\wt\D+1-\frac{5}{4\al})})
     (U_{n_k} - U)
     =-\int_0^1 e^{(\wt\D+1-\frac{5}{4\al})(1-s)}(\wt\M+\wt{\boldsymbol{S}}) e^{s\lss} (U_{n_k}-U) \dif s
\end{equation*}
and applying the dominated convergence theorem
we obtain
\begin{equation*}
    \norm{ (e^{s\lss}-e^{s (\wt\D+1-\frac{5}{4\al})})
     (U_{n_k} - U)}_{L^2}
    \leq \norm{\int_0^1 (\wt\M+\wt{\boldsymbol{S}}) e^{s\lss}(U_{n_k}-U)\dif s}_{L^2} \rightarrow 0,\;\text{as}\; n\rightarrow \infty.
\end{equation*}
This gives the desirable compactness and \eqref{ess-Lss-wtD}.

$(ii)$.
Regarding the parabolic regularity estimate \eqref{parabolic_est},
in the small time regime where $0<\tau<2$,
\eqref{parabolic_est} follows immediately from \eqref{Semi-smallt-esti} .

Moreover, for the large time regime where $\tau\geq 2$,
by \eqref{Semi-smallt-esti} with $\tau =1$,
\begin{align*}
    \norm{e^{\tau\lss}U_0}_{H^{k_2}}
      =  \norm{e^{\lss}e^{(\tau-1)\lss}U_0}_{H^{k_2}}
    \lesssim  \norm{e^{(\tau-1) \lss } U_0}_{L^2},
\end{align*}
Then, by the $L^2$-growth \eqref{semigroup-L2-growth}
and the boundedness
$e^{-\frac{1}{2}\tau \delta} \tau^{k_2 - k_1} \lesssim 1$,
we obtain
\begin{align*}
   \norm{e^{\tau\lss}U_0}_{H^{k_2}}
   \lesssim e^{\tau(a+\frac 12 \delta)} \|U_0\|_{L^2}
   \lesssim \frac{e^{\tau(a+\delta)}}{\tau^{k_2 - k_1}}
             e^{-\frac{1}{2}\tau \delta} \tau^{k_2 - k_1}
             \norm{U_0}_{H^{k_1}}
    \lesssim \frac{e^{\tau(a+\delta)}}{\tau^{k_2 - k_1}}
             \norm{U_0}_{H^{k_1}},
\end{align*}
which yields \eqref{ess-Lss-wtD} for $\tau \geq 2$.
Therefore, the proof of Proposition \ref{Prop-Semigroup} is complete.
\end{proof}

\begin{corollary}\label{corol-eigenvalue}
Let $\lambda$ and $\eta$ be the eigenvalue and corresponding eigenvector of $\lss$ in Proposition \ref{Prop-Semigroup}.
Then we have that
$\eta\in H^k$, $\forall\ k\in \mathbb{N}$.
Moreover,
letting $\ul(\tau) :=\operatorname{Re}(e^{\lambda \tau} \eta)$
be the solution of the linearized
equation $\ptau \ul=\lss\ul$,
we have
\begin{equation}\label{U_Hk_est}
    \|\ul(\tau)\|_{H^k}=C(k, \eta) e^{a \tau} \quad \forall\ \tau \geq 0, \ \ \forall\ k\in \mathbb{N}.
\end{equation}
\end{corollary}

\section{Local pathwise non-uniqueness}

In this section we prove the non-uniqueness of
local-in-time probabilistically strong Leray solutions
in Theorem \ref{Thm-Nonuniq-Local}.

We note that
the non-uniquenness here holds in the probabilistically strong sense,
with the given filtrated probability space $(\Omega,(\mathcal{F}_t), \mathbb{P})$
and Wiener process $(w(t))$.
Below we fix the probabilistic argument $\omega$
and perform the pathwise analysis.
For simplicity, the dependence on
$\omega$ is omitted.

As mentioned earlier,
in view of the similarity transforms \eqref{sim-va},
it is equivalent to prove the non-uniqueness
for the similarity formulation of
the forced random Navier-Stokes equations \eqref{SNSE-Similarity}.
Recall that the ansatz for the second Leray solution is
\begin{equation*}
    U_2=\bar{U}+\ul+\up.
\end{equation*}
In view of \eqref{SNSE-Similarity}, \eqref{F-def} and \eqref{equa-Ulin},
$\up$ formally satisfies the equation
\begin{align}  \label{U_per}
  & \ptau \up-\lss \up \notag  \\
  =&-
  P_H [(\up \cdot \nabla) \up+(\ul \cdot \nabla) \up
      +(\up \cdot \nabla) \ul+(\ul \cdot \nabla) \ul] \notag \\
&- P_H [(\ul\cdot\na)W+(W\cdot\na)\ul+(\up\cdot\na)W+(W\cdot\na)\up ].
\end{align}

The proof of Theorem \ref{Thm-Nonuniq-Local}
is thus reduced to proving
Proposition \ref{Prop-Uper} below.

\begin{proposition} [Remainder profile] \label{Prop-Uper}
Fix any integer $N>\frac 52$,
$\kappa \in (\frac 14, \frac 12)$
and $\ve_0>0$ small enough such that
$1+\kappa - \frac{5}{4\alpha} - \ve_0>0$.
Let $\sigma$ be the $(\mathcal{F}_t)$-stopping time
given by \eqref{sigma-stop-def},
i.e.,
\begin{equation}  \label{sigma-def}
		\sigma:= (20C(N,\kappa))^{\frac{-1}{a-\ve_0}}\wedge (20C(N,\kappa))^{\frac{-1}{1+\kappa-\frac{5}{4\al}-\ve_0}} \wedge \inf\{t>0:\norm{w}_{C^{\kappa}([0,t]; H_{x}^{N})} \geq 1 \},
\end{equation}
where
$C(N,\kappa)$ is a large deterministic constant given by \eqref{TU-esti} and \eqref{TU1-TU2-esti} below,
depending on $N$ and $\kappa$.
Set $T:= \log \sigma$.
Then,
there exists a unique
$\up \in$ $C ((-\infty, T] ; H^N )$ to \eqref{U_per} satisfying
\begin{equation}  \label{Up-HN-reg}
    \left\|\up(\tau)\right\|_{H^N} \leq e^{(a+\ve_0) \tau},
    \quad \forall \tau \leq T .
\end{equation}
\end{proposition}

\begin{remark}
Note that,
the linear propagation profile $\ul$
and the remainder profile $\up$
have different decay rates at $-\infty$.
Hence, $\bar{U}$ and $U$ are two distinct Leray solutions to
the forced equation \eqref{SNSE-Similarity}
in the self-similarity formulation on $(-\infty, \log \sigma]$,
where the forcing term $F$ is given by \eqref{F-def}.
Thus,
using the random shift transform \eqref{v-z-u}
and the self-similarity transforms \eqref{sim-va} and \eqref{sim-va-pf}
we obtain the non-uniqueness of local-in-time
probabilistically strong Leray solutions in Theorem \ref{Thm-Nonuniq-Local}.
Moreover,
the high regularity of the constructed remainder in \eqref{Up-HN-reg}
and of the noise in $H^N$ with $N>5/2$ is sufficient to
verify the energy inequality and bound in Definition \ref{def-Prob-strong}.
\end{remark}

{\bf Proof of Proposition \ref{Prop-Uper}.}
We apply fixed point arguments to construct
the unique remainder $\up$ to \eqref{U_per}.

For this purpose,
for any $s\in \mathbb{R}$,
let us define the Banach space by
\begin{equation*}
    \mathbb{X}_s:= \{U \in C ((-\infty, s] ; H^N ):\|U\|_{\bbx_s} <\infty \},
\end{equation*}
equipped with the norm
\begin{equation} \label{X-def}
    \|U\|_{\bbx_s} :=\sup _{r<s} e^{-(a+\ve_0) r}\|U(r)\|_{H^N},
\end{equation}
where $a$ is the spectral bound of $\lss$ as in Proposition \ref{Prop-Semigroup}.
Define the map $\mathcal{T}$ on $\mathbb{X}_T$ by,
for any $\tau \leq T$,
\begin{align*}
       \mathcal{T}(U)(\tau) :=&-\int_{-\infty}^\tau e^{(\tau-s) \lss}
        P_H [(U \cdot \nabla) U+(\ul \cdot \nabla) U+(U \cdot \nabla) \ul+(\ul \cdot \nabla) \ul] \dif s \\
       &-\int_{-\infty}^\tau e^{(\tau-s) \lss}
       P_H [(\ul\cdot\na)W+(W\cdot\na)\ul+(U\cdot\na)W+(W\cdot\na)U] \dif s
\end{align*}
for any $U\in \bbx_T$,
which exactly relates to the Duhamel formulation of \eqref{U_per}.

It suffices to prove that
$\mathcal{T}$ is a contraction mapping
on the ball  $B_{\bbx_T, 1}(0):= \{U\in \bbx_T: \|U\|_{\bbx_T} \leq 1\}$.

To this end,
we decompose $\mathcal{T}$ into three parts
\begin{align*}
   \mathcal{T}(U) = B(U,U) + L U + G,    \ \ \forall U\in \bbx_T,
\end{align*}
where
$B(\cdot, \cdot)$ is the bilinear operator defined by,
for any $U,V \in \bbx_T$,
\begin{align*}
   B(U, V)(\tau) :=&-\int_{-\infty}^\tau e^{(\tau-s) \lss} P_H[(U \cdot \nabla) V(s)] \dif s,
\end{align*}
$L$ is the linear operator
\begin{align*}
        L U(\tau)
        :=&-\int_{-\infty}^\tau e^{(\tau-s) \lss} P_H\left[(\ul \cdot \nabla) U(s)
                 +(U \cdot \nabla) \ul(s) \right] \dif s \\
                           &-\int_{-\infty}^\tau e^{(\tau-s)
                           \lss} P_H\left[(W \cdot \nabla)
                 U(s)+(U \cdot \nabla) W(s)\right] \dif s \notag \\
        =:&L_1 U+L_2 U,
\end{align*}
and $G$ corresponds to the remaining inhomogeneous part
\begin{align*}
        G(\tau)
        :=&-\int_{-\infty}^\tau e^{(\tau-s) \lss}
        P_H (\ul \cdot \nabla) \ul(s) \dif s \notag \\
         & -\int_{-\infty}^\tau e^{(\tau-s) \lss}
        P_H \left[(\ul\cdot\na)W(s)+(W\cdot\na)\ul(s)\right] \dif s \notag \\
        =:& G_1 + G_2.
\end{align*}
In the sequel, let us treat three operators separately.

\textit{$(i)$ Estimate of $B(U, U)$.}
We claim that
\begin{equation}\label{B_est}
    \|B(U, U)\|_{\bbx_\tau}
    \lesssim e^{\tau (a+\ve_0) }\|U\|_{\bbx_\tau}^2,
\end{equation}
where the implicit constant depends on $N$.

To this end,
using \eqref{parabolic_est} in Proposition \ref{Prop-Semigroup}
we get that for any $\delta>0$,
\begin{equation*}
    \|B(U, U)(\tau)\|_{H^{N}}
    \lesssim \int_{-\infty}^\tau  {(\tau-s)^{-\frac{1}{2\al}}} {e^{(\tau-s)(a+\delta)}}\|(U \cdot \nabla) U(s)\|_{H^{N-1}} \dif s
\end{equation*}
Since $H^N$ is a Banach algebra when $N>5/2$,
by the definition \eqref{X-def},
\begin{equation*}
    \|(U \cdot \nabla) U(s)\|_{H^{N-1}}
    \lesssim \|U(s)\|_{H^N}^2
    \lesssim e^{2(a+\ve_0) s}\|U\|_{\bbx_s}^2.
\end{equation*}
Then, taking $\delta$ small such that
$a+ \ve_0-\delta>0$
and using a change of variable we get
\begin{align*}
       \|B(U, U)(\tau)\|_{H^{N}}
       \lesssim&
        \int_{-\infty}^\tau {(\tau-s)^{-\frac{1}{2\al}}}{e^{(\tau-s)(a+\delta)} e^{2(a+\ve_0) s}} \dif s
        \|U\|_{\bbx_\tau}^2  \\
       \lesssim&   e^{2 \tau (a+\ve_0)}
                \int_0^\infty s^{-\frac{1}{2\alpha}} e^{-s(a+2 \ve_0 - \delta)} \dif s \|U\|_{\bbx_\tau}^2 \\
       \lesssim&  e^{2 \tau (a+\ve_0) }  \|U\|_{\bbx_\tau}^2 .
\end{align*}
Thus, \eqref{B_est} is verified.

\textit{$(ii)$ Estimate of $LU$.}
We claim that
\begin{equation}\label{L-est}
    \|L U\|_{\bbx_\tau}
    \lesssim
    (e^{ a  \tau}
       + e^{ \tau(1+\kappa - \frac{5}{4\alpha})}
              \|w\|_{C^\kappa([0,e^\tau]; H^N)}) \|U\|_{\bbx_\tau}.
\end{equation}

For this purpose,
using \eqref{parabolic_est} again we have that
for any $\tau \in (-\infty, T)$,
\begin{align*}
    \|L_1 U\|_{H^{N}}
     \lesssim \int_{-\infty}^\tau  {(\tau-s)^{-\frac{1}{2\al}}}{e^{(\tau-s)(a+\delta)}}
     \left(\left\|(\ul \cdot \nabla) U\right\|_{H^{N-1}}+\left\|(U \cdot \nabla) \ul\right\|_{H^{N-1}}\right) \dif s.
\end{align*}
Then, by the algebraic property of $H^N$, \eqref{U_Hk_est} and \eqref{X-def},
\begin{align*}
    \left\|(\ul \cdot \nabla) U\right\|_{H^{N-1}}
    +\left\|(U \cdot \nabla) \ul\right\|_{H^{N-1}}
    \leq C(N, \delta) e^{(2 a+\ve_0)s}\|U\|_{\bbx_s},
\end{align*}
which yields that
\begin{align*}
    \|L_1 U\|_{H^N}
    \lesssim& \int_{-\infty}^\tau
          (\tau-s)^{-\frac{1}{2\alpha}} e^{(\tau-s)(a+\delta)} e^{s(2a+\ve_0)} \dif s \|U\|_{\bbx_\tau} \notag
    \lesssim e^{(2a+\ve_0)\tau }\|U\|_{\bbx_\tau},
\end{align*}
and thus
\begin{align}\label{L1-esti}
    \|L_1 U\|_{\bbx_\tau}
     \lesssim e^{a \tau }\|U\|_{\bbx_\tau}.
\end{align}

Moreover,
by \eqref{parabolic_est},
\begin{align*}
    \norm{L_2 U}_{H^N}
    &\les \int_{-\infty}^\tau (\tau-s)^{-\frac{1}{2\al}} e^{(\tau-s)(a+\delta)} \norm{(U\cdot\na)W+(W\cdot\na)U}_{H^{N-1}} \dif s\\
    &\les \int_{-\infty}^\tau (\tau-s)^{-\frac{1}{2\al}} e^{(\tau-s)(a+\delta)}  \norm{W}_{H^{N}}\norm{U}_{H^{N}} \dif s.
\end{align*}
Then, by \eqref{Z-est} involving the $C^{\kappa}$-H\"older continuity of $w$ in $H^N$,
where $\kappa \in (1/4, 1/2)$,
the right-hand side above can be bounded by
\begin{align*}
    & \int_{-\infty}^\tau (\tau-s)^{-\frac{1}{2\al}} e^{(\tau-s)(a+\delta)}
       e^{s(1+\kappa- \frac{5}{4\alpha})}
        \norm{w}_{C^{\kappa}([0,e^\tau]; H_{x}^{N})}  e^{s(a+\ve_0)}\norm{U}_{\bbx_s} \dif s\\
     &\les e^{\tau(a+1+\kappa -\frac{5}{4\alpha}+\ve_0)}
          \int_0^\infty
          s^{-\frac{1}{2\alpha}}
          e^{-s(1+\kappa - \frac{5}{4\alpha} -\delta + \ve_0)} ds
            \norm{w}_{C^{\kappa}([0,e^\tau]; H_{x}^{N})}  \norm{U}_{\bbx_\tau}.
\end{align*}
Choose $\delta \in (0,1/{12})$ small enough
such that
\begin{align*}
   1+\kappa - \frac{5}{4\alpha} -\delta
   \geq \kappa - \frac 14  - \delta
   \geq \frac{1}{12} - \delta >0.
\end{align*}
Then, the above integration is finite.
It follows that
\begin{equation}\label{L2-esti}
    \|L_2 U(\cdot, \tau)\|_{\bbx_\tau}
    \lesssim e^{\tau(1+\kappa -\frac{5}{4\alpha})}
     \norm{w}_{C^{\kappa}([0,e^\tau]; H_{x}^{N})}   \|U\|_{\bbx_\tau}.
\end{equation}
A combination of \eqref{L1-esti} and \eqref{L2-esti}
verifies \eqref{L-est}.

\textit{$(iii)$ Estimate of $G$.}
We show that for any $\kappa \in (1/4,1/2)$,
\begin{equation}\label{G_est}
    \|G\|_{\bbx_\tau} \lesssim e^{\tau(a-\ve_0)}
     +  e^{\tau(1+\kappa -\frac{5}{4\al}-\ve_0)}  \norm{w}_{C^{\kappa}([0,e^\tau]; H_{x}^{N})}.
\end{equation}

To this end, applying Proposition \ref{Prop-Semigroup}  we have
\begin{equation*}
 \|G_1(\tau)\|_{H^N}
 \lesssim \int_{-\infty}^\tau e^{(\tau-s)(a+\delta)}\left\|\ul \cdot \nabla \ul(s)\right\|_{H^N} \dif s.
\end{equation*}
Since by \eqref{U_Hk_est},
\begin{align*}
   \left\|\ul \cdot \nabla \ul(s)\right\|_{H^N} \leq C(N)\left\|\ul(s)\right\|_{H^{N+1}}^2 \leq C(N) e^{2as},
\end{align*}
we get
\begin{align*}
   \|G_1(\tau)\|_{H^N}
   \lesssim \int_{-\infty}^\tau e^{(\tau-s) (a+\delta)} e^{2a s} \dif  s
   \lesssim e^{2a \tau},
\end{align*}
which yields that
\begin{align} \label{G1-esti}
   \|G_1(\tau)\|_{\bbx_\tau}
   \lesssim e^{\tau(a-\ve_0)}.
\end{align}

Regarding the second term,
using Proposition \ref{Prop-Semigroup}, \eqref{U_Hk_est} and \eqref{Z-est}
again we derive
\begin{align*}
    \|G_2(\tau)\|_{H^N}
     &\lesssim \int_{-\infty}^\tau {(\tau-s)^{-\frac{1}{2\al}}} e^{(\tau-s)(a+\delta)} \norm{\ul}_{H^{N}} \norm{W}_{H^{N}} \dif s\\
    &\lesssim \int_{-\infty}^\tau
    (\tau-s)^{-\frac{1}{2\alpha}}  e^{(\tau-s)(a+\delta)} e^{sa}
              e^{s(1+\kappa -\frac{5}{4\al})} \dif s
               \norm{w}_{C^{\kappa}([0,e^\tau]; H_{x}^{N})} \\
    &\lesssim e^{\tau(a+1+\kappa-\frac{5}{4\al})} \norm{w}_{C^{\kappa}([0,e^\tau]; H_{x}^{N})},
\end{align*}
where in the last step we used
$1+\kappa - \frac{5}{4\alpha} - \delta >0$.
This yields that
\begin{align} \label{G2-esti}
   \|G_2(\tau)\|_{\bbx_\tau}
   \lesssim e^{\tau(1+\kappa-\frac{5}{4\al}-\ve_0)} \norm{w}_{C^{\kappa}([0,e^\tau]; H_{x}^{N})}.
\end{align}
Thus, it follows from \eqref{G1-esti} and \eqref{G2-esti} that
\eqref{G_est} holds.

Thus, combining estimates \eqref{B_est}, \eqref{L-est} and  \eqref{G_est} altogether
we conclude that
there exists a large deterministic constant
$C(N,\kappa)$,
only depending on $N, \kappa$,
such that
\begin{align}  \label{TU-esti}
       \norm{\mathcal{T}(U)}_{\bbx_\tau}
       \leq& C(N,\kappa)
            \big[e^{ \tau(a-\ve_0)}   + e^{\tau(1+\kappa - \frac{5}{4\alpha} - \varepsilon_0)}  \norm{w}_{C^{\kappa}([0,e^\tau]; H_{x}^{N})} \notag  \\
          &\qquad \quad + ( e^{\tau a} + e^{\tau(1+\kappa-\frac{5}{4\al})}
              \norm{w}_{C^{\kappa}([0,e^\tau]; H_{x}^{N})}) \|U\|_{\bbx_\tau}
            + e^{\tau(a+\ve_0)}  \|U\|_{\bbx_\tau}^2 \big],
\end{align}
Similarly,
we also have that  for any $U_i \in B_{\bbx_\tau, 1}(0)$,  $i=1,2$,
\begin{align}  \label{TU1-TU2-esti}
   \norm{\mathcal{T}(U_1) - \mathcal{T}(U_2)}_{\bbx_\tau}
   \leq& C(N,\kappa)
         ( e^{\tau  a } + e^{\tau(1+\kappa-\frac{5}{4\al})}
        \norm{w}_{C^{\kappa}([0,e^\tau]; H_{x}^{N})} )
        \|U_1 - U_2\|_{\bbx_\tau}.
\end{align}

Therefore,
for $T=\log \sigma$ with $\sigma$ given by \eqref{sigma-def},
we derive that $\mathcal{T}$ is a contraction mapping on
the ball $B_{\bbx_T,1}(0)$.
The proof is complete.
\hfill $\square$

\section{Global  martingale solutions}   \label{Sec-Global-Mart}

This section is devoted to martingale solutions
to stochastic Navier-Stokes equation,
including the global existence,
stability with respect to initial conditions
and the gluing procedure.

\subsection{Setups.}\label{Subsec-Mart-Setup}
Let us first prepare some preliminaries concerning
the functional spaces
and canonical probability spaces for martingale solutions.

{\bf Function spaces.}
Let $\calc_0^\infty:= \{u\in C_c^\infty: \div\ u =0\}$,
$\calh$, $\calv$ and $\calw$ be the closure of $\calc_0^\infty$ in $L^2$, $H^\alpha$
and $H^N$, respectively, where $\alpha \in [1,5/4)$, $N> 5/2$.
Let $\calh'$ and $\calv'$ denote the corresponding dual spaces.

Since the embedding $\calw \hookrightarrow \calh$ is continuous,
by \cite[Lemma C.1]{Brz13},
there exists a Hilbert space $\calu$ such that $\calu$
is dense in $\calw$ and the embedding $\calu \hookrightarrow \calw$ is compact.

One has the embeddings
\begin{align*}
   \calu  \hookrightarrow  \calw \hookrightarrow \calv \hookrightarrow \calh \cong \calh'
   \hookrightarrow \calv'  \hookrightarrow \calw' \hookrightarrow \calu'.
\end{align*}

As in \cite{Brz13,Roz2005},
we consider the following four functionals spaces for the
solvability of martingale solutions to stochastic Navier-Stokes equation
on the whole space:
\begin{enumerate}
  \item[$(i)$]
   $C([0,T]; \calu')$ denotes  the space of continuous functions
             $u:[0,T]\to \calu'$ with
           the topology $\calt_1$
         induced by  the  norm
          \begin{align*}
               \|u\|_{C([0,T]; \calu')}:=\sup_{t\in [0,T]} \|u(t)\|_{\calu'};
          \end{align*}
  \item[$(ii)$]
   $L^2_w(0,T; \calv)$ is the space
             $L^2(0,T; \calv)$ with the weak topology $\calt_2$;
  \item[$(iii)$]
   $L^2(0,T; \calh_{\loc})$ is the space of measurable functions
            $u:[0,T]\to \calh$
           with the topology $\calt_3$ induced by the seminorms
   \begin{align*}
      \|u\|_{L^2(0,T;{\calh_R})}^2:= \int_0^T \int_{B_R} |u(t,x)|^2 \dif x ,\ \ R\in \mathbb{N};
   \end{align*}
  \item[$(iv)$]
   $C([0,T]; \calh_w)$ denotes the space of weakly continuous functions
               $u: [0,T]\to \calh$ with the weakest topology $\calt_4$ such that
               the mappings $C([0,T]; \calh_w) \ni u\mapsto (u(\cdot), v) \in C([0,T]; \mathbb{R})$
               are continuous for all $v\in \calh$.
\end{enumerate}

Let
\begin{align*}
		\calz_{[0,T]}:= C([0,T]; \calu') \cap L^2_w(0,T; \calv)
		\cap L^2(0,T; \calh_{\loc}) \cap  C([0,T]; \calh_{w})
\end{align*}
with the topology being the supremum of $\calt_i$, $1\leq i\leq 4$.
Similarly, let \begin{align*}
		\wh\calz_{[0,T]}:= C([0,T]; \calu') \cap  L^2(0,T; \calh_{\loc}) \cap  C([0,T]; \calh_{w}),
\end{align*}
and
	\begin{align}  \label{whcalz-def}
		\scrx_{[0,T]} := C([0, T]; \calu') \cap L^2(0,T; \calh_{\loc}).
	\end{align}
Define the global version by
	\begin{align}\label{whcalz-loc-def}
		\scrx_{[0,\infty), \loc}
		:= C_{\loc}([0, \infty); \calu') \cap L^2_{\loc}(0,\infty; \calh_{\loc})
	\end{align}
with the locally uniform topology.
$\calz_{[0,\infty), \loc}$ and $\wh\calz_{[0,\infty), \loc}$
can be defined similarly.

Let
\begin{align*}
   \scry_{[0,T]}  := C([0,T]; \calw)
\end{align*}
with the uniform convergence topology, and
\begin{align*}
    \scry_{[0,\infty);\loc} := C_{\loc}([0,\infty); \calw)
\end{align*}
with the locally uniform topology.

{\bf Canonical spaces and processes.}
Define the product canonical space by
\begin{align*}
   \wh \Omega:=  \scrx_{[0,\infty),\loc} \times \scry_{[0,\infty),\loc}.
\end{align*}

Let $\mathscr{P}(\wh \Omega)$ denote the set of all probability measures
on $(\wh \Omega, \mathcal{B}(\wh \Omega))$
with $\mathcal{B}(\wh \Omega)$ being the Borel $\sig$-algebra coming
from the topology of locally uniform convergence on $\wh \Omega$.

Let $z=(x,y)$
denote the canonical process on $\wh \Omega$
given by
$$
z(t,\omega) = (x(t,\omega), y(t,\omega)) := \omega(t), \ \ \forall \omega\in \wh \Omega.
$$
We also use the notation $z_0=(x_0,y_0)$
to denote the deterministic initial data.

Similarly, for any $t \geq 0$,
let
\begin{align*}
         \wh \Omega^t:=\scrx_{[t,\infty), \loc} \times \scry_{[t,\infty), \loc},
\end{align*}
equipped with its Borel $\sig$-algebra $\mathcal{B}^t$ which coincides with
$\sig\{z(s), s \geq t\}$.
Denote the canonical filtration by
$\mathcal{B}_t^0:=\sig\{z(s), s \leq t\}$, $t \geq 0$,
and its right continuous version by
$\mathcal{B}_t:=\cap_{s>t} \mathcal{B}_s^0, t \geq 0$.

{\bf Martingale solutions and probabilistically weak solutions.}
Global martingale solutions are taken in the sense of Definition \ref{def-mart-1}
in Section \ref{Sec-Intro}.

Recall that
$\mathscr{M}(s_0, z_0, C_{\cdot})$ is the set of all martingale solutions
starting from $z_0$ at time $s_0$ and satisfying $(M1)$-$(M3)$.
As we will see Corollary \ref{Cor-Mart-Comp} below,
$\mathscr{M}(s_0, z_0, C_{\cdot})$ is a compact set in $\mathscr{P}(\wh \Omega)$.

Similarly,
one can define martingale solutions up to a stopping time $\tau: \wh \Omega\rightarrow [0,\infty)$.
For this purpose, we define the trajectory space stopped at time $\tau$ by
$$
 \wh \Omega_{\tau}:= \{\omega(\cdot \wedge \tau(\omega)) ; \omega \in \wh \Omega \} .
$$

\begin{definition} [Martingale solutions up to stopping times] \label{def-mart-2}
Let $z_0 :=(x_0,y_0) \in \calh  \times \calw$,
$\tau \geq 0$ be a $(\mathcal{B}_t )_{t\geq 0}$-stopping time. A probability measure
$P \in \mathscr{P} ( \wh \Omega_{\tau})$
is a martingale solution to \eqref{Hyper-SNSE} on $[0,\tau]$
with the initial value $z_0$ at time $0$
provided
\begin{enumerate}
  \item[$(M1')$] $P(z(0)=z_0 )=1$ and $P$-a.s.
  for any $t\geq 0$ and any $\psi\in \calc_0^\infty$,
\begin{align*}
   (x({t\wedge \tau}) - x_0, \psi) =&
   \int_0^{t\wedge \tau} (-(-\Delta)^\alpha x(s) - \div (x(s)\otimes x(s)) + f(s, y(s))
   , \psi)\dif s\\
   & + (y({t\wedge \tau}) - y_0, \psi).
\end{align*}

 \item[$(M2')$]
 Energy inequality:
    for any $t\geq 0$,
    \begin{align*}
       & \mathbb{E}^{P} \|x(t\wedge \tau)\|_{\calh}^2
       + 2 \mathbb{E}^{P} \int_{0}^{t\wedge \tau} \|(-\Delta)^{\frac \alpha 2} x(s)\|_{\calh}^2 \dif s \\
       \leq&   \|x_0\|_{\calh}^2
            +  2 \mathbb{E}^{P} \int_{0}^{t\wedge \tau} (f(s,y(s)), x(s)) \dif s
            + \sum_{j=1}^{\infty}\lambda_j\|\epsilon_j\|^2_{\calh} (t\wedge \tau).
    \end{align*}
  Moreover, the energy bound holds:
    \begin{align*}
       \mathbb{E}^{P}  (\sup_{0\leq s\leq t\wedge \tau}\|x(s)\|_{\calh}^2
       + \int_{0}^{t\wedge \tau} \| x(s)\|_{\calv}^2 \dif s)
       \leq C_t (\|x_0\|_{\calh}^2 +1),
\end{align*}
  where $t\mapsto C_t$ is a positive increasing constant.

  \item[$(M3')$]  $(y(t))$ is a $\mathcal{Q}$-Wiener process starting from $y_0$ at time $0$ under $P$.
\end{enumerate}
As in Definition \ref{def-mart-1},
let $\mathscr{M}(0, z_0, C_{\cdot}, \tau)$ denote the set of all martingale solutions on $[0,\tau]$
satisfying $(M1')$-$(M3')$.
\end{definition}

As in the usual case,
martingale solutions defined above are closely related to probabilistically weak solutions.

\begin{definition} [Probabilistically weak solutions]  \label{def-weak}
We say that $(\wt \Omega, (\widetilde{\mathcal{F}}_t), \wt \bbp, \wt v, \wt w)$
is a probabilistically weak solution to  \eqref{Hyper-SNSE}
with the initial condition $z_0=(x_0,y_0) \in  \calh  \times \calw$
at time $s_0 (\geq 0)$,
if
$(\wt \Omega, (\widetilde{\mathcal{F}}_t), \wt \bbp)$
is a filtrated probability space,
$\wt v$ and $\wt w$ are $(\widetilde{\mathcal{F}}_t)$-adapted
continuous processes
in $\calu'$ and $\calw$, respectively,
satisfying
\begin{enumerate}
  \item[$\wt{(M1)}$] $\wt \bbp \left( (\wt v(t), \wt w(t)) =(x_0, y_0), 0 \leq t \leq s_0 \right)=1$
  and for any $t\geq s_0$ and any $\psi\in \calc_0^\infty$,
\begin{equation*}
	\begin{aligned}
		(\wt v(t) -  x_0, \psi) = &\int_0^t (-(-\Delta)^\alpha \wt v(s) -  \div (\wt v(s)\otimes \wt v(s))
		+  f(s, \wt w(s)), \psi) \dif s\\
		 &+ ( \wt w(t) -  y_0, \psi),\ \  \wt{\bbp}-a.s.
	\end{aligned}
\end{equation*}

  \item[$\wt{(M2)}$]
  For any $t\geq s_0$,
  \begin{align*}
  	& \mathbb{E}^{\wt\bbp} \|\wt v(t)\|_{\calh}^2
  	+ 2 \mathbb{E}^{\wt\bbp} \int_{s_0}^{t} \|(-\Delta)^{\frac \alpha 2} \wt v(s)\|_{\calh}^2 \dif s \\
  	\leq&   \|x_0\|_{\calh}^2
  	+  2 \mathbb{E}^{\wt\bbp} \int_{s_0}^{t} (f(s,\wt w(s)), \wt v(s)) \dif s
  	+ \sum_{j=1}^{\infty}\lambda_j\|\epsilon_j\|^2_{L^2} (t-s_0 ).
  \end{align*}
  Moreover, there exists a positive increasing function $t \mapsto C_{t}$ such that for all $t \geq s_0$,
\begin{equation*}
    \mathbb{E}^{\wt \bbp} (\sup _{s \in[0, t]}\|\wt v(s)\|_{\calh}^{2}
    +\int_{s_0}^t\|\wt v(s)\|_{\calv}^2  \dif s )
     \leq C_{t} (\|x_0 \|_{\calh}^{2}+1 ).
\end{equation*}

  \item[$\wt{(M3)}$] $(\wt w(t))$ is an $(\widetilde{\mathcal{F}}_t)$-adapted $\mathcal{Q}$-Wiener process starting from $y_0$ at time $s_0$ under $\wt \bbp$.

\end{enumerate}
\end{definition}

Proposition \ref{Prop-equi-mart} below shows the equivalence between
martingale solutions and probabilistically weak solutions.

\begin{proposition}  \label{Prop-equi-mart}
If
$(\wt \Omega, (\widetilde{\mathcal{F}}_t), \wt{\mathbb{P}}, \wt v, \wt w)$
is a probabilistically weak solution to \eqref{Hyper-SNSE}
with the initial condition $z_0 \in  \calh  \times \calw$
at time $s_0(\geq 0)$
as in the sense of Definition \ref{def-weak},
then there exists a martingale solution
$P_{s_0, z_0} \in \mathscr{M}(s_0, z_0, C_{\cdot})$  to \eqref{Hyper-SNSE}
for some function $(C_t)$ as
in the sense of Definition \ref{def-mart-1}.
The reverse statement also holds.
\end{proposition}

\begin{proof}
Suppose that $(\wt \Omega, (\widetilde{\mathcal{F}}_t), \wt{\mathbb{P}}, \wt v, \wt w)$
is a probabilistically weak solution to \eqref{Hyper-SNSE}
with the initial condition $z_0$ at time $s_0$,
then $P_{s_0, z_0}:=\wt{\mathbb{P}} \circ (\wt v, \wt w)^{-1}$
is the desired martingale solutions to \eqref{Hyper-SNSE}
satisfying $(M1)$-$(M3)$ in Definition \ref{def-mart-1}.

Conversely,
suppose that $P_{s_0, z_0} \in \mathscr{M}(s_0, z_0, C_{\cdot})$ is
a martingale solution to \eqref{Hyper-SNSE},
then,
$(\wh \Omega, (\mathcal{B}_t(\wh \Omega)), P_{s_0, z_0}, (x(t)), (y(t)))$
is a probabilistically weak solution to \eqref{Hyper-SNSE}
satisfying $\wt{(M1)}$-$\wt{(M3)}$ in Definition \ref{def-weak}.
\end{proof}

\subsection{Global existence}

This subsection mainly contains the proof of  Proposition \ref{Prop-GWP-Mart}
concerning the existence of global martingale solutions
to \eqref{Hyper-SNSE}.

Let us first show the integrability of the stochastic forcing term,
which is important for the subsequent compactness arguments.

\begin{lemma} [Integrability of forcing term] \label{Lem-f-L2}
Let $w$ be the $\mathcal{Q}$-Wiener process \eqref{w-BM}
and set
\begin{equation*}
    f(\cdot, w):=\pt\bar{u}+(-\Delta)^\al \bar{u}+ \bar{u}\cdot\na \bar{u}+
            \div\lc\bar{u}\otimes w+w\otimes\bar{u}\rc +
            \div(w\otimes w)+(-\Delta)^\al w,
\end{equation*}
where $\alpha\in [1,5/4)$, $\ol{u}$ is related to the background profile $\ol{U}$
via the self-similar transform in \eqref{sim-va}.
Then, for any $T\in (0,\infty)$,
\begin{equation}\label{f-reg}
    f|_{[0,T]} \in L^2({\Omega},L^2(0,T; \calv'))\cap L^2({\Omega},L^1(0,T; L^2)).
\end{equation}
\end{lemma}

\begin{proof}
	For simplicity,
	let us write $f = I_1 + I_2 + I_3$,
	where
	$ I_1:= \pt\bar{u}+(-\Delta)^\al \bar{u}+ \bar{u}\cdot\na \bar{u}$,
	$I_2:= \div\lc\bar{u}\otimes w+ w\otimes\bar{u}\rc$
	and $I_3:= \div(w\otimes w)+(-\Delta)^\al w$.
	
	Since $\bar{u}(x, t)={t^{\frac{1}{2\al}-1}} \bar{U}(\xi)$
	with $\xi = t^{-\frac{1}{2\alpha}}x$,
	one has
	\begin{equation*}
		I_1=t^{\frac{1}{2\al}-2}g(\xi),
	\end{equation*}
	where $g$ is smooth and compactly supported with respect to $\xi$.
	The Fourier transformation of $g(\xi)$ is
	\begin{equation*}
		\begin{aligned}
			\F(g(\frac{\cdot}{t^{1/2\al}})) (\eta)
			:&=\int_{\R^3} g(\frac{\cdot}{t^{1/2\al}})e^{-i\eta\cdot  x}\dif x
			= t^{\frac{3}{2\al}}\F(g(\cdot)) (t^{\frac{1}{2\al}}\eta).
		\end{aligned}
	\end{equation*}
	Then we have
	\begin{equation*}
		\|I_1\|_{L^2}
		=t^{\frac{1}{2\al}-2}\cdot t^{\frac{3}{4\al}}\|g(\cdot)\|_{L^2}
		\leq C t^{\frac{5}{4\al}-2},
	\end{equation*}
	and
	\begin{equation*}
		\begin{aligned}
			\norm{I_1}_{\calv'}^2 &=t^{\frac{1}{\al}-4}\int_{\R^3}(1+|\eta|^2)^{-\al}|\F(g(\frac{\cdot}{t^{1/2\al}})) (\eta)|^2\dif \eta \\
			&=t^{\frac{4}{\al}-4}\int_{\R^3}(1+|\eta|^2)^{-\al}|\F(g(\cdot)) (t^{1/{2\al}}\eta)|^2\dif \eta \\
			&\!\!\!\!\!\!\!\stackrel{\wt\eta=t^{\frac{1}{2\al}}\eta}{=}t^{\frac{5}{2\al}-3}\int_{\R^3}(t^{\frac{1}{\al}}+|\wt\eta|^2)^{-\al}|\F(g(\cdot)) (\wt\eta)|^2\dif \wt\eta\\
			&\leq t^{\frac{5}{2\al}-3}\left( \int_{|\wt\eta|\leq 1}|\wt\eta|^{-2\al}|\F(g(\cdot)) (\wt\eta)|^2\dif \eta+ \int_{|\wt\eta|>1}|\F(g(\cdot)) (\wt\eta)|^2\dif \wt\eta\right)\\
			&\leq t^{\frac{5}{2\al}-3} (\norm{|\wt\eta|^{-2\al}}_{L^\frac{6}{5}(|\wt\eta|\leq 1)}\norm{\F(g)}_{L^{12}}^2 +\norm{\F(g)}_{L^2}^2)\leq C t^{\frac{5}{2\al}-3},
		\end{aligned}
	\end{equation*}
	where the last inequality is due to the fact that $\F(g)$ belongs to Schwartz class.
	
	Regarding $I_2$ we have
	\begin{equation*}
		\begin{aligned}
			\norm{I_2}_{L^2}=\norm{\bar{u}\cdot\na w+ w\cdot\na\bar{u}}_{L^2}
			&\leq C\norm{w}_{H^N}(\norm{\bar{u}}_{L^2}+\norm{\na\bar{u}}_{L^2})\\
			&\leq C  \norm{w}_{H^N}(t^{\frac{5}{4\al}-1}+t^{\frac{3}{4\al}-1})\leq C t^{\frac{3}{4\al}-1} \norm{w}_{H^N},
		\end{aligned}
	\end{equation*}
	and
	\begin{align}  \label{J2}
		\norm{I_2}_{\calv'}^2 &=\int_{\R^3}(1+|\eta|^2)^{-\al}|
		\F\lc{\div\lc\bar{u}\otimes w+w\otimes\bar{u}\rc}(\cdot,t)\rc (\eta)|^2\dif \eta \notag  \\
		&=\int_{\R^3} \frac{|\eta|^2}{(1+|\eta|^2)^\al}|\F\lc{\lc\bar{u}\otimes w
			+w\otimes\bar{u}\rc}(\cdot,t)\rc(\eta)|^2\dif \eta   \notag   \\
		&\leq C\norm{\bar{u}\otimes w + w\otimes \bar{u} }_{L^2}^2
		\leq C\norm{w}_{L^\infty}^2 \norm{\bar{u}}_{L^2}^2
		\leq C t^{\frac{5}{2\al}-2} \norm{w}_{H^N}^2.
	\end{align}
	
	Similarly, for the last term $I_3$,
	we have
	\begin{align*}
		\norm{I_3}_{L^2}=\norm{w\cdot\na w+ (-\Delta)^\al w}_{L^2}\leq \norm{w}_{H^{N}}^2
		+\norm{w}_{H^{N}},
	\end{align*}
and
	\begin{align}  \label{J3}
		\norm{I_3}_{\calv'}^2 &=\int_{\R^3}(1+|\eta|^2)^{-\al}|\F\lc\div(w\otimes w)(\cdot,t)+(-\Delta)^\al w(\cdot,t)\rc (\eta)|^2\dif\eta \notag  \\
		&\leq C\int_{\R^3} \frac{|\eta|^2}{(1+|\eta|^2)^\al} |\F(w\otimes w(\cdot,t)) (\eta)|^2
		+\frac{|\eta|^{4\al}}{(1+|\eta|^2)^\al} |\F(w(\cdot,t)) (\eta)|^2  \dif \eta \notag \\
		&\leq C \int_{\R^3} |\F(w\otimes w(\cdot,t)) (\eta)|^2
		+|\eta|^{2\al} |\F(w(\cdot,t)) (\eta)|^2  \dif \eta \notag\\
		&\leq C (\norm{w}_{L^4}^4
		+ \norm{w}_{H^{\al}}^2)
		\leq C (\norm{w}_{H^{N}}^4
		+\norm{w}_{H^{N}}^2).
	\end{align}
	
	Thus,
	combining the above estimates and taking into account the integrability
	$t^{\frac{5}{4\al}-2}$, $t^{\frac{5}{2\alpha} -3}, t^{\frac{3}{4\al}-1} \in L^1_t$ when $\alpha \in [1,5/4)$
	and
	$ \mathbb{E} \sup_{t\leq T} \|w(t)\|_{H^N}^4  <\infty $
	we obtain \eqref{f-reg}.
\end{proof}

The following tightness criteria generalizes Corollary 3.9 of \cite{Brz13}
to the whole time regime.

\begin{lemma} [Tightness criteria] \label{Lem-Tight}
Let $(v_n)_{n\in \bbn}$ be a sequence of  $\calu'$-valued continuous processes
satisfying that for any $T\in (0,\infty)$,
\begin{enumerate}
  \item[$(i)$]
  $ \sup\limits_{n\in \bbn} \mathbb{P}
   (\sup\limits_{t\in [0,T]}\|v_n(s)\|_{\calh}^2
    + \|v_n\|_{L^2(0,T; \calv)}^2 )
   \leq C_T<\infty$;
  \item[$(ii)$]  $(v_n)_{n\in \bbn}$ satisfies that
  for any $T\in (0,\infty)$
  and for any $\ve, \eta>0$,
  there exists $\delta>0$ such that
  \begin{align*}
       \sup\limits_{n\in\bbn}
       \bbp (m_T(v_n, \delta):= \sup\limits_{s,t\in [0,T], |s-t|\leq \delta}
        \|v_n(t) - v_n(s)\|_{\calu'} \geq \eta) \leq \ve,
  \end{align*}
  where $m_T(v_n, \delta)$ is called the modulus of continuity of $v_n$ on $[0,T]$.
\end{enumerate}
Then, $(\bbp\circ v_n^{-1})$ is tight on $\calz_{[0,\infty),\loc}$.
\end{lemma}

\begin{proof}
Fix any $\ve >0$.
For every $k,j\in \bbn$,
let
\begin{align*}
  & K_{k}^1:= \{x\in C([0,\infty); \calu'): \sup_{t\in [0,k]} \|x(t)\|_{\calh}^2
      + \|x\|_{L^2(0, k; \calv)}^2 \leq R^2_k\},  \\
  & K^2_{k,j}:= \{x\in C([0,\infty); \calu'): m_k(x,\delta_{jk}) \leq 1/j\},
\end{align*}
where $R_k^2, \delta_{jk}>0$.
By the properties $(i)$ and $(ii)$,
there exist $R_k$ large enough
and $\delta_{jk}$ small enough
such that
\begin{align*}
   \sup\limits_{n\geq 1} \bbp\circ v_n^{-1} ((K_{k}^1)^c) \leq \frac{\ve}{2^{k+1}},  \ \
   \sup\limits_{n\geq 1} \bbp\circ v_n^{-1} (( K^2_{k,j})^c) \leq \frac{\ve}{2^{j+k+1}}
\end{align*}
for any $k,j\geq 1$.
Then, letting
\begin{align*}
   K:=  \bigcap_{k\in \bbn}
      \big(K^1_{k} \cap \big(\bigcap_{j\in \bbn}K^2_{k,j}\big) \big),
\end{align*}
we have
\begin{align*}
   \sup\limits_{n\geq 1}  \bbp\circ v_n^{-1} (K^c) \leq \ve.
\end{align*}
Moreover,
by the compactness in \cite[Lemma 3.3]{Brz13},
for every $k\geq 1$,
$K^1_{k} \cap \big(\bigcap_{j\in \bbn}K^2_{k,j}\big)$
is relatively compact in $\calz_{[0,k]}$.
Thus,
by diagonal arguments,
$K$ is compact in $\calz_{[0,\infty),\loc}$.
\end{proof}

We are now ready to prove Proposition \ref{Prop-GWP-Mart}.

{\bf Proof of Proposition \ref{Prop-GWP-Mart}.}
By virtue of Proposition \ref{Prop-equi-mart},
we only need to prove the existence of
probabilistically weak solutions to \eqref{Hyper-SNSE}
as in the sense of Definition \ref{def-weak}.
Without loss of generality
we may let $s_0=0$ and $z_0=0$.
The proof proceeds in the following four steps.

{\it Step 1. Approximations of stochastic Navier-Stokes equations.}
As in \cite[Section 2]{Brz13} one can take an orthonormal basis  $\{e_i\} \subseteq \calu$ of $\calh$,
and
define the operator $P_n: \calu' \to \calu$ by
\begin{align*}
   P_n v:= \sum\limits_{i=1}^n
   \ _{\calu'}(v,  e_i)_{\calu} e_i, \ \ v\in \calu'.
\end{align*}
The operator $P_n$ satisfies the following properties:

$(i)$ The restriction of $P_n$ to $\calh$ is an orthonormal projection from $\calh$
onto $\calh_n:= span\{e_1, \cdots, e_n\}$, i.e.,
\begin{align*}
   P_n u = \sum\limits_{i=1}^n (u,e_i)_\calh e_i, \ \ u\in \calh.
\end{align*}

$(ii)$ For any $v_1 \in \calu'$ and $v_2\in \calu$,
\begin{align*}
    _{\calu'}(v_1, P_n v_2)_{\calu}
   = (v_2, P_n v_1)_\calh.
\end{align*}

$(iii)$ For any $u\in \calu$,
\begin{align*}
     P_n v \to v \ \ in\ \calu,
\end{align*}
and
\begin{align} \label{Pnv-v-calv}
     P_n v \to v \ \ in\ \calv.
\end{align}

For every $n\in \bbn$,
consider the approximate equation
\begin{align}  \label{equa-vn}
   \dif v_n = P_n[- (-\Delta)^\alpha v_n - \div (v_n\otimes v_n) + f(t,w_t)] \dif t +  \dif w^{(n)}(t)
\end{align}
with the initial condition $v_n(0) = 0$,
where $w^{(n)}(t) = \sum_{j=1}^n \sqrt{\lambda_j} \epsilon_j \beta_j(t)$,
$t\geq 0$, with $\{(\lambda_j, \epsilon_j)\}$ as in \eqref{w-BM}.

Then,  by the Krylov theory of monotone-type stochastic differential equations (cf. \cite{LR15}, Section 3),
there exists a unique solution $v_n\in C([0,T]; \calh_n)$ $\mathbb{P}$-a.s.,
satisfying that for any $0\leq t<T<\infty$,
\begin{align} \label{Energy-ineq-SDE}
		\bbe \|v_n(t)\|_\calh^2 + 2 \bbe \int_0^t \|(-\Delta)^{\frac \alpha2} v_s(s)\|_\calh^2 \dif s
		\leq 2 \bbe \int_0^t (f(s,w(s)), v_n) \dif s + \sum_{j=1}^{n}\lambda_j\|\epsilon_j\|^2_{\calh} t.
\end{align}

In particular, one has the uniform energy bound
\begin{align}   \label{vn-CtL2-L2V-bdd}
    \sup\limits_{n\geq 1} \bbe (  \|v_n \|_{C([0,T];\calh)}^2
     +  \|v_n \|^2_{L^2(0,T; \calv)} ) <\infty.
\end{align}
(See Appendix \ref{App-C} for the detailed proof.)

{\it Step 2. Tightness.}
We claim that $(\bbp\circ v_n^{-1})$ is tight on $\calz_{[0,\infty),\loc}$.
In view of \eqref{vn-CtL2-L2V-bdd},
the tightness criteria in Lemma \ref{Lem-Tight}
and \cite[Lemma 3.8]{Brz13},
it suffices to prove that for any $T\in (0,\infty)$,
$(v_n)$ satisfies the Aldous condition in $\calu'$,
that is,
for any $\ve>0$ and $\eta>0$,
there exists $\delta>0$ such that for every sequence $\{\tau_n\}_{n\in \mathbb{N}}$
of stopping times with $0 \leq \tau_n, \tau_n+\theta \leq T$ one has
\begin{align} \label{vn-Aldous}
   \sup\limits_{n\in \bbn} \sup\limits_{0\leq \theta\leq \delta}
   \bbp ( \|v_n(\tau_n+\theta) - v_n(\tau_n) \|_{\calu'} \geq \eta) \leq \ve.
\end{align}

In order to prove \eqref{vn-Aldous},
we infer from \eqref{equa-vn} that
for any $t \geq 0$,
\begin{align*}
   v_n(t) =&  - \int_0^t P_n (-\Delta)^\alpha v_n(s) \dif s
              - \int_0^t P_n (\div (v_n(s)\otimes v_n(s)) )  \dif s \\
          & + \int_0^t P_n f(s,w(s)) \dif s + w^{(n)}(t)
          =: J_1^n + J_2^n + J_3^n + J_4^n.
\end{align*}
Since $(-\Delta)^\alpha \in L(\calv, \calv')$
and $\calv' \hookrightarrow \calu'$,
one has
\begin{align*}
   \bbe \|J_1^n(\tau_n+\theta) - J_1^n(\tau_n)\|_{\calu'}
   \leq C \bbe \int_{\tau_n}^{\tau_n+\theta}  \|v_n\|_{\calv} \dif s
   \leq C\theta^\frac 12 (\bbe \|v_n\|^2_{L^2(0,T; \calv)})^{\frac{1}{2}}.
\end{align*}
Moreover,
by the embedding $\calu \hookrightarrow W^{1,\infty}$,
for any $u\in \calu$,
\begin{align*}
   |\<\div (v\otimes v), u\>|
   = |\< v\otimes v, \na u\>|
   \leq C \|v\|_{\calh}^2 \|\na u\|_{L^\infty}
   \leq C \|v\|_{\calh}^2 \|u\|_{\calu}
\end{align*}
with $C$ independent of $n$,
which yields that
\begin{align*}
   \|\div (v\otimes v)\|_{\calu'}
   \leq C \|v\|_{\calh}^2.
\end{align*}
It then follows that
\begin{align*}
    \bbe \|J_2^n(\tau_n+\theta) - J_2^n(\tau_n)\|_{\calu'}
    \leq C \bbe \int_{\tau_n}^{\tau_n+\theta} \|v_n\|_{\calh}^2 \dif s
    \leq C \theta \bbe \|v_n\|^2_{C([0,T]; \calh)} .
\end{align*}
Regarding $J_3$,
Lemma \ref{Lem-f-L2}
together with the embedding $\calv' \hookrightarrow \calu'$
yield that
\begin{align*}
   \bbe \|J_3^n(\tau_n+\theta) - J_3^n(\tau_n)\|_{\calu'}
   \leq  C \bbe \int_{\tau_n}^{\tau_n+\theta}   \|f(s,w(s))\|_{\calv'} \dif s
   \leq  C \theta^\frac 12 (\bbe \|f\|^2_{L^2(0,T; \calv')})^\frac{1}{2 }.
\end{align*}
For the last term,
since $\calw \hookrightarrow \calu'$
and $(\epsilon_j)$ is an orthonormal basis of $\calw$,
one has
\begin{align*}
    \bbe \|J_4^n(\tau_n+\theta) - J_4^n(\tau_n)\|_{\calu'}
    \leq& (\bbe \|w^{(n)}(\tau_n+\theta) - w^{(n)}(\tau_n)\|^2_{\calw})^{\frac 12} \\
    =& (\sum\limits_{j=1}^n \lambda_j \mathbb{E} |\beta_j(\tau_n+\theta) - \beta_j(\tau_n)|^2)^{\frac 12}
    \leq  \|\lambda\|_{\ell^1}^{\frac 12} \theta^{\frac 12},
\end{align*}
where $C$ is independent of $n$.
Thus, combining these estimates altogether we
obtain that for any $\eta>0$ and $\theta\in [0,\delta]$,
\begin{align*}
   \bbp ( \|v_n(\tau_n+\theta)- v_n(\tau_n)\|_{\calu'} \geq \eta)
   \leq \eta^{-1} \sum\limits_{i=1}^4 \bbe \|J_i^n(\tau_n+\theta) - J_i^n(\tau_n)\|_{\calu'}
   \leq C \frac{\delta^{\frac 12}}{\eta}
   \leq \ve
\end{align*}
provided $\delta \leq C^{-2} (\varepsilon \eta)^2$.
This yields \eqref{vn-Aldous}, as claimed.

{\it Step 3. Skorokhod representation.}
Since $(\bbp\circ v_n^{-1})$ is tight on $\calz_{[0,\infty),\loc}$,
so is $(\bbp\circ (v_n, w)^{-1})$
on  $\calz_{[0,\infty),\loc} \times \scry_{[0,\infty),\loc}$,
we can apply the Jakubowski's version of Skorokhod representation theorem (cf. \cite[Theorem 3.11]{Brz13})
to derive that
there exist a probability space $(\wt \Omega, \wt{\bbf}, \wt \bbp)$
and a subsequence (still denoted by $\{n\}$) of random variables $(\wt v_{n}, \wt w_{n})$
and $(\wt v, \wt w)$ in $\calz_{[0,\infty),\loc} \times \scry_{[0,\infty),\loc}$,
such that
\begin{align}  \label{wtvn-wtWn-wtv-wtW-law}
   \wt \bbp \circ (\wt v_{n}, \wt w_{n})^{-1} = \bbp \circ (v_n, w)^{-1},\ \ \forall n\geq 1,
\end{align}
and
\begin{align} \label{wtvn-wtWn-limit}
   (\wt v_{n}, \wt w_{n}) \to (\wt v, \wt w)  \ \ {\rm in}\ \calz_{[0,\infty),\loc}\times \scry_{[0,\infty),\loc}, \ \ {\rm as}\ n \to \infty,\ \wt{\bbp}-a.s.
\end{align}

Let $\wt{\calf}_t := \sigma(\wt v_n(s), \wt w_n(s), n\geq 1, s\leq t)$,
$t\geq 0$.
Then,
$\wt w$ is an $(\wt{\calf}_t)$-adapted $\mathcal{Q}$-Wiener process starting from $0$ at time zero under $\wt \bbp$.
Moreover, estimate  \eqref{vn-CtL2-L2V-bdd} holds
for $(\wt v_{n})$ under the expectation $\bbe^{\wt \bbp}$, $n\geq 1$.
In particular,
along a further subsequence if necessary,
\begin{align}
    &  \wt v_n \stackrel{w^*}{\rightharpoonup} \wt v\ \ {\rm in}\ L^2(\wt \Omega; L^\infty(0,T; \calh)), \label{vn-v-CtH} \\
    &  \wt v_n \rightharpoonup  \wt v\ \ {\rm in}\ L^2(\wt\Omega; L^2(0,T; \calv)). \label{vn-v-L2V}
\end{align}

Thus, it follows that
\begin{align}  \label{wtv-wtvn-calh}
     \|\wt v\|_{L^2(\wt \Omega; L^\infty(0,T; \calh))}
     \leq \liminf\limits_{n\to \infty} \|\wt v_n\|_{L^2(\wt \Omega; L^\infty(0,T; \calh))}
     \leq C_T,
\end{align}
and
\begin{align}   \label{wtv-wtvn-calv}
     \|\wt v\|_{L^2(\wt \Omega; L^2(0,T; \calv))}
     \leq \liminf\limits_{n\to \infty} \|\wt v_n\|_{L^2(\wt \Omega; L^2(0,T; \calv))}
     \leq C_T.
\end{align}

It also follows from equation \eqref{equa-vn}
and the identical joint distributions \eqref{wtvn-wtWn-wtv-wtW-law} that
for any $t\geq 0$ and any test function $\psi \in \calc_0^\infty$,
\begin{align}  \label{equa-wtvn-psi}
   (\wt v_n(t), \psi)
   =& \int_0^t (-P_n(-\Delta)^\alpha \wt v_n
       -  P_n \div (\wt v_n \otimes \wt v_n)
      +  P_n f(\wt w_n), \psi) \dif s
     + (\ol{P}_n \wt w_n(t), \psi)  \notag \\
   =&   - \int_0^t ((-\Delta)^\alpha \wt v_n, P_n\psi) \dif s
       -  \int_0^t (\div (\wt v_n \otimes \wt v_n) , P_n \psi) \dif s \notag \\
    &  + \int_0^t (f(\wt w_n), P_n\psi) \dif s
      + (\ol{P}_n  \wt w_n(t),  \psi),\ \ \wt{\bbp}-a.s.
\end{align}
where we abuse the notation
$(\cdot, \cdot)$ for the inner product in $\calh$
and also for the dual pair between $\calv'$ and $\calv$,
$\ol{P}_n$ is the orthogonal projection from $\calw$ to
$\calw_n:=span\{\epsilon_1, \cdots, \epsilon_n\}$.
Note that for any $w\in \calw$,
\begin{align}  \label{Pnw-w}
     \ol{P}_n w \to w\ \ {\rm in}\ \calw.
\end{align}

{\it Step 4. Passing to the limit.}
Since by \eqref{wtvn-wtWn-limit},
$\wt v_n \rightarrow \wt v\ \text{in} \ L_w^2(0,T;\calv)$, in particular,
$\wt v\in L^2(0,T;\calv)$, and the sequence $(\wt v_n)$ is bounded in $L^2(0,T;\calv)$, we have
\begin{align} \label{Deltavn-limit}
	&|\int_0^t ((-\Delta)^\alpha \wt v_n(s), P_n\psi) - ((-\Delta)^\alpha \wt v(s), \psi) \dif s|  \notag  \\
 \leq  & |\int_0^t ((-\Delta)^\alpha (\wt v_n-\wt v), \psi) \dif s|+ |\int_0^t ((-\Delta)^\alpha \wt v_n(s), P_n\psi-\psi) \dif s|  \notag  \\
 \leq&  \|\wt v_n - v\|_{L^2(0,t; \calv)}\|\psi\|_{L^2(0,t; \calv)}
         + \|P_n\psi - \psi\|_{\calv} \sup_n \|\wt v_n\|_{L^2(0,t; \calv)} T^\frac 12
         \to 0  ,\ \ \wt{\bbp}-a.s.
\end{align}
where we also used the convergence \eqref{Pnv-v-calv}.

Regarding the nonlinear term, again by the boundedness of $(\wt v_n)$ in $L^2(0,T;\calv)$, we have
\begin{align*}
	  |\int_0^t (\div (\wt v_n(s) \otimes \wt v_n(s)) , P_n\psi-\psi) \dif s|
	 =& |\int_0^t (\wt v_n(s) \otimes \wt v_n(s) , \nabla (P_n\psi-\psi)) \dif s|\\
	 \leq& \int_0^t \norm{\wt v_n}_{L^4}^2 \norm{P_n\psi-\psi}_\calv \dif s  \\
     \leq& C\norm{P_n\psi-\psi}_\calv  \sup\limits_{n\geq 1} \int_0^t \norm{\wt v_n}_{\calv}^2 \dif s \to 0,\ \ \wt{\bbp}-a.s.
\end{align*}
where the last inequality is due to the Sobolev embedding $\calv \hookrightarrow L^4$.
Moreover, using the convergence $\wt v_n \to \wt v$ in $ L^2(0,T; \calh_R)$,
due to \eqref{wtvn-wtWn-limit},
where $R>0$ is such that ${\rm supp} \psi \subseteq B_R$,
we have
\begin{align}\label{div-limit}
     & | \int_0^t (\div (\wt v_n(s) \otimes \wt v_n(s)) - \div (\wt v(s) \otimes \wt v(s)), \psi) \dif s |  \notag \\
    =& | \int_0^t ( (\wt v_n(s) - \wt v(s) )\otimes \wt v_n(s) + \wt v(s) \otimes (\wt v_n(s)-\wt v(s)), \nabla\psi) \dif s | \notag \\
   \leq& \|\wt v_n - \wt v\|_{L^2(0,T; \calh_R)}
        (\sup_{n\geq 1} \|\wt v_n \|_{L^2(0,T; \calh_R)}
         + \|\wt v\|_{L^2(0,T; \calh_R)}) \|\nabla\psi\|_{L^\infty}
   \to 0,\ \wt{\bbp}-a.s.
\end{align}
The above two estimates thus lead to the convergence
\begin{align}  \label{divvn-divv-limit}
	&  \int_0^t (\div (\wt v_n(s) \otimes \wt v_n(s)) , P_n\psi) \dif s
	\to  \int_0^t (\div (\wt v(s) \otimes \wt v(s)) , \psi) \dif s, \ \ \wt\bbp-a.s.
\end{align}

Since by \eqref{wtvn-wtWn-limit}, $\wt w_n \to \wt w$ in $\scry_{[0,\infty),\loc}$,
using \eqref{Pnw-w} one also has
\begin{align}   \label{Pnwn-w-limit}
      (\ol{P}_n \wt w_n,  \psi) \to (\wt w, \psi), \ \wt\bbp-a.s.
\end{align}

Regarding the stochastic forcing term,
computing as in \eqref{J2} yields
\begin{align*}
    \| \div (\bar u \otimes \wt w_n) - \div (\bar u \otimes \wt w)\|_{\calv'}
   \lesssim  s^{\frac{5}{4\al}-1} \norm{\wt w_n-\wt w}_{H^N}.
\end{align*}
Taking into account \eqref{wtvn-wtWn-limit} and
the fact that $s^{\frac{5}{4\al}-1} \in L^1(0,T)$ when $\alpha \in [1,5/4)$
we obtain that $\wt{\bbp}-a.s.$ as $n\to \infty$,
\begin{align*}
     \int_0^t \left(\div (\bar u \otimes \wt w_n) - \div (\bar u \otimes \wt w), \psi \right)\dif s\leq \int_0^{t}\| \div (\bar u \otimes \wt w_n) - \div (\bar u \otimes \wt w)\|_{\calv'} \norm{\psi}_{\calv} \dif s
     \to 0,
\end{align*}
Moreover, using \eqref{wtvn-wtWn-limit} we get
$\wt{\bbp}-a.s.$ as $n\to \infty$,
\begin{align*}
    &|\int_0^t \left( \div (\wt w_n \otimes \wt w_n- \wt w \otimes \wt w),\psi\right)  \dif s|\\
    \leq& \|\nabla\psi\|_{L^\infty}
    \int_0^t (\sup_{n\in\mathbb{N}}\norm{\wt w_n}_{H^N}+\norm{\wt w}_{H^N})\norm{\wt w_n-\wt w}_{H^N} \dif s\\
    \leq& \|\nabla\psi\|_{L^\infty}
    t (\sup_{n\in\mathbb{N}}\norm{\wt w_n}_{C([0,T];H^N)}+\norm{\wt w}_{C([0,T];H^N)})\norm{\wt w_n-\wt w}_{C([0,T];H^N)}
    \to 0.
\end{align*}
Hence, it follows that
\begin{align*}
   \int_0^t (f(s,\wt w_n(s)),\psi) \dif s
   \to \int_0^t (f(s,\wt w(s)), \psi) \dif s,\ \ \wt{\bbp}-a.s.
\end{align*}
Taking into account \eqref{J2}-\eqref{J3} and $\wt w_n \to \wt w$ in $\scry_{[0,\infty),\loc}$,
we have the uniform boundedness of $\norm{f(\cdot, \wt w_n)}_{L^2(0,T;{\calv'})}$,
$n\geq 1$,
which along with \eqref{Pnv-v-calv} leads to
\begin{align}  \label{fwn-fw-limit}
   \int_0^t (f(s,\wt w_n(s)),P_n\psi) \dif s
   \to \int_0^t (f(s,\wt w_s), \psi) \dif s,\ \ \wt{\bbp}-a.s.
\end{align}

Therefore, combining \eqref{Deltavn-limit}, \eqref{divvn-divv-limit}, \eqref{Pnwn-w-limit},
\eqref{fwn-fw-limit} and passing to the limit in \eqref{equa-wtvn-psi}
we obtain that $\wt \bbp$-a.s. for any $t\geq 0$,
\begin{align}\label{limit-wtv}
   (\wt v(t), \psi)
   =& \int_0^t ( -  (-\Delta)^\alpha \wt v(s)
       -   \div (\wt v(s) \otimes \wt v(s))
      +   f(s, \wt w(s)), \psi) \dif s
      + (\wt w(t), \psi).
\end{align}

Below, we prove the energy inequality for $\wt v$.
Note that,
by \eqref{wtvn-wtWn-limit},  $\wt \bbp$-a.s for $t\in[0,T]$,

$$\wt v_n(t)\stackrel{w}{\rightharpoonup}\wt v(t) \quad\text{in}\,\, \calh,$$
which yields
\begin{equation} \label{Evnt-Evt}
	\bbe^{\wt \bbp} \|\wt v(t)\|_{\calh}^2
	\leq \bbe^{\wt \bbp} \liminf_{n\to \infty} \|\wt v_n(t)\|_{\calh}^2
	\leq \liminf_{n\to \infty} \bbe^{\wt \bbp}  \|\wt v_n(t)\|_{\calh}^2.
\end{equation}

Moreover, since $\wt v_n\to\wt v$ in $L^2_w(0,T;\calv)$
and $f(\wt w_n)\to f(\wt w)$ in $L^2(0,T;\calv')$, $\wt{\bbp}$-a.s.,
we see that
\begin{equation*}
	\int_0^t ( f(\wt w_n), \wt v_n) \dif s \to \int_0^t (f(\wt w), \wt v) \dif s,\ \ \wt{\bbp}-a.s.
\end{equation*}
It also holds the uniform integrability, that is, for some $p>0$,
\begin{equation*}
	\sup_{n\in \bbn} \bbe^{\wt\bbp}|\int_0^t (f(\wt w_n), \wt v_n)\dif s|^{1+p}<\infty.
\end{equation*}
To this end,
by \eqref{vn-CtL2-L2V-bdd}, \eqref{wtvn-wtWn-wtv-wtW-law} and Lemma \ref{Lem-f-L2},
we derive that for $0<p<\frac{5-4\al}{8\al-5}$,
\begin{equation*}
	\begin{aligned}
		\sup_{n\in \bbn}\bbe^{\wt\bbp}|\int_0^t (f(\wt w_n),\wt v_n) \dif s|^{1+p}
		&\les \sup_{n\in \bbn}\bbe^{\wt\bbp} \int_0^t  \norm{f(\wt w_n)}_{\calv'}^{p+1} \norm{\wt v_n}_\calv^{p+1} \dif s\\
		&=\sup_{n\in \bbn}\bbe^{\bbp} \int_0^t  \norm{f( w)}_{\calv'}^{p+1} \norm{v_n}_\calv^{p+1} \dif s\\
		&\les \sup_{n\in \bbn} ( \bbe^{\bbp} \int_0^t \norm{f( w)}_{\calv'}^{\frac{2(p+1)}{1-p}}\dif s)^\frac{1-p}{2}
                (\bbe^{\bbp} \int_0^t \norm{v_n}_\calv^{2} \dif s)^{\frac{1+p}{2}}
         <\infty.
	\end{aligned}
\end{equation*}
Hence,
we get
\begin{align}  \label{Efwn-Efw-limit}
  \bbe^{\wt\bbp} \int_0^t (f(\wt w_n), \wt v_n) \dif s \to  \bbe^{\wt\bbp} \int_0^t (f(\wt w), \wt v) \dif s.
\end{align}

Thus, using \eqref{Energy-ineq-SDE}, \eqref{vn-v-L2V}, \eqref{Evnt-Evt} and \eqref{Efwn-Efw-limit}
we obtain the energy inequality
\begin{align}  \label{energy-ineq-wtv}
		\bbe^{\wt \bbp} \|\wt v(t)\|_{\calh}^2 + 2 \bbe^{\wt \bbp} \int_0^t \|(-\Delta)^{\frac \alpha 2} \wt v(s)\|_{\calh}^2 \dif s
		\leq& \liminf\limits_{n\to \infty} \left(\bbe^{\wt \bbp} \|\wt v_n(t)\|_{\calh}^2
              + 2 \bbe^{\wt \bbp} \int_0^t \|(-\Delta)^{\frac \alpha 2} \wt v_n(s)\|_{\calh}^2 \dif s\right)  \notag \\
		\leq& 2 \liminf\limits_{n\to \infty}  \bbe^{\wt \bbp} \int_0^t (f(\wt w_n),\wt v_n) \dif s + \sum_{j=1}^{\infty}\lambda_j\|\epsilon_j\|^2_{\calh} t  \notag \\
		\leq&  2 \bbe^{\wt \bbp} \int_0^t ( f(\wt w), \wt v) \dif s + \sum_{j=1}^{\infty}\lambda_j\|\epsilon_j\|^2_{\calh} t.
\end{align}

Therefore, taking into account \eqref{wtv-wtvn-calh}, \eqref{wtv-wtvn-calv}, \eqref{limit-wtv} and \eqref{energy-ineq-wtv}
we conclude that
$(\wt \Omega, \wt{\bbf}, \wt \bbp, \wt v, \wt w)$ is
a probabilistically weak solution to \eqref{Hyper-SNSE}.
The proof is complete.
\hfill $\square$

We note that the above arguments also give the
compactness of martingale solutions $\mathscr{M}(s_0, z_0, C_{\cdot})$
in $\mathscr{P}(\widehat{\Omega})$.
Specifically, we have

\begin{corollary} [Compactness of $\mathscr{M}(s_0, z_0, C_{\cdot})$] \label{Cor-Mart-Comp}
For any $s_0\geq 0$, $z_0\in \calh \times \calw$
and any positive increasing function $t\mapsto C_t$,
$\mathscr{M}(s_0, z_0, C_{\cdot})$ is compact in $\mathscr{P}(\wh \Omega)$.
\end{corollary}

\subsection{Stability}

In this subsection, we mainly prove the stability
of martingale solutions
which is important in the next gluing procedure.

\begin{proposition} [Stability of martingale solutions]   \label{Prop-Stab-Mart}
Let $z_n = (x_{0,n}, y_{0,n}), z_0 = (x_0, y_0) \in \calh \times \calw$,
$P_n \in \mathscr{M}(s_n, z_n, C_{\cdot})$
for a common positive increasing function $t\mapsto C_t$, $n\geq 1$.
Assume that $(s_n, z_n) \to (s_0, z_0)$ in $[0,\infty) \times \calh \times \calw$.
Then, there exists a subsequence $\{n_k\}$ and
$P\in \mathscr{M}(s_0, z_0, C_{\cdot})$
such that $P_{n_k}$ convergence weakly to $P$
as $n_k\to \infty$.
\end{proposition}

In order to prove Proposition \ref{Prop-Stab-Mart},
let us first show the tightness of $(P_n\circ x^{-1})$
and $(P_n\circ y^{-1})$, respectively,
which are the contents of Lemmas \ref{Lem-Pnx-tight} and \ref{Lem-Pny-tight} below.

\begin{lemma} [Tightness of $(P_n\circ x^{-1})$]  \label{Lem-Pnx-tight}
Assume the conditions in Proposition \ref{Prop-Stab-Mart} to hold.
Then,
	$(P_n\circ x^{-1})$ is tight on $\wh\calz_{[0,\infty),\loc}$.
\end{lemma}

\begin{proof}
First note that,
by Definition \ref{def-mart-1} $(M2)$
and $x_{0,n}\to x_n$ in $\calh$,
for every $k\geq 1$,
\begin{align}  \label{Pn-x-CL2-L2V-bdd}
  & \sup\limits_{n\in \bbn} \bbe^{P_n} (\sup\limits_{t\in [0,k]}\|x(t)\|_{\calh}^2
  + \|x\|_{L^2(s_n,k;\calv)} ) <\infty,
\end{align}
and  for any $\ve, \eta>0$,
there exists $\delta>0$ such that
\begin{align}  \label{Pn-Aldous}
    \sup\limits_{n\in \bbn} P_n
   (m_k(x, \delta) \geq \eta)
  \leq \ve,
\end{align}
where $m_k(x, \delta)$ is the modulus of continuity of $x$ on $[0,k]$ as in Lemma \ref{Lem-Tight}.

Let $\Omega_n:= \{x\in C([0,\infty);\calu'): x|_{[0,s_n]} \equiv x_{0,n} \}$,
\begin{align*}
  & K_{n,k}^1:= \{x\in \Omega_n: \sup\limits_{t\in [0,k]} \|x(t)\|_{\calh}^2 + \|x\|_{L^2(s_n, k; \calv)}^2 \leq R^2_k\},  \\
  & K^2_{n,k,j}:= \{x\in \Omega_n: m_k(x,\delta_{jk}) \leq 1/j\},
\end{align*}
where $n,k,j\in \bbn$, $R_k, \delta_{jk}>0$.
By \eqref{Pn-x-CL2-L2V-bdd} and \eqref{Pn-Aldous},
there exist $R_k$ large enough
and $\delta_{jk}$ small enough
such that for any $n,k,j\geq 1$,
\begin{align*}
   \sup\limits_{n\geq 1} P_n ((K_{n,k}^1)^c) \leq \frac{\ve}{2^{k+1}},  \ \
   \sup\limits_{n\geq 1} P_n (( K^2_{n,k,j})^c) \leq \frac{\ve}{2^{j+k+1}}.
\end{align*}
Then, setting
\begin{align*}
   K:= \bigcup_{n\in \bbn} \bigcap_{k\in \bbn}
      \big(K^1_{n,k} \cap \big(\bigcap_{j\in \bbn} K^2_{n,k,j}\big) \big)
\end{align*}
one has
\begin{align*}
   \sup\limits_{n\geq 1}  P_n (K^c) \leq \ve.
\end{align*}

Below we prove that $K$ is a compact set in $\wh\calz_{[0,\infty),\loc}$.
Take any sequence $(x_m)\subseteq K$,
it suffices to prove that
for every $k\in \bbn$,
$(x_m|_{[0,k]})$ has a subsequence convergent in $\wh\calz_{[0,k]}$.
Then,
an application of diagonal arguments permits to extract
a uniform subsequence of $(x_m)$ convergent in $\wh\calz_{[0,\infty),\loc}$.

To this end,
let us first see that if there exist infinitely many $x_m$ in a set $\Omega_M$ for some $M\geq 1$,
then by the compactness result in \cite[Lemma 3.3]{Brz13},
there exists a subsequence of $(x_m)$ convergent in $\wh\calz_{[0,k]}$.

Below, without loss of generality we may assume that $x_m\in \Omega_m$, $m\geq 1$.
Note that,
\begin{align}
    & \sup\limits_{m\geq 1} ( \sup\limits_{t\in [0,k]}\|x_m(t)\|_{\calh}
    +  \|x_m\|_{L^2(s_m, k; \calv)}) \leq R_k, \label{xm-CtL2-L2V-bdd}  \\
    & \sup\limits_{m\geq 1} m_k(x_m, \delta_{jk}) \leq 1/j,\ \  \forall j\geq 1.   \label{xm-Aldoux-bdd}
\end{align}
For simplicity,
we write $x_m$ for $x_m|_{[0,k]}$ below.
In particular,
there exists a subsequence, still denoted by $(x_m)$,
such that
\begin{align}
   & x_m \stackrel{w^*}{\rightharpoonup} x^*\ \ in\ L^\infty(0,k; \calh),   \label{xm-CL2-limit}  \\
   & x_m I_{[s_m,k]} \stackrel{w}{\rightharpoonup}  x^* I_{[s_0,k]} \ \ {\rm in}\ L^2(0, k;\calv). \label{xm-L2V-limit}
\end{align}

Since for any $t<s_0$, $x_m|_{[0,t]} \equiv x_{0,m}$ for $m$ large enough,
taking into account \eqref{xm-CL2-limit} and the convergence
\begin{align}\label{x0m-x0-limit}
   x_{0,m} \to x_0\ \ {\rm in}\ \calh,
\end{align}
we have
\begin{align*}
     x^*|_{[0,s_0)} \equiv x_0,\ \ dt\otimes dx-a.e.
\end{align*}
Hence,
we may modify $x^*$ on a null set such that
$x^*|_{[0,s_0]} \equiv x_0$
and still preserve the limits \eqref{xm-CL2-limit} and \eqref{xm-L2V-limit}.

For every $m\geq 1$, by \eqref{xm-CtL2-L2V-bdd},
$x_m(t) \in \calv$ for a.e. $t\in (s_m, k)$.
Taking into account $s_m\to s_0$,
we can take a uniform null set $\caln$ of $[s_0, k]$
such that for any $t\in \mathcal{N}^c$,
$x_{m(t)}(t) \in \calv$ for $m(t)$ large enough,
which, via the compact embedding $\calv \hookrightarrow \calu'$,
yields that there exists a subsequence of $(x_{m(t)}(t))$  (depending on $t$)
convergent in $\calu'$.

Moreover, fix $\delta_0(>0)$ small enough.
For any $t\in [s_0-\delta_0, s_0)$,
since $x_m(t) = x_{0,m}$ for $m$ large enough,
by \eqref{x0m-x0-limit}
and the embedding $\calh \hookrightarrow \calu'$,
one has $x_m(t)$ converges in $\calu'$.

Hence, taking a dense subset $(t_i)$ of $[s_0-\delta_0, k]$,
using diagonal arguments
one can extract a uniform subsequence,
still denoted by $(x_m)$,
such that for every $i\geq 1$, $(x_m(t_i))$ converges in $\calu'$.

Then,  the Aldous condition \eqref{xm-Aldoux-bdd} permits to give that
\begin{align*}
   x_m|_{[s_0-\delta_0,k]} \to x^*|_{[s_0-\delta_0,k]} \ \ {\rm in}\ C([s_0-\delta_0, k]; \calu').
\end{align*}
Actually,
for any $\ve'>0$,
take $j$ large enough such that $1/j\leq \ve' /3$.
The compact interval $[s_0-\delta_0, k]$
can be covered by finitely many open balls
$B(t_i, \delta_j)$, $1\leq i\leq M$,
for some $M\geq 1$.
Then, for any $t \in [s_0-\delta_0, k]$,
there exists $1\leq i\leq M$ such that $t\in B(t_i, \delta_j)$,
and
\begin{align*}
  \|x_n(t) - x_m(t)\|_{\calu'}
  \leq& \|x_n(t) - x_n(t_i)\|_{\calu'} + \|x_m(t) - x_m(t_i)\|_{\calu'}
         + \|x_n(t_i) - x_m(t_i) \|_{\calu'} \\
  \leq& m_k(x_n, \delta_{jk}) + m_k(x_m, \delta_{jk})
        + \|x_n(t_i) - x_m(t_i)\|_{\calu'},
\end{align*}
which along with \eqref{xm-Aldoux-bdd}
and $1/j\leq \ve'/3$
yields that
\begin{align*}
   \|x_n - x_m\|_{C([s_0-\delta_0, k]; \calu')}
   \leq \frac{2\ve'}{3} + \sum\limits_{i=1}^M \|x_n(t_i) - x_m(t_i)\|_{\calu'}
   \leq \ve'
\end{align*}
for $n,m$ large enough.
This yields that $\{x_m|_{[s_0-\delta_0, k]}\}$ is a Cauchy sequence in
$C([s_0-\delta_0,k]; \calu')$.
Taking into account
\begin{align}  \label{xmsm-xm}
   x_m|_{[0,s_0-\delta_0]} \equiv x_{0,m}
   \to x^*|_{[0,s_0-\delta_0]} \equiv x_0 \ \ {\rm in}\ \calh
\end{align}
and $\calh \hookrightarrow \calu'$
we obtain
\begin{align} \label{xm-CU'-limit}
   x_m \to x^* \ \ {\rm in}\ C([0,k]; \calu').
\end{align}

Below we show that $\{x_m\}$ converges in $L^2(0, k; \calh_{\loc})$,
or equivalently, for every $R\geq 1$,
\begin{align}  \label{xm-L2Hloc-limit}
   x_m \to x^* \ \  {\rm in}\ L^2(0, k; \calh_R), \ \ {\rm as}\ m\to \infty.
\end{align}

To this end,
by \eqref{xmsm-xm},
\begin{align}  \label{xm-x-L2-0s0}
   x_m \to x^*\ \ {\rm in}\ L^2(0,s_0-\delta_0; \calh_R).
\end{align}
Hence,
we focus on the temporal regime $[s_0-\delta_0, k]$ below.

Take $\delta_0$ possibly smaller such that
$ \delta_0^{1/2} \leq  \ve' /(10 R_k)$.
Then,
\begin{align}  \label{xm-x-L2-s0s*}
       & \sup\limits_{m\geq 1} \|x_m-x^*\|_{L^2(s_0-\delta_0, s_0+\delta_0; \calh_R)}  \notag \\
   \leq& (\sup\limits_{m\geq 1} \|x_m\|_{L^\infty(s_0-\delta_0, s_0+\delta_0; \calh)}
          + \|x^*\|_{L^\infty(s_0-\delta_0, s_0+\delta_0; \calh)})\sqrt{2\delta_0}  \notag \\
   \leq& 4 R_k \delta_0^{\frac 12} \leq \frac{\ve'}{2}.
\end{align}

Moreover,
for any $t\in [s_0+\delta_0, k]$,
since the embedding $\calv \hookrightarrow \calh_R$ is compact,
and the embedding $\calh_{R} \hookrightarrow \calu'$ is continuous, using the Lions Lemma we get
\begin{align*}
   \|x_m(t) - x^*(t)\|^2_{\calh_R}
   \leq& \frac{\ve'}{10 R_k^2} \|x_m(t) - x^*(t)\|_{\calv}^2
        + C_{\ve', R_k} \|x_m(t) - x^*(t)\|^2_{\calu'}
\end{align*}
for some constant $C_{\ve', R_k}$ independent of $t$.
This along with \eqref{xm-CU'-limit} yields that
\begin{align} \label{xm-x-L2-s*k}
    \|x_m - x^*\|^2_{L^2(s_0+\delta_0, k; \calh_R)}
    \leq&  \frac{\ve'}{10 R_k^2} \|x_m - x^*\|_{L^2(s_0+\delta_0, k; \calv)}^2  \notag  \\
        &  + C_{\ve', R_k} (k-s_0-\delta_0) \|x_m - x^*\|^2_{C([s_0+\delta_0,k];\calu')} \notag  \\
    \leq& \frac{\ve'}{5} +  C_{\ve', R_k} k \|x_m - x^*\|^2_{C([s_0+\delta_0,k];\calu')}
    \leq \frac{\ve'}{2}
\end{align}
for $m$ large enough.

Thus, combining  \eqref{xm-x-L2-s0s*}  and \eqref{xm-x-L2-s*k}
we lead to
\begin{align*}
    \|x_m - x^*\|^2_{L^2(s_0-\delta_0, k; \calh_{R})}
    \leq \ve'
\end{align*}
for $m$ large enough.
This together with \eqref{xm-x-L2-0s0} yield
\eqref{xm-L2Hloc-limit}.

At last,
in view of \cite[Lemma 3.2]{Brz13},
the uniform boundedness \eqref{xm-CtL2-L2V-bdd}
together with
the convergence \eqref{xm-CU'-limit}
yield that
\begin{align}   \label{xm-CHw-limit}
   x_m \to x \ \ {\rm in}\ C([0, k]; \mathbb{B}_w), \ \ {\rm as}\ m\to \infty,
\end{align}
where $ \mathbb{B}_w$ is a metrizable subspace of $C([0,k]; \calh_w)$,
with the metric comparable with the weak topology
(see \cite[p.1641]{Brz13} for more details).
Hence, $(x_m)$ is also convergent in $C([0,k]; \calh_w)$.

Therefore,
it follows from \eqref{xm-CU'-limit}, \eqref{xm-L2Hloc-limit}
and \eqref{xm-CHw-limit} that $K$ is compact in $\wh\calz_{[0,\infty),\loc}$.
\end{proof}

Recall that $\scry_{[0,T]} = C([0,T]; \calw)$,
$\mathcal{C}^\kappa_{[0,T]} = C^\kappa([0,T]; \calw)$
and the global counterparts
$\scry_{[0,\infty), \loc} = C_{\loc}  ([0,\infty); \calw)$,
$\mathcal{C}^\kappa_{[0,\infty),\loc} = C^\kappa_{\loc} ([0,\infty); \calw)$,
where $\kappa\in (0,1/2)$.

\begin{lemma} [Tightness of $(P_n\circ y^{-1})$] \label{Lem-Pny-tight}
Assume the conditions in Proposition \ref{Prop-Stab-Mart} to hold.
Then,
$(P_n\circ y^{-1})$ is tight on $\scry_{[0,\infty),\loc}$.
\end{lemma}

\begin{proof}
Fix $\kappa \in (0,1/2)$.
Set $\ol{\Omega}_n:= \{y\in \scry_{[0,\infty), \loc}: y|_{[0,s_n]}\equiv y_{0,n}\}$,
\begin{align*}
   K^3_{n,k}:= \{y\in \Omega_n: \|y(\cdot+s_n) - y_{0,n}\|_{\calc^\kappa_{[0,k]}}^2  \leq R^2_k\},
\end{align*}
where $n,k\in \bbn$, $R_k>0$,
and set
\begin{align*}
   \ol{K}:= \bigcup_{n\in \bbn} \bigcap_{k\in \bbn} K^3_{n,k}.
\end{align*}

Since for the Wiener measure $\mu$, for every $k\geq 1$,
\begin{align*}
   \bbe^\mu \|y\|_{\calc^\kappa_{[0,k]}}^2
   <\infty,
\end{align*}
one can take an increasing large sequence $(R_k)$ such that
\begin{align*}
   \sup\limits_{n\in \bbn} P_n ((K^3_{n,k})^c)
   =& \sup\limits_{n\in \bbn} \mu (y\in \scry_{[0,\infty), \loc}: \|y\|_{\calc_{[s_n,s_n+k]}^\kappa}^2  >R_k^2 ) \notag \\
   \leq& \mu (y\in \scry_{[0,\infty), \loc}: \|y\|_{\calc_{[0, \sup_n s_n+k]}^\kappa}^2  >R_k^2 )
   \leq \frac{\ve}{2^k},
\end{align*}
which yields that
\begin{align*}
   \sup\limits_{n\in \bbn} P_n (\ol{K}^c) \leq  \ve.
\end{align*}

It remains to show that $\ol{K}$ is a compact set in
$\scry_{[0,\infty),\loc}$.
Take any sequence $(y_m) \subseteq \ol{K}$,
by the diagonal arguments,
it suffices to prove that
$(y_m|_{[0,k]})$ has a subsequence convergent in $\scry_{[0,k]}$.
Below we fix $k\geq 1$
and write $y_m$ for $y_m|_{[0,k]}$
for simplicity.

This is true if there exist infinitely many $y_m$ in
the same set $\ol{\Omega}_M$ for some $M\geq 1$.

Actually,
write
\begin{align*}
   h_m(t):= y_m(t+s_M) - y_{0,M},\ \ t\in [0,k].
\end{align*}
One has
\begin{align*}
   \sup\limits_{m\in \bbn}
   \|h_m\|_{\calc^\kappa_{[0,k]}} \leq R_k.
\end{align*}
By the Arzel\`{a}-Ascoli theorem,
there exists a subsequence (still denoted by $(h_m)$)
convergent in $\scry_{[0,k]}$,
and so does $(y_m)$.

Below,
without loss of generality,
we assume $x_m\in \Omega_m$ below, $m\geq 1$.
In particular,
if
\begin{align*}
   g_m(t) := y_m(t+s_m) - y_{0,m},\ \ t\in [0,\infty),
\end{align*}
we have
\begin{align*}
   \sup\limits_{m\in \bbn}
   \|g_m\|_{\calc^{\kappa}_{[0,k]}}
   \leq   R_k.
\end{align*}
Then, the above arguments show that
there exists a subsequence, still denoted by $(g_m)$,
convergent in $\scry_{[0,k]}$.

Then,
rewrite
\begin{align}  \label{ym-gm}
   y_m(t) = g_m(t-s_m) + y_{0,m}, \ \ t\in [s_m, \infty),
\end{align}
and $y_m(t) = y_{0,m}$ for $t\in [0,s_m]$, $m\geq 1$.
Without loss of generality
we may assume that $s_0<s_m<s_n<k$ below.

For any $\ve'>0$ fixed,
take any $m,n\in \bbn$ large enough
such that
\begin{align*}
   |s_m - s_n|^{\kappa} + \|y_{0,m} - y_{0,n}\|_{\calw}
   + \|g_m-g_n\|_{\scry_{[0,k]}}
   \leq  \frac{\ve'}{10 R_{k}}.
\end{align*}

For $t\in [0, s_m]$,
since $y_m|_{[0,s_m]} \equiv y_{0,m}$
and $y_n|_{[0,s_m]} \equiv y_{0,n}$,
one has for $m,n$ large enough,
\begin{align} \label{ym-yn-Cs0sm}
   \| y_m  - y_n \|_{\scry_{[0,s_m]}}
   = \|y_{0,m}- y_{0,n}\|_{\calw}
   \leq \frac{\ve'}{10}
\end{align}
For $t\in [s_m, s_n]$,
we see that
\begin{align*}
    y_m(t) - y_n(t)
   = g_m(t-s_m) + y_{0,m} - y_{0,n},
\end{align*}
which yields that
\begin{align}  \label{ym-yn-Csmsn}
   \|y_m - y_n \|_{\scry_{[s_m,s_n]}}
   \leq&  \|g_m\|_{C^{\kappa}_{[0,s_n-s_m]}} (s_n-s_m)^{\kappa} + \|y_{0,m} - y_{0,n}\|_{\calw} \notag \\
   \leq& R_{k}(s_n- s_m)^{\kappa} + \|y_{0,m} - y_{0,n}\|_{\calw}
   \leq \frac{\ve'}{5}.
\end{align}
At last, for $t\in [s_n, k]$,
by \eqref{ym-gm},
\begin{align*}
   \|y_m(t) - y_n(t)\|_{\calw}
    =& \| g_m(t-s_m) - g_n(t-s_n) + y_{0,m} - y_{0,n} \|_{\calw} \\
    \leq& \| g_m(t-s_m) - g_m(t-s_n)  \|_{\calw}
           + \| g_m(t-s_n) - g_n(t-s_n)  \|_{\calw}  \\
        &   + \|y_{0,m} - y_{0,n} \|_{\calw} \\
    \leq& \|g_m\|_{C^{\kappa}_{[0,k]}} (s_n-s_m)^{\kappa}
         + \|g_m-g_n\|_{\scry_{[0,k]}}
         +  \|y_{0,m} - y_{0,n} \|_{\calw}
    \leq \frac{3\ve'}{10}.
\end{align*}
Hence, we get that
\begin{align}  \label{ym-yn-Csmk}
   \|y_m - y_n \|_{\scry_{[s_n,k]}}
    \leq \frac{3\ve'}{10}.
\end{align}

Thus,
it follows from \eqref{ym-yn-Cs0sm}, \eqref{ym-yn-Csmsn} and \eqref{ym-yn-Csmk} that
for $m,n$ large enough
\begin{align*}
   \|y_m - y_n \|_{\scry_{[0,k]}} \leq \ve'.
\end{align*}
This yields that $\{y_m\}$
is a Cauchy sequence in $\scry_{[0,k]}$,
and so converges in $\scry_{[0,k]}$.
Therefore,
by diagonal arguments,
$\ol{K}$ is a compact set in $\scry_{[0,\infty),\loc}$.
\end{proof}

{\bf Proof of Proposition \ref{Prop-Stab-Mart}.}
By virtue of Lemmas \ref{Lem-Pnx-tight} and \ref{Lem-Pny-tight},
$(P_n)$ is tight on $\scrx_{[0,\infty),\loc} \times \scry_{[0,\infty),\loc}$.
Thus, there exists a subsequence (still denoted by $(n)$)
such that
$(P_n)$ converges weakly to $P$.

It remains to prove that $P\in \mathscr{M}(s_0,z_0, C_{\cdot})$
for some positive increasing function $(C_t)$.

For this purpose,
applying
the Jakubowski version of Skorokhod representation theorem
we have
a probability space $(\wt \Omega, \wt{\bbf}, \wt \bbp)$
and, along a further subsequence if necessary,
random variables $(\wt v_n, \wt w_n)$, $(\wt v, \wt w)$
	in $\wh\calz_{[0,\infty),\loc} \times \scry_{[0,\infty),\loc}$,
	$n \geq 1$,
	such that
	\begin{align}  \label{wtvn-wtWn-wtv-wtW}
		\wt \bbp \circ (\wt v_n, \wt w_n)^{-1} = P_n,\ \
		\wt  \bbp \circ (\wt v, \wt w)^{-1} = P,
	\end{align}
	and
	\begin{align} \label{wtvn-wtWn-limit2}
		(\wt v_n, \wt w_n) \to (\wt v, \wt w) \ \ {\rm in}\ \wh\calz_{[0,\infty),\loc} \times \scry_{[0,\infty),\loc},\ \ {\rm as}\ n\to \infty, \ \wt \bbp-a.s.
	\end{align}

It also follows from the definition of $\mathscr{M}(s_n, z_n, C_{\cdot})$
and \eqref{wtvn-wtWn-wtv-wtW} that
for every $n\geq 1$ and any $\psi\in \calc_0^\infty$,
\begin{align*}
     (\wt v_n(t)-x_{0,n}, \psi)
     =& \int_{s_n}^t (  -(-\Delta)^\alpha \wt v_n(s)
          - \div (\wt v_n(s) \otimes \wt v_n(s))
       + f(s, \wt w_n(s)), \psi)  \dif s  \\
  &  +  (\wt w_n(t) - y_{0,n},\psi),\ \ \wt \bbp-a.s.
\end{align*}
Now we pass to the limit term by term.
Since $\wt v\in C([0,T];\calh_w)$,
using the pathwise convergence $\wt v_n\to \wt v$ in $C([0,T];\calh_w)$, due to \eqref{wtvn-wtWn-limit2},
one has
\begin{align*}
	&|\int_{s_n}^t (- (-\Delta)^\al \wt v_n , \psi)\dif s-
	\int_{s_0}^t (- (-\Delta)^\al \wt v , \psi)\dif s|\\
	\leq& |\int_{s_n}^t (\wt v_n-\wt v, (-\Delta)^\al \psi)\dif s|+ |\int_{s_0}^{s_n} (\wt v,(-\Delta)^\al \psi)\dif s| \to 0,\ \ \wt \bbp-a.s.
\end{align*}
where we also used the convergence
\begin{align} \label{snzn-s0z0}
    (s_n, z_{n}) \to (s_0, z_0)\ \ {\rm in}\ [0,T] \times \calh \times \calw
\end{align}
and the fact that $(-\Delta)^\al \psi\in L^2$, which can be verified by the following argument:
since $\calf(\psi) \in \cals$, one has $|\calf(\psi)(\eta)|\leq C|\eta|^{-2\al-2}$ when $|\eta|>1$, thus
\begin{align*}
	\norm{(-\Delta)^\al \psi}_{L^2}^2
    =& \norm{|\eta|^{4\al}|\calf(\psi)(\eta)|^2}_{L^1} \notag \\
    \leq& C+ \norm{|\eta|^{4\al}|\calf(\psi)(\eta)|^2}_{L^1(|\eta|>1)}
    \leq C+C\norm{|\eta|^{-4}}_{L^1(|\eta|>1)} <\infty.
\end{align*}

Similarly, as in the proof of \eqref{div-limit},
since $\supp \psi\in B_R$ for some $R>0$,
we have
\begin{align*}
	&|\int_{s_n}^t( \div(\wt v_n\otimes \wt v_n), \psi)\dif s-\int^{t}_{s_0}( \div(\wt v\otimes \wt v), \psi)\dif s|\\
	=&|\int_{s_n}^t(\wt v_n\otimes \wt v_n-\wt v\otimes\wt v,\nabla\psi)\dif s-\int_{s_0}^{s_n}(\wt v\otimes\wt v, \nabla\psi)\dif s| \\
	\leq&\norm{\nabla\psi}_{L^\infty}\|\wt v_n - \wt v\|_{L^2(0,T; \calh_R)}
	(\sup_{n\geq 1} \|\wt v_n \|_{L^2(0,T; \calh_R)}
	+ \|\wt v\|_{L^2(0,T; \calh_R)})\\
	&+\norm{\nabla\psi}_{L^\infty}\norm{\wt v}_{L^2(s_0,s_n; \calh_R)}^2 \to 0,\ \ \wt \bbp-a.s.
\end{align*}

Moreover,
since $\wt v_n\to \wt v$ in $C([0,T];\calh_w)$ and $\wt w_n \to \wt w$ in $\scry_{[0,\infty),\loc}$,
using \eqref{snzn-s0z0} again we have
\begin{align*}
(\wt v_n(t)-x_{0,n}, \psi) -(\wt w_n(t) - y_{0,n},\psi)\to (\wt v(t)-x_{0}, \psi) -(\wt w(t) - y_{0},\psi),\ \ \wt \bbp-a.s.
\end{align*}

The treatment for the random forcing term is similar to that as in the proof of Proposition \ref{Prop-GWP-Mart}.

Thus, putting altogether we conclude that $\wt \bbp$-a.s. for any $t> s_0$,
\begin{align*}
	(\wt v(t)-x_{0}, \psi) =& \int_{s_0}^t ( - (-\Delta)^\alpha \wt v(s)
	  -  \div (\wt v(s) \otimes \wt v(s)
	  +  f(s, \wt w(s)), \psi)  \dif s
	 +  (\wt w(t) - y_{0},\psi),
\end{align*}
which, via \eqref{wtvn-wtWn-wtv-wtW},
yields that
$P$ verifies $(M1)$ of Definition \ref{def-mart-1}.

Regarding the energy bound and inequality in $(M2)$,
since
\begin{align*}
   \sup\limits_{n\in \bbn}
   \bbe^{\wt \bbp}
   \sup\limits_{s\in [0,t]} \|\wt v_n\|_{\calh}^2
   = \sup\limits_{n\in \bbn}
   \bbe^{P_n} \sup\limits_{s\in [0,t]} \|x\|_{\calh}^2
   \leq C_t (\|x_0\|_\calh^2 +1),
\end{align*}
taking into account \eqref{wtvn-wtWn-limit2}
we infer that
\begin{align*}
		\wt v_n  \stackrel{w^*}{\rightharpoonup} \wt v\ \
		{\rm in}\ L^2(\wt \Omega; L^\infty(0,t;\calh)),
\end{align*}
which yields that
\begin{align*}
    \bbe^{P} \|x\|^2_{L^\infty(0,t; \calh)}
   =\bbe^{\wt \bbp} \|\wt v\|^2_{L^\infty(0,t; \calh)}
   \leq \liminf_{n\to \infty}
   \bbe^{\wt \bbp} \|\wt v_n\|^2_{L^\infty(0,t; \calh)}
   \leq C_t (\|x_0\|_\calh^2 +1).
\end{align*}

Similarly, one has
\begin{align*}
   \sup\limits_{n\in \bbn}
   \bbe^{\wt \bbp} \|\wt v_n I_{[s_n,t]}\|^2_{L^2(0,t; \calv)}
   = \sup\limits_{n\in \bbn}
   \bbe^{P_n} \|x I_{[s_n,t]}\|^2_{L^2(0,t; \calv)}
   \leq  C_t (\sup\limits_{n\in \bbn}\|x_{0,n}\|_\calh^2 +1)<\infty,
\end{align*}
it follows that along a further subsequence if necessary,
\begin{align} \label{wtvn-wtv-L2V}
    \wt v_n I_{[s_n,t]}   \stackrel{w}{\rightharpoonup} \wt v I_{[s_0,t]}\ \
    {\rm in}\ L^2(\wt \Omega; L^2(0,T; \calv)),
\end{align}
and thus
\begin{align*}
   \bbe^{P} \|x\|_{L^2(s_0,t; \calv)}
   = \bbe^{\wt \bbp} \|\wt v I_{[s_0,t]}\|_{L^2(0,T; \calv)}
   \leq \liminf_{n\to \infty}
   \bbe^{\wt \bbp} \|\wt v_n  I_{[s_n,t]}\|_{L^2(0,T;\calv)}
   \leq C_t (\|x_0\|_\calh^2 +1).
\end{align*}

Since $P_n$ satisfies the energy inequality $(M2)$, by \eqref{wtvn-wtWn-wtv-wtW},
\begin{align*}
		\bbe^{\wt \bbp} \|\wt v_n(t)\|_{\calh}^2
		+ 2 \bbe^{\wt \bbp} \int_{s_n}^t \|(-\Delta)^{\frac \alpha 2} \wt v_n(s)\|_\calh^2 \dif s
		\leq& \|x_{0,n}\|_\calh^2 + 2 \bbe^{\wt \bbp} \int_{s_n}^t (f(\wt w_n), \wt v_n) \dif s  \\
            & + \sum_{j=1}^{\infty}\lambda_j\|\epsilon_j\|^2_{\calh} (t-s_n).
\end{align*}

Thus, by the similar argument as in the proof of \eqref{energy-ineq-wtv}, we derive the following
\begin{equation*}
		\begin{aligned}
			\bbe^{\wt \bbp} \|\wt v(t)\|_{\calh}^2
			+ 2 \bbe^{\wt \bbp} \int_{s_0}^t \|(-\Delta)^{\frac \alpha 2} \wt v(s)\|_\calh^2 \dif s
			\leq& \|x_{0}\|_\calh^2 + 2\liminf_{n\to \infty}\bbe^{\wt \bbp} \int_{s_n}^t (f(\wt w_n),\wt v_n) \dif s  \\
                &+ \sum_{j=1}^{\infty}\lambda_j\|\epsilon_j\|^2_{\calh} (t-s_0).
		\end{aligned}
\end{equation*}

In order to pass to the limit, we write
\begin{equation*}
	\begin{aligned}
		&\bbe^{\wt \bbp} \int_{s_n}^t (f(\wt w_n),\wt v_n)\dif s -\int_{s_0}^t(f(\wt w), \wt v )\dif s\\
		=&   \bbe^{\wt \bbp} \int_{s_0}^t(f(\wt w),\wt v_n I_{[s_n,t]}-\wt v)\dif s
            + \bbe^{\wt \bbp} \int_{s_0}^t ( f(\wt w_n)-f(\wt w), \wt v_n I_{[s_n,t]})\dif s
	\end{aligned}
\end{equation*}
Note that, the first term tends to $0$ due to \eqref{wtvn-wtv-L2V} and Lemma \ref{Lem-f-L2}.
Moreover,
the second term tends to $0$
by using similar proof as in  \eqref{Efwn-Efw-limit},
based on the pathwise convergence
$$\|f(\wt w_n)-f(\wt w)\|_{L^2(0,T;L^2)}\to 0,\ \ \wt \bbp-a.s,$$
the uniform boundedness of $\norm{\wt v_n}_{L^2(0,T;L^2)}$
together with the uniform integrability.

Thus, passing to the limit we derive the energy inequality
\begin{equation*}
		\begin{aligned}
			\bbe^{\wt \bbp} \|\wt v(t)\|_{\calh}^2
			+ 2 \bbe^{\wt \bbp} \int_{s_0}^t \|(-\Delta)^{\frac \alpha 2} \wt v(s)\|_\calh^2 \dif s
			\leq \|x_{0}\|_\calh^2 + 2\bbe^{\wt \bbp} \int_{s_0}^t (f(\wt w),\wt v) \dif s + \sum_{j=1}^{\infty}\lambda_j\|\epsilon_j\|^2_{\calh} (t-s_0),
		\end{aligned}
\end{equation*}

which, via \eqref{wtvn-wtWn-wtv-wtW}, verifies $(M2)$ for $P$.

In order to verify $(M3)$,
it is equivalent to prove that
for any $s_0 < t_1 <  t_2<  \cdots <  t_m<  s<t<\infty$,
$m\in \bbn$,
any bounded continuous function $g$ on $\mathbb{R}^m$
and $\xi \in \calw$,
\begin{align} \label{P-yt-ys}
   \bbe^P(e^{i \<y{(t-s_0)} - y_0,\xi\> + \frac 12 (t-s_0) \<\mathcal{Q}\xi, \xi\>} g_s )
    = \bbe^P(e^{i \<y{(s-s_0)} - y_0, \xi\> + \frac 12 (s-s_0)\<\mathcal{Q}\xi, \xi\>} g_s ),
\end{align}
where $g_s:=g(y{(t_1)},\cdots, y{(t_m)}) $
and $\<\cdot, \cdot\>$ denotes the inner product in $\calw$.
We also write $\wt g_s:=g(\wt w{(t_1)},\cdots, \wt w{(t_m)})$ below.

To this end,
since $(y(t))$ is a standard Wiener process starting from $z_n$ at time $s_n$ under $P_n$,
we infer that
$\{e^{i \<y{(t-s_n)} - y_{0,n}, \xi\> + \frac 12(t-s_n) \<\mathcal{Q}\xi, \xi\>}\}_{t\geq s_n}$
is a continuous martingale under $P_n$
and thus
for $n$ large enough,
\begin{align*}
   \bbe^{P_n}(e^{i \<y{(t-s_n)} - y_{0,n}, \xi \> + \frac 12 (t-s_n)\<\mathcal{Q}\xi, \xi\>} g_s )
    = \bbe^{P_n}(e^{i \<y{(s-s_n)} - y_{0,n},\xi\> + \frac 12  (s-s_n)\<\mathcal{Q}\xi, \xi\>} g_s),
\end{align*}
which along with \eqref{wtvn-wtWn-wtv-wtW} yields  that
\begin{align}  \label{Pntgs-Pnsgs}
   \bbe^{\wt \bbp}(e^{i \<\wt w_n{(t-s_n)} - y_{0,n},\xi \> +\frac 12  (t-s_n)\<\mathcal{Q}\xi, \xi\>} \wt g_s )
    = \bbe^{\wt \bbp}(e^{i  \<\wt w_n{(s-s_n)} - y_{0,n}, \xi \> + \frac 12 (s-s_n)\<\mathcal{Q}\xi, \xi\>} \wt g_s ).
\end{align}
Note that
\begin{align*}
      & |\<\wt w_n(t-s_n) - y_{0,n}, \xi\> -  \<\wt w(t-s_0) - y_0, \xi\>|  \\
  \leq& |\<\wt w_n(t-s_n), \xi \> - \<\wt w(t-s_n), \xi\>|
       + |\<\wt w(t-s_n), \xi \> - \<\wt w(t-s_0), \xi\>|
        + |\<y_{0,n}-y_0, \xi\>|  \\
  \leq& \|\wt w_n - \wt w\|_{C([0,t]; \calw)} \|\xi\|_{\calw}
         + \|\wt w\|_{C^{\frac 13}([0,t]; \calw)} \|\xi\|_{\calw}  (s_n-s_0)^{\frac 13}
         + \|y_{0,n} - y_0\|_\calw \|\xi\|_\calw.
\end{align*}
By \eqref{wtvn-wtWn-limit2} and \eqref{snzn-s0z0},
it follows that
\begin{align*}
   \<\wt w_n(t-s_n) - y_{0,n}, \xi \> - \<\wt w(t-s_0) - y_0, \xi\>
   \to 0,\ \ \wt{\bbp}-a.s.
\end{align*}
Similar results also holds when $t$ is replaced by $s$.
As
$|e^{i \<\wt w_n(t-s_n)- y_{0,n}, \xi\>}| = |e^{i \<\wt w_n(s-s_n)- y_{0,n}, \xi\>}| =1$,
applying the bounded convergence theorem to \eqref{Pntgs-Pnsgs}
we get
\begin{align*}
   \bbe^{\wt \bbp}(e^{i \<\wt w{(t-s_0)} - y_0, \xi \> + \frac 12  (t-s_0) \<\mathcal{Q}\xi, \xi\>} \wt g_s )
    = \bbe^{\wt \bbp}(e^{i \<\wt w{(s-s_0)} - y_0, \xi \> + \frac 12 (s-s_0) \<\mathcal{Q}\xi, \xi\>} \wt g_s ),
\end{align*}
which, via \eqref{wtvn-wtWn-wtv-wtW}, yields \eqref{P-yt-ys}
and so verifies $(M3)$ for $P$.
The proof is complete.
\hfill $\square$

\subsection{Gluing procedure}  \label{Subsec-Glue-Martingale}

The aim of this subsection is to  glue two martingale solutions together,
as stated in Proposition \ref{Prop-Glue} below.

\begin{proposition} [Gluing martingale solutions]  \label{Prop-Glue}
Let $z_0:= (x_0, y_0) \in \calh \times \calw$
and $P \in \mathscr{M}(0,z_0,C_{\cdot}, \tau)$
as in the sense of Definition \ref{def-mart-2},
where $\tau$ is a $(\mathcal{B}_t)_{t \geq 0}$-stopping time.
Suppose that there exist a Borel set $\mathcal{N} \subset \wh \Omega_{\tau}$
and $(Q_\omega)\subseteq \mathscr{P}(\wh \Omega)$,
such that $P(\mathcal{N})=0$ and for every $\omega \in \mathcal{N}^c$ it holds
\begin{equation}\label{qw2}
    Q_\omega (\omega^{\prime} \in \wh{\Omega}; \tau (\omega^{\prime} )=\tau(\omega))=1 .
\end{equation}
Then, the probability measure
$P \otimes_\tau R$ defined by
\begin{align*}
P \otimes_\tau R(\cdot):=\int_{\wh \Omega} Q_\omega(\cdot) P(d \omega)
\end{align*}
satisfies
\begin{align} \label{PR-P-before}
   P \otimes_\tau R=P \ \ on\ \wh{\Omega}_\tau,
\end{align}
and
\begin{align}  \label{PR-Mart}
   P\otimes_\tau R \in \mathscr{M}(0,z_0, C_{\cdot})
\end{align}
for some positive increasing function $t \mapsto C_t$.
\end{proposition}

Like the definition of $\sigma$ in \eqref{sigma-def},
we define the following stopping times on the canonical spaces:
\begin{align*}
	\tau^n(\omega):= (20C(N,\kappa))^{\frac{-1}{a-\ve_0}}\wedge (20C(N,\kappa))^{\frac{-1}{1+\kappa-\frac{5}{4\al}-\ve_0}} \wedge \inf\{t>0:\norm{y(\omega)}_{C^{\kappa}([0,t]; H_{x}^{N})} > 1-\frac1n \},
\end{align*}
and its limit
\begin{align*}
    \tau^\infty(\omega):=\lim_{n\to\infty} \tau^n(\omega),\ \ w\in \wh \Omega.
\end{align*}
As in \cite[Lemma 3.5]{HZZ19},
$(\tau^n)$ are $(\mathcal{B}_t)_{t \geq 0}$-stopping times, and so is $\tau^\infty$.
Using the stability result in Proposition \ref{Prop-Stab-Mart}
and similar arguments as in \cite{HZZ19},
the following result holds which permits
to glue a single trajectory with global martingale solutions.

\begin{lemma}\label{Lem-Glue-Qomega}
Let $\tau$ be a bounded $(\mathcal{B}_t)_{t \geq 0}$-stopping time.
Then, for every $\omega \in \wh \Omega$
there exists $Q_\omega \in \mathscr{P} (\wh \Omega)$ such that
\begin{equation}\label{Qw-tau-before}
   Q_\omega (\omega^{\prime} \in \wh \Omega ; \omega'|_{[0,\tau(\omega)]} =\omega|_{[0,\tau(\omega)]})=1,
\end{equation}
and
\begin{equation} \label{Qw-tau-after}
    Q_\omega(A)=R_{\tau(\omega), z(\tau(\omega), \omega)}(A), \ \ \forall A \in \mathcal{B}^{\tau(\omega)},
\end{equation}
where $R_{\tau(\omega), z(\tau(\omega), \omega)} \in \mathscr{M}(\tau(\omega), z(\tau(\omega),\omega)), C_{\cdot})$
with a positive increasing function $(C_t)$.

In addition, for every $B \in \mathcal{B}(\wh \Omega)$ the mapping $\omega \mapsto Q_\omega(B)$ is $\mathcal{B}_\tau$-measurable.
\end{lemma}

{\bf Proof of Proposition \ref{Prop-Glue}.}
Identity \eqref{PR-P-before} follows from \eqref{qw2} and \eqref{Qw-tau-before}.
Below we prove \eqref{PR-Mart},
that is,
$P\otimes_\tau R$ satisfies the properties $(M1)$-$(M3)$ in Definition \ref{def-mart-1}.

{\it $(i)$. Verification of $(M1)$.}
Let us first verify $(M1)$ in Definition \ref{def-mart-1}.
By the continuity of canonical process,
it suffices to prove that
\begin{align}  \label{M1-PR}
   P \otimes_\tau R (A_{t,0}) =1
\end{align}
for every  $0<t<\infty$,
where we set for $0\leq s<t<\infty$,
\begin{align*}
   A_{t,s} :=& \big\{x_t-x_s= \int_s^t P_H[-(-\Delta)^\alpha x_r - \div (x_r\otimes x_r) + f(r, y(r))]\dif s
              + (y_t - y_s) \big\}.
\end{align*}

To this end, we see that
\begin{align}  \label{PR-Ats}
    P \otimes_\tau R (A_{t,0})
    = P \otimes_\tau R (A_{t,0}, t<\tau) + P \otimes_\tau R (A_{t,0}, t\geq \tau).
\end{align}
Since $P \in \mathcal{M}(0,z_0, C_\cdot, \tau)$,
by $(M1')$,
$A_{t,0}$ occurs $P$-almost surely for $t<\tau$,
and thus by  \eqref{PR-P-before},
\begin{align} \label{PR-At-1}
   P \otimes_\tau R (A_{t,0}, t<\tau)
   = P (A_{t,0}, t<\tau)
   = P(t<\tau).
\end{align}
Moreover, since
$P \otimes_\tau R (A_{\tau,0})=1$,
by \eqref{qw2} and \eqref{Qw-tau-after},
\begin{align*}
  P \otimes_\tau R (A_{t,0}, t\geq \tau)
  = P \otimes_\tau R (A_{t,\tau}, t \geq \tau)
     =  \int_{\wh \Omega} R_{\tau(\omega), z(\tau(\omega),\omega)} (A_{t,\tau(\omega)}) I_{t\geq \tau(\omega)} P(\dif\omega).
\end{align*}
Taking into account
$R_{\tau(\omega), z(\tau(\omega),\omega)} \in \mathscr{M}(\tau(\omega),z(\tau(\omega),\omega), C_{\cdot})$,
and so,
$R_{\tau(\omega), z(\tau(\omega),\omega)} (A_{t,\tau(\omega)}) =1,$
we get
\begin{align} \label{PR-At-2}
   P \otimes_\tau R (A_{t,0}, t\geq \tau)  = P(t\geq \tau).
\end{align}
Thus, plugging \eqref{PR-At-1} and \eqref{PR-At-2} into \eqref{PR-Ats}
we get \eqref{M1-PR}, as claimed.

{\it $(ii)$ Verification of $(M2)$.}
Note that, for any $t>0$,
by \eqref{PR-P-before} and \eqref{Qw-tau-after},
\begin{align*}
	  & \bbe^{P\otimes_\tau R} (\|x\|_{C([0,t];\calh)}^2 + \|x\|_{L^2(0,t; \calv)}^2)  \\
	\leq& \bbe^{P} (\|x\|_{C([0,t\wedge \tau];\calh)}^2 + \|x\|_{L^2(0,t\wedge \tau; \calv)}^2)   \\
	& + \int_{\wh \Omega} \bbe^{R_{\tau(\omega), z(\tau(\omega),\omega)}}
       (\|x\|_{C([\tau(\omega), t];\calh)}^2
	+ \|x\|_{L^2(\tau(\omega),t; \calv)}^2) I_{\tau(\omega)<t} P(\dif\omega).
\end{align*}
Since $P$ and $R_{\tau(\omega), z(\tau(\omega), \omega)}$
satisfies the energy bounds, by \eqref{qw2}, one has
\begin{align*}
	& \bbe^{P\otimes_\tau R} (\|x\|_{C([0,t];\calh)}^2 + \|x\|_{L^2(0,t; \calv)}^2)  \\
	\leq& C_{t}(\norm{x_0}_{\calh}^2+1)
         + \int_{\wh \Omega} C_{\tau(\omega)}(\|x(\tau(\omega))\|_\calh^2+1)I_{\tau(\omega) < t}  P(\dif\omega)\\
	\leq& (2C_t+C_t^2)(\norm{x_0}_{\calh}^2+1).
\end{align*}
Hence, the energy bound is verified
with $C_t^2+C_t$ replacing $C_t$.

Below we verify the energy inequality for the glued measure.
To this end, we rewrite
\begin{align*}
     & \bbe^{P\otimes_\tau R}  (\|x(t)\|_\calh^2 + 2 \int_0^t \|(-\Delta)^{\frac \alpha 2} x(s)\|_\calh^2 \dif s )  \notag \\
   =&  \bbe^{P\otimes_\tau R} [ (\|x(t)\|_\calh^2 + 2 \int_0^t \|(-\Delta)^{\frac \alpha 2} x(s)\|_\calh^2 \dif s ) I_{t\leq \tau} ]
    + 2 \bbe^{P\otimes_\tau R} [\int_0^\tau \|(-\Delta)^{\frac \alpha 2} x(s)\|_\calh^2 \dif s I_{t>\tau}] \notag \\
      &  + \bbe^{P\otimes_\tau R}  [ (\|x(t)\|_\calh^2 + 2 \int_\tau^t \|(-\Delta)^{\frac \alpha 2} x(s)\|_\calh^2 \dif s ) I_{t>\tau}].
\end{align*}
By \eqref{qw2}, \eqref{Qw-tau-after} and
energy inequality $(M3)$ for $R_{\tau(\omega), z(\tau(\omega),\omega)}$,
the third term on the right-hand side above equals to
\begin{align*}
   & \int \bbe^{R_{\tau(\omega), z(\tau(\omega),\omega)}}
        (\|x(t)\|_\calh^2 + 2 \int_{\tau(\omega)}^t \|(-\Delta)^{\frac \alpha2} x\|_\calh^2 \dif s)I_{t>\tau(\omega)} P(\dif\omega)  \notag \\
   \leq&  \int \big(\|x(\tau(\omega))\|_\calh^2
           + 2 \bbe^{R_{\tau(\omega), z(\tau(\omega),\omega)}}  \int_{\tau(\omega)}^t (f,x) \dif s
           + \|\lambda\|_{\ell^1}(t-\tau(\omega) \big) I_{t>\tau(\omega)} P(\dif\omega)  \notag \\
   =& \bbe^{P\otimes_\tau R} [(\|x(\tau)\|_\calh^2
           + 2 \int_{\tau}^t (f,x) \dif s
           + \sum_{j=0}^{\infty}\lambda_j\|\epsilon_j\|^2_{\calh}(t-\tau)) I_{t>\tau}].
\end{align*}
Taking into account
\begin{align*}
   & \|x(t)\|_\calh^2 I_{t\leq \tau} + \|x_\tau\|_\calh^2 I_{t>\tau}
   + 2 \int_0^t \|(-\Delta)^{\frac \alpha2} x\|_\calh^2 \dif s  I_{t\leq \tau}
   + 2 \int_0^\tau \|(-\Delta)^{\frac \alpha2} x\|_\calh^2 \dif s  I_{t>\tau}  \\
   =&  \|x(t\wedge \tau)\|_\calh^2
       + 2 \int_0^{t\wedge \tau} \|(-\Delta)^{\frac \alpha2} x\|_\calh^2 \dif s,
\end{align*}
we lead to
\begin{align*}
     & \bbe^{P\otimes_\tau R} (\|x(t)\|_\calh^2 + 2 \int_0^t \|(-\Delta)^{\frac \alpha 2} x\|_\calh^2 \dif s )   \\
   =&   \bbe^{P\otimes_\tau R}  ( \|x(t\wedge \tau)\|_\calh^2
       + 2 \int_0^{t\wedge \tau} \|(-\Delta)^{\frac \alpha2} x\|_\calh^2 \dif s) \notag
   +  \bbe^{P\otimes_\tau R} (2\int_\tau^t (f,x)\dif s  + \sum_{j=0}^{\infty}\lambda_j\|\epsilon_j\|^2_{\calh}(t-\tau)) I_{t>\tau}).
\end{align*}
Then, using \eqref{qw2}, \eqref{PR-P-before}
and energy inequality $(M3')$ for $P$
we have
\begin{align*}
    & \bbe^{P\otimes_\tau R}  ( \|x(t\wedge \tau)\|_\calh^2
       + 2 \int_0^{t\wedge \tau} \|(-\Delta)^{\frac \alpha2} x\|_\calh^2 \dif s)  \\
   =& \bbe^{P}  ( \|x(t\wedge \tau)\|_\calh^2
       + 2 \int_0^{t\wedge \tau} \|(-\Delta)^{\frac \alpha2} x\|_\calh^2 \dif s) \\
   \leq& \|x_0\|_\calh^2 + 2 \bbe^P  \int_0^{t\wedge \tau} (f,x) \dif s +  \bbe^P   \|\lambda\|_{\ell^1} (t\wedge \tau) \\
   =& \|x_0\|_\calh^2 +  2 \bbe^{P\otimes_\tau R}  \int_0^{t\wedge \tau} (f,x) \dif s +  \sum_{j=0}^{\infty}\lambda_j\|\epsilon_j\|^2_{\calh} (t\wedge \tau).
\end{align*}
Thus, plugging this into the above identity we arrive at
\begin{align*}
    \bbe^{P\otimes_\tau R} (\|x(t)\|_\calh^2 + 2 \int_0^t \|(-\Delta)^{\frac \alpha 2} x\|_\calh^2 \dif s )
   \leq \|x_0\|_\calh^2 +  2 \bbe^{P\otimes_\tau R}  \int_0^{t} (f,x) \dif s + \sum_{j=0}^{\infty}\lambda_j\|\epsilon_j\|^2_{\calh} t,
\end{align*}
thereby justifying $(M2)$.

{\it $(iii)$. Verification of $(M3)$.}
It order to verify $(M3)$,
it is equivalent to prove that
for any $0\leq t_1<\cdots < t_n< s<t<\infty$,
$g= \prod_{i=1}^n g_i$, $g_i \in C_b(\mathbb{R})$
and $\xi \in \calw$,
\begin{align*}
   \bbe^{P \otimes_\tau R} (e^{i \<y(t)-y_0, \xi\> + \frac 12 t \<\mathcal{Q}\xi, \xi \>} g(y_{t_1},\cdots, y_{t_n}))
   = \bbe^{P \otimes_\tau R} (e^{i \<y(s) -y_0, \xi\> + \frac 12 s \<\mathcal{Q}\xi, \xi \>} g(y_{t_1},\cdots, y_{t_n})).
\end{align*}
For simplicity
we write
$g_s := g(y({t_1}),\cdots, y({t_n}))$.
Since $y_0$ is deterministic,
it is equivalent to prove that
\begin{align}  \label{M2-PR}
   \bbe^{P \otimes_\tau R} (e^{i \<y(t),\xi\> + \frac 12 t \<\mathcal{Q}\xi, \xi\>} g_s)
   = \bbe^{P \otimes_\tau R} (e^{i \<y(s),\xi\> + \frac 12 s \<\mathcal{Q}\xi, \xi\>} g_s).
\end{align}

In order to prove \eqref{M2-PR},
let us decompose
\begin{align*}
    & \bbe^{P \otimes_\tau R}(e^{i \<y(t),\xi\> + \frac 12 t (\mathcal{Q}\xi, \xi)} g_s)
   = \int_{\wh \Omega} \bbe^{Q_\omega} (e^{i \<y(t),\xi\> + \frac 12  t \<\mathcal{Q}\xi, \xi\>} g_s) P(\dif\omega) \\
   =& \int_{\wh \Omega} \bbe^{Q_\omega} (e^{i \<y(t), \xi\> + \frac 12 t \<\mathcal{Q}\xi, \xi\>} g_s) I_{s<t\leq \tau(\omega)} P(\dif\omega)
     + \int_{\wh \Omega} \bbe^{Q_\omega} (e^{i \<y(t), \xi\> + \frac 12 t \<\mathcal{Q}\xi, \xi\>} g_s) I_{s\leq \tau(\omega) <t} P(\dif\omega)  \\
    & + \int_{\wh \Omega} \bbe^{Q_\omega} (e^{i \<y(t), \xi\> + \frac 12 t \<\mathcal{Q}\xi, \xi\>} g_s) I_{\tau(\omega)< s<t} P(\dif\omega) \\
   =&: \wt  J_1 + \wt  J_2 + \wt  J_3.
\end{align*}

For the first term $\wt J_1$,
by \eqref{PR-P-before},
the property $(M3')$ of $P$
and the facts that
$\tau$ is $(\mathcal{G}_t)$-stopping time,
$(e^{i \<y(t),\xi\> + \frac 12 t  \<\mathcal{Q}\xi, \xi\>})$ is a continuous martingale under $P$,
we see that
\begin{align*}
     \bbe^{\mu}(e^{i \<y(\tau),\xi\> + \frac 12 \tau  \<\mathcal{Q}\xi, \xi\>} g_s I_{s<t\leq \tau})
     =&  \bbe^{P}(g(s) I_{s<t\leq \tau}  \bbe^{P} (e^{i\<y(\tau),\xi\> + \frac 12 \tau  \<\mathcal{Q}\xi, \xi\>} |\mathcal{G}_t))  \\
     =&  \bbe^{P}(g(s) I_{s<t\leq \tau}  \bbe^{P} (e^{i \<y(\tau), \xi\> + \frac 12 \tau  \<\mathcal{Q}\xi, \xi\>} |\mathcal{G}_{t\wedge \tau}))  \\
     =&  \bbe^{P}(g(s) I_{s<t\leq \tau}    e^{i \<y({t\wedge \tau}), \xi\> + \frac 12 (t\wedge \tau) \<\mathcal{Q}\xi, \xi\>})   \\
     =&  \bbe^{P}(g(s) I_{s<t\leq \tau}     e^{i \<y({t}),\xi\> + \frac 12 t   \<\mathcal{Q}\xi, \xi\>}).
\end{align*}
This yields that
\begin{align}  \label{K1}
    \wt  J_1
    =    \bbe^{P}( e^{i \<y({\tau}),\xi\> + \frac 12 \tau \<G\xi, \xi\>} g_s I_{s<t\leq \tau} ).
\end{align}

Regarding  the second term $\wt  J_2$, we use \eqref{Qw-tau-before} to rewrite
\begin{align*}
   \wt  J_2 = \int_{\wh \Omega}
        \bbe^{Q_\omega}  (e^{i \<y(t),\xi\> + \frac 12 t \<\mathcal{Q}\xi, \xi\>}) g_s(\omega) I_{s\leq \tau(\omega) <t} P(\dif\omega).
\end{align*}
Note that, for $t>\tau(\omega)$,
by \eqref{Qw-tau-after} and
the property $(M3)$ of $R_{\tau(\omega), z(\tau(\omega),\omega)}$,
\begin{align*}
  \bbe^{Q_\omega} (e^{i \<y(t),\xi\> + \frac 12t \<\mathcal{Q}\xi, \xi\>} )
  =& \bbe^{R_{\tau(\omega), z(\tau(\omega),\omega)}} (e^{i\<y(t),\xi\> + \frac 12 t \<\mathcal{Q}\xi, \xi\>} ) \\
  =& \bbe^{\mu} (e^{i \<y({\tau(\omega)},\omega) + y({t-\tau(\omega)}), \xi\> + \frac 12 t \<\mathcal{Q}\xi, \xi\>} ) \\
  =&e^{i \<y({\tau(\omega)},\omega), \xi\> + \frac 12 \tau(\omega) \<\mathcal{Q}\xi, \xi\> }
      \bbe^{\mu} (e^{i  \<y({t-\tau(\omega)}), \xi\> + \frac 12 (t-\tau(\omega)) \<\mathcal{Q}\xi, \xi\>} )  \\
  =& e^{i \<y(\tau(\omega), \omega), \xi\> + \frac 12 \tau(\omega) \<\mathcal{Q}\xi, \xi\>},
\end{align*}
where $\mu$ is the standard Wiener measure,
and the last step is to the fact that
for $\omega$ fixed,
$(e^{i \<y({t-\tau(\omega)}), \xi\> + \frac 12 (t-\tau(\omega)) \<\mathcal{Q}\xi, \xi\>})_{t-\tau(\omega)\geq 0}$
is a continuous martingale under $\mu$,
and so
\begin{align*}
    \bbe^{\mu} (e^{i  \<y({t-\tau(\omega)}), \xi\> + \frac 12 (t-\tau(\omega)) \<\mathcal{Q}\xi, \xi\>} ) =1.
\end{align*}
This yields that
\begin{align} \label{K2}
   \wt J_2  =&  \int_{\wh \Omega}  e^{i \<y({\tau(\omega)},\omega\>, \xi)
          + \frac 12  \tau(\omega) \<\mathcal{Q}\xi, \xi\>}
           g_s(\omega) I_{s\leq \tau(\omega) <t} P(\dif\omega)  \notag  \\
           =&  \bbe^{P} (e^{i \<y({\tau}), \xi\> + \frac 12  \tau \<\mathcal{Q}\xi, \xi\> } g_s  I_{s\leq \tau <t} ).
\end{align}

For the last term $\wt  J_3$,
by \eqref{Qw-tau-before} and \eqref{Qw-tau-after},
\begin{align*}
   \wt  J_3 =& \int_{\wh \Omega} \bbe^{R_{\tau(\omega), z(\tau(\omega), \omega)}}
              \big(e^{i \<y(t),\xi\> + \frac 12  t \<\mathcal{Q}\xi, \xi\>} \prod_{i=i_0+1}^n g(y(t_i)) \big)
               \prod_{i=1}^{i_0} g(y(t_i)(\omega))   I_{\tau(\omega)< s<t} P(\dif\omega) \\
       =& \int_{\wh \Omega} \bbe^{\mu}
           \big(e^{i \<y({\tau(\omega)},\omega) + y({t-\tau(\omega)}),\xi\>
                 + \frac 12 t \<\mathcal{Q}\xi, \xi\>} \prod_{i=i_0+1}^n g(y(\tau(\omega))+y(t_i-\tau(\omega))) \big ) \\
        & \qquad    \times \prod_{i=1}^{i_0} g(y(t_i)(\omega))
           I_{\tau(\omega)< s<t} P(\dif\omega),
\end{align*}
where $i_0\in \{1,\cdots, n\}$ satisfies $\tau(\omega) \in [t_{i_0}, t_{i_0+1})$
(we take $t_{n+1} =s$).
Again, by the martingale property of
$(e^{i \<y({t-\tau(\omega)}), \xi\> + \frac 12 (t-\tau(\omega)) \<\mathcal{Q}\xi, \xi\>})_{t - \tau(\omega)\geq 0}$
under $\mu$,
\begin{align*}
   \wt  J_3 =&  \int_{\wh \Omega} \bbe^{\mu}
              \big(e^{i \<y({\tau(\omega)},\omega) + y({s-\tau(\omega)}\>, \xi)
          + \frac 12 s \<\mathcal{Q}\xi, \xi\>} \prod_{i=i_0+1}^n g(y(\tau(\omega))+y(t_i-\tau(\omega))) \big)   \\
          &\qquad \times \prod_{i=1}^{i_0} g(y(t_i)(\omega))
          I_{\tau(\omega) < s<t} P(\dif\omega) \\
       =&  \int_{\wh \Omega} \bbe^{R_{\tau(\omega), z(\tau(\omega), \omega)}}
           \big(e^{i\<y(s),\xi \> + \frac 12 s \<\mathcal{Q}\xi, \xi\>} \prod_{i=i_0+1}^n g(y(t_i)) \big)
            \prod_{i=1}^{i_0} g(y(t_i)(\omega))
            I_{\tau(\omega) < s<t} P(\dif\omega),
\end{align*}
which, via \eqref{qw2}, \eqref{Qw-tau-before} and \eqref{Qw-tau-after}, yields that
\begin{align} \label{K3}
   \wt  J_3
     =   \int_{\wh \Omega} \bbe^{Q_\omega}
        (e^{i \<y(s),\xi\> + \frac 12 s \<\mathcal{Q}\xi, \xi\>} g_s) I_{\tau(\omega)<s<t} P(\dif\omega)
     =  \bbe^{P\otimes_\tau R} (e^{i \<y(s),\xi\> + \frac 12 s \<\mathcal{Q}\xi, \xi\>} g_s  I_{\tau<s<t} )
\end{align}

Now, a combination of \eqref{K1}, \eqref{K2} and \eqref{K3} together
leads to
\begin{align}  \label{PRht-PRhs.1}
    \bbe^{P \otimes_\tau R}(e^{i\<y(t),\xi\> + \frac 12  t \<\mathcal{Q}\xi, \xi\>} g_s)
    =&   \bbe^{P} (e^{i  \<y(\tau),\xi\> + \frac 12  \tau \<\mathcal{Q}\xi, \xi\>} g_s I_{s\leq \tau} ) \notag  \\
     & + \bbe^{P\otimes_\tau R} (e^{i \<y(s),\xi\> + \frac 12  s \<\mathcal{Q}\xi, \xi\>} g_s I_{\tau < s} ).
\end{align}
Note that, for the first term on the right-hand side above,
{\begin{align*}
    \bbe^{P} (e^{i \<y(\tau), \xi\> + \frac 12 \tau \< \mathcal{Q}\xi, \xi\>} g_s I_{s\leq \tau} )
     =& \bbe^{P} (  g_s I_{s\leq \tau}
         \bbe^P (e^{i\<y(\tau),\xi \> + \frac 12 \tau \<\mathcal{Q}\xi, \xi\>} |\mathcal{G}_{s}) )  \\
    =& \bbe^{P} (  g_s I_{s\leq \tau}
         \bbe^P (e^{i \<y(\tau),\xi\> + \frac 12 \tau \<\mathcal{Q}\xi, \xi\>} |\mathcal{G}_{s\wedge \tau}) ).
\end{align*}
By \eqref{PR-P-before} and the martingale property of $(e^{i \<y(t), \xi\> + \frac 12 t \< \mathcal{Q} \xi, \xi\>} )$ under $P$,
the right-hand side above equals to
\begin{align*}
       \bbe^{P} (  g_s I_{s\leq \tau}
            e^{i \<y({s\wedge \tau}),\xi\> + \frac 12 (s\wedge \tau) \<\mathcal{Q}\xi, \xi\>}  )
    =& \bbe^{P} (  g_s I_{s\leq \tau}  e^{i \<y(s),\xi \> + \frac 12 s \<\mathcal{Q}\xi, \xi\>}  ) \\
    =&  \bbe^{P\otimes_\tau R} ( e^{i \<y(s),\xi\> + \frac 12  s \<\mathcal{Q}\xi, \xi\>}   g_s I_{s\leq \tau}  ).
\end{align*}
Hence,
plugging this into \eqref{PRht-PRhs.1} we arrive at
\begin{align*}
     & \bbe^{P \otimes_\tau R}(e^{i \<y(t),\xi\> + \frac 12  t \<\mathcal{Q}\xi, \xi\>} g_s)  \\
    =&  \bbe^{P\otimes_\tau R} ( e^{i \<y(s),\xi\>  + \frac 12 s \<\mathcal{Q}\xi, \xi\> }  g_s I_{s \leq \tau}   )
      + \bbe^{P\otimes_\tau R} (e^{i \<y(s),\xi\>  + \frac 12 s \<\mathcal{Q}\xi, \xi\> } g_s I_{\tau < s} ) \\
    =&  \bbe^{P\otimes_\tau R} ( e^{i \<y(s),\xi\>  + \frac 12 s \<\mathcal{Q}\xi, \xi\> } g_s  ),
\end{align*}
which yields \eqref{M2-PR},
and so $(M3)$ follows.
Therefore,
the proof is complete.

\section{Proof of main results}  \label{Sec-Proof-Main}

\subsection{Non-uniqueness in law: below the Lions exponent}  \label{Subsec-Nonuniq-Lions}

We are now ready to prove the main result of this paper,
concerning the joint non-uniqueness in law of
Leray solutions to forced stochastic Navier-Stokes equations.

First,
by virtue of Theorem \ref{Thm-Nonuniq-Local},
for the given filtrated probability space $(\Omega, (\mathcal{F}_t), \bbp)$
and $\mathcal{Q}$-Wiener process $w$,
there exist two distinct Leray solutions $v_1, v_2$ given by \eqref{v1-def}
and \eqref{v2-def}, respectively, to
the forced stochastic equation \eqref{SNSE}
on $[0,\sigma]$,
where the forcing term $f$ is given by \eqref{f-force-def}
and $\sigma$ is an $(\mathcal{F}_t)$-stopping time given by \eqref{sigma-def},
satisfying
\begin{align} \label{bbp-sigma-1}
   \bbp(0<\sigma<\infty) =1.
\end{align}

We note that
$v_1 = \bar{u} + w$
is a global probabilistically strong solution to \eqref{SNSE},
and thus
$P_1 = \bbp \circ (v_1(\cdot), w(\cdot) )^{-1}$
is a global martingale solution to \eqref{SNSE} starting from zero at time zero,
that is,
\begin{align*}
    P_1 \in \mathscr{M}(0,0,C_{\cdot}).
\end{align*}

Below we construct a different global martingale solution to \eqref{SNSE} based on $v_2$.
For this purpose,
define the probability measure $P_2$ on the canonical space $\wh{\Omega}$ by
\begin{align} \label{P2-pviw-def}
   P_2:= \bbp \circ (v_2(\cdot \wedge \sigma), w(\cdot))^{-1}.
\end{align}
Then,
\begin{align*}
   P_2(0<\tau^\infty<\infty)
   = \bbp(0<\sigma<\infty) =1.
\end{align*}

Since $v_2$ is a probabilistically strong Leray solution to \eqref{SNSE} on $[0,\sigma]$,
it is also a weak solution to \eqref{SNSE} on $[0,\sigma]$
in the sense of Definition \ref{def-weak}.
As in Proposition \ref{Prop-equi-mart},
we infer that $P_2\in \mathscr{M}(0,0,C_{\cdot},\tau)$
for some positive increasing function $(C_t)$,
that is, $P_2$ is a local martingale solution to \eqref{SNSE}
on $[0,\sigma]$ with the initial condition zero at time $t=0$.

Then, in view of Lemma \ref{Lem-Glue-Qomega},
for every $\omega\in \wh{\Omega}$,
there exist $Q_{2, \omega} \in \mathcal{P}(\wh{\Omega})$
satisfying \eqref{Qw-tau-before} and \eqref{Qw-tau-after}.

Thus,
in order to construct a global martingale solution to \eqref{SNSE},
in view of the gluing result in Proposition \ref{Prop-Glue},
it remains to prove that \eqref{qw2} holds,
i.e., for a null set $\mathcal{N} \subseteq \wh{\Omega}$ with $P_2(\mathcal{N}) =0$
and for every $\omega \in \mathcal{N}^c$,
\begin{equation}  \label{Qomega-tau-1}
	Q_{2,\omega}
	(\omega^{\prime} \in \wh{\Omega};\  \tau^\infty(\omega^{\prime})=\tau^\infty(\omega))=1.
\end{equation}
This can be proved by using analogous arguments as in \cite{HZZ19,HZZ23}.

More precisely,
one has
\begin{align*}
	P_2(\omega\in\wh\Omega: y(\cdot\wedge \tau^\infty(\omega))\in C^{\kappa}_{\loc} H_{x}^{N})=1
	\end{align*}
and for all $\omega \in \wh \Omega$,
\begin{align*}
	Q_{2,\omega}(\omega'\in\wh\Omega: y(\omega')\in C^{\kappa}_{\loc} H_{x}^{N}) = 1,
\end{align*}
which gives a $P_2$-measurable null set $\mathcal{N}\subset \wh\Omega_{\tau^\infty}$ such that for all $\omega\in\mathcal{N}^c$,
\begin{align*}
	y(\cdot\wedge \tau^\infty(\omega)) \in C^{\kappa}_{\loc} H_{x}^{N}.
\end{align*}
Moreover, for all $\omega\in\mathcal{N}\cap \{x(\tau)\in \calh\}$, there exists a $Q_{2,\omega}-$measurable null set $N_\omega$ such that for all $\omega'\in N_\omega^c$, it holds $y(\omega') \in C^{\kappa}_{\loc} H_{x}^{N}$.
Therefore, for all $\omega'\in N_\omega^c$,
it holds that $\tau^\infty(\omega')=\wt\tau(\omega')$,
where
\begin{align*}
	\wt\tau(\omega'):=(20C(N,\kappa))^{\frac{-1}{a-\ve_0}}\wedge (20C(N,\kappa))^{\frac{-1}{1+\kappa-\frac{5}{4\al}-\ve_0}} \wedge \inf\{t>0:\norm{y(\omega')}_{C^{\kappa}([0,t]; H_{x}^{N})} \geq 1 \}.
\end{align*}
Hence,
\begin{equation*}
   \begin{aligned}
   	\{\omega'\in N_\omega^c: \wt\tau(\omega')\leq t\}
   	=&\{ \omega'\in N_\omega^c: \sup_{s_<s_2\in \mathbb{Q}\cap[0,t]}\frac{\|y(\omega',s_1)-y(\omega',s_2)\|_{H^N}}{(s_2-s_1)^\kappa} \geq 1  \}\\
   	& \cup \{\omega'\in N_\omega^c:(20C(N,\kappa))^{\frac{-1}{a-\ve_0}}\leq t \}\\
   	& \cup \{\omega'\in N_\omega^c: (20C(N,\kappa))^{\frac{-1}{1+\kappa-\frac{5}{4\al}-\ve_0}} \leq t \}\\
   	=&: N_\omega^c \cap A_t \subseteq N_\omega^c \cap \mathcal{B}_t^0,
   \end{aligned}
	\end{equation*}
which yields
\begin{equation*}
		\{\omega'\in N_\omega^c: \wt\tau(\omega') = \tau^\infty(\omega)\} = N_\omega^c \cap (A_{\tau^\infty(\omega)} \backslash \cup_{n=1}^{\infty}A_{\tau^\infty(\omega)-\frac1n} ) \in N_\omega^c \cap \mathcal{B}_{\tau^\infty(\omega)}^0.
	\end{equation*}
Thus, for all $\omega\in \mathcal{N}^c\cap \{x(\tau)\in \calh\}$ with $P_2(x(\tau)\in \calh)=1$,
\begin{equation*}
\begin{aligned}
		&Q_{2,\omega}
	(\omega^{\prime} \in \wh{\Omega}; \tau^\infty(\omega^{\prime})=\tau^\infty(\omega))\\
	=&Q_{2,\omega}
	(\omega^{\prime} \in N_\omega^c; \tau^\infty(\omega^{\prime})=\tau^\infty(\omega))\\
	=& Q_{2,\omega}
	(\omega^{\prime} \in N_\omega^c; \omega'\in (A_{\tau^\infty(\omega)} \backslash \cup_{n=1}^{\infty}A_{\tau^\infty(\omega)-\frac1n}))\\
	=&  Q_{2,\omega}
	(\omega^{\prime} \in N_\omega^c; \omega'\in (A_{\tau^\infty(\omega)} \backslash \cup_{n=1}^{\infty}A_{\tau^\infty(\omega)-\frac1n}), \omega'|_{[0,\tau^\infty(\omega)]}=\omega|_{[0,\tau^\infty(\omega)]} )\\
	=& 1,
\end{aligned}
\end{equation*}
which yields \eqref{Qomega-tau-1}.

Now, an application of Proposition \ref{Prop-Glue}
	gives a global martingale solution
	$P_2\otimes_{\tau^\infty} R \in \mathscr{M}(0,0,C_{\cdot})$
	with a positive increasing function $(C_t)$.
	Since by the proof of Theorem \ref{Thm-Nonuniq-Local},
	$v_1$ and $v_2$ have difference decay rates near the initial time $0$,
	it follows that their laws are different,
	and thus, $P_1 \not = P_2\otimes_{\tau^\infty} R$.	
	Therefore,
	the proof of Theorem \ref{Thm-Nonuniq} is complete.

\subsection{Strong well-posedness: above the Lions exponent} \label{Subsec-GWP-Lions}

We shall show that in the case with viscosity above the Lions exponent,
the stochastic Navier-Stokes system \eqref{Hyper-SNSE} is globally well-posed
in the probabilistically strong sense.

Below let us still use the space $\calh, \calv, \calv'$ as in Section \ref{Sec-Global-Mart},
but with $\alpha \geq 5/4$.

We shall apply the local monotonicity framework
developed in \cite{LR15}
to prove Theorem \ref{Thm-GWP-HyperSNSE}.
We also would like to refer to the recent progress \cite{RSZ22}.
Note that,
the spaces $\calh, \calv, \calv'$ form a Gelfand triple
\begin{align*}
    \calv \subset \calh \equiv \calh' \subset \calv'.
\end{align*}

Define the operator $A: [0,T]\times  \calv \times \Omega \mapsto \calv'$ by
\begin{align*}
   A(t,v, \omega):= P_{H} [-(-\Delta)^\alpha v(t)- \div (v(t)\otimes v(t)) + f(t,\omega)],
\end{align*}
where $t\in [0,T]$,
$v\in \calv$, $\omega\in \Omega$,
$P_H$ denotes the Helmhotz-Leray projection,
and the forcing term $f \in L^2(\Omega; L^2(0,T; H^{-\alpha}_x))$.

\begin{lemma}
For any   $v_i\in \calv$, $i=1,2,3$,
one has
\begin{align}   \label{div-vv-0}
     \< \div (v_1\otimes v_2), v_1 \> =0,
\end{align}
and
\begin{align} \label{div-vvv-bdd}
     |\<\div (v_1\otimes v_2), v_3\>|
     \leq \|v_1\|_{\calv} \|v_2\|_{\calh} \|v_3\|_{\calv}.
\end{align}
In particular,
\begin{align} \label{div-vv-bdd}
     \|\div  (v_1\otimes v_2)\|_{\calv'} \leq \|v_1\|_{\calv} \|v_2\|_{\calh}.
\end{align}
\end{lemma}

\begin{proof}
Identity \eqref{div-vv-0} follows from integration-by-part formula.
Regarding \eqref{div-vvv-bdd},
using
the integration-by-parts formula and
H\"older's inequality we get
\begin{align*}
   |\<\div (v_1\otimes v_2), v_3\>|
   =|\<v_1\otimes v_2, \na v_3\>|
   \leq \|v_1\|_{L^{12}} \|v_2\|_{L^{2}} \|\na v_3\|_{L^{\frac{12}{5}}} ,
\end{align*}
which, via the Sobolev embeddings
$H^\alpha \hookrightarrow H^{\frac 54} \hookrightarrow L^{12}$
and $ H^\alpha \hookrightarrow H^{\frac 14} \hookrightarrow L^{\frac{12}{5}}$
when $\alpha \geq 5/4$,
yields that
\begin{align*}
   |\<\div (v_1\otimes v_2), v_3\>|
    \leq \|v_1\|_{H^{\frac 54}} \|v_2\|_{L^2} \|v_3\|_{H^{\frac 14}}
   \leq \|v_1\|_{\calv} \|v_2\|_{\calh} \|v_3\|_{\calv},
\end{align*}
thereby proving \eqref{div-vvv-bdd}.
Estimate \eqref{div-vv-bdd} then follows from duality.
\end{proof}

{\bf Proof of Theorem \ref{Thm-GWP-HyperSNSE}.}
By virtue of Theorem \cite[Theorem 5.13]{LR15}
in the local monotonicity framework,
it suffices to prove that
the operator $A$ satisfies the following
conditions:
\begin{enumerate}
  \item[(H1)] Hemicontinuity: The map
  $\lambda \mapsto _{\calv'}(A(t,u_1+\lambda u_2), u_3)_{\calv}$
  is continuous on $\mathbb{R}$.

  \item[(H2)] Local monotonicity:
  \begin{align*}
     _{\calv'}(A(t, u) - A(t, v), u-v)_{\calv}
     \leq (g(t) + \rho(v)) \|u-v\|_{\calh}^2,
  \end{align*}
  where $g\in L^1(\Omega\times [0,T])$,
  $\rho: \calv \to [0,\infty)$ is a measurable hemicontinuous function
  and satisfies $\rho(v) \leq C(1+\|v\|_{\calv}^2)(1+\|v\|_{\calh}^2)$.

  \item[(H3)] Coercivity:
  \begin{align*}
     2 _{\calv'} (A(t,v), v)_\calv + \|\lambda\|_{\ell^1} t
      \leq C_0 \|v\|_{\calh}^2 - \theta \|v\|_{\calv}^2 + g(t),
  \end{align*}
  where $C_0, \theta>0$.

  \item[(H4)] Growth:
  \begin{align*}
       \|A(t,v)\|_{\calv'}^2 \leq (g (t) + C_0\|v\|_{\calv}^2) (1+ \|v\|_{\calh}^2).
  \end{align*}
\end{enumerate}

It is clear that the operator $A$ satisfies the hemicontinuity condition $(H1)$.

For the local monotonicity,
using \eqref{div-vv-0} and \eqref{div-vvv-bdd} we have
\begin{align*}
      _{\calv'} (A(\cdot, u) - A(\cdot, v), u-v )_{\calv}
   =&   _{\calv'}(-(-\Delta)^\alpha(u-v) - \div ( v \otimes (u-v)), u-v)_{\calv}  \\
   \leq& -\|u-v\|_{\calv}^2 + C \|u-v\|_{\calh} \|v\|_{\calv} \|u-v\|_{\calv}+C\|u-v\|_{\calh}^2,
\end{align*}
which via the Cauchy inequality yields that
\begin{align*}
     _{\calv'} (A(\cdot, u) - A(\cdot, v), u-v )_{\calv}
    \leq -\frac 12 \|u-v\|_{\calv}^2 + C\|v \|_{\calv}^2 \|u-v\|_{\calh}^2+C\|u-v\|_{\calh}^2.
\end{align*}
Thus, the local monotonicity $(H2)$ is verified
with $\rho(v) = C(1+ \|v \|_{\calv}^2)$.

Regarding the coercivity property,
identity \eqref{div-vv-0} and the Cauchy inequality
yield that for any $v\in \calv$,
\begin{align*}
   _{\calv'}(A(v), v)_{\calv}
   \leq - \|v\|_{\calv}^2 + \|f\|_{\calv'} \|v\|_{\calv}+C\|v\|_{\calh}^2
   \leq - \frac 12  \|v\|_{\calv}^2 + C \|f\|_{\calv'}^2+C\|v\|_{\calh}^2,
\end{align*}
which verifies the coercivity condition $(H3)$ with
$\theta = 1/2$ and $g = C(1+ \|\lambda\|_{\ell^1} t+ \|f\|_{\calv'}^2) \in L^1(\Omega\times [0,T])$
because $f\in L^2(\Omega; L^2_{\loc} H^{-\alpha})$.

At last, by estimate \eqref{div-vv-bdd},
\begin{align*}
    \|A(\cdot, v)\|_{\calv'}
    \leq& C(\|v\|_{\calv} + \|v\|_{\calh}\|v\|_{\calv} + \|f\|_{\calv'})  \notag \\
    \leq& C(1+\|f\|_{\calv'} + \|v\|_{\calv}) (1+ \|v\|_{\calh}).
    \end{align*}
Hence, the growth condition $(H4)$ follows as $g = C(1+ \|\lambda\|_{\ell^1} t+ \|f\|_{\calv'}^2)$ .

Thus,
applying Theorem 5.1.3 of \cite{LR15} we
obtain the global well-posedness of \eqref{Hyper-SNSE}.
The energy inequality and bound can be proved standardly
by using It\^o's formula.
\hfill $\square$

\subsection{Well-posedness with high probability: small force}

We first rewrite the Duhamel formulation of \eqref{SNSEdelta}:
\begin{equation*}
	v(t) =  - \int_0^t e^{(t-s)\Delta} P_H \div(v \otimes v)\, \dif s
	+\delta \int_0^t e^{(t-s)\Delta} P_H f(s) \dif s
	+ \int_0^t e^{(t-s)\Delta} \dif w(s).
\end{equation*}
To shorten notation we set
\begin{equation*}
	g_1(x,t) := \int_0^t e^{(t-s)\Delta} P_H f(s) \dif s \, , \quad
	g_2(x,t) := \int_0^t e^{(t-s)\Delta} \dif w(s)\, .
\end{equation*}
We recall standard heat flow estimates that will play a role in our proof.
\begin{equation}\label{eq: heat est}
	\| \nabla^k e^{t\Delta} u \|_{L^q}
	\le C(k,q,p) t^{-\frac{k}{2} + \frac{3}{2}\left(\frac{1}{p} - \frac{1}{q}\right)}\| u \|_{L^p}\, .
\end{equation}

{\bf Proof of Theorem \ref{thm:well-posedness small force}. }
	Fix $\delta<1$ and $t<1/2$.
	By means of \eqref{eq: heat est}, we estimate
	\begin{align*}
		\| g_1(\cdot, t) \|_{L^4} & \le \int_0^t (t-s)^{-\frac 3 8} \| f(\cdot, s)\|_{L^2}\dif s
		\\
		&\le C\int_0^t (t-s)^{- \frac 3 8} s^{- \frac 3 2} s^{\frac 3 4} \|F \|_{L^\infty((-\infty,1),L^2_\xi)}\, \dif s
		\\
		&\le C t^{-\frac 1 8} \|F \|_{L^\infty((-\infty,1),L^2_\xi)}\, .
	\end{align*}
	Hence,
	\begin{equation*}
		\bbe \left[ \sup_{t\in (0,1/2)} t^{\frac 1 8}\| g_1(\cdot,t)\|_{L^4} \right] \le C\cdot C_1 \, .
	\end{equation*}
	Moreover, since $\bbe \|g_2\|_{C([0,T]; H^N)}<\infty$,
taking into account Sobolev's embedding we have
	\begin{equation*}
		\bbe \left[ \sup_{t\in (0,T)} t^{\frac 1 8}\| g_2(\cdot,t)\|_{L^4} \right]
        \leq C' T^{\frac 18} \bbe \sup\limits_{t\in(0,T)} \|g_2(\cdot,t)\|_{H^N}
		\le C T^{\frac 1 8}.
	\end{equation*}

	We now define
	\begin{equation*}
		E:=
		\left\lbrace
		\omega \in \Omega\, :
		\sup_{t\in (0,1/2)} t^{\frac 1 8}\|g_1(\cdot, t)\|_{L^4}\le \delta^{- \frac 1 2} \, , \quad
		\sup_{t\in (0,\delta^2)} t^{\frac 1 8}\|g_2(\cdot, t)\|_{L^4}\le \delta^{\frac 18}
		\right\rbrace \, .
	\end{equation*}
	It follows from the previous estimates that $\bbp(\Omega \setminus E)< \epsilon$ if $\delta\le \delta(\epsilon, C_1)$.

	Fix $\omega\in E$.
	We prove that \eqref{SNSEdelta} is pathwise well-posed in the space defined by the norm
	\begin{equation*}
		\| v \|_X := \sup_{t\in (0,\delta^2)} t^{\frac 1 8} \| v(\cdot, t)\|_{L^4}\, ,
	\end{equation*}
	provided $\delta$ is small enough. We apply a fixed point argument with the operator
	\begin{equation*}
		\mathcal{T} v(t) := - \int_0^t e^{(t-s)\Delta} P_H \div(v \otimes v)\, \dif s
		+ \delta g_1(t) + g_2(t)\, .
	\end{equation*}
	Wa aim to show that $\mathcal{T} $ is a contraction on a ball of $X$ centered at the origin.
Being a quadratic operator, it is enough to show that
	\begin{itemize}
		\item[(i)] The bilinear form
		\begin{equation*}
			B(u,v) := - \int_0^t e^{(t-s)\Delta} P_H \div(v \otimes u)\, \dif s
		\end{equation*}
		is bounded, i.e. $\| B \|_{X\times X \to X} \le C_2<\infty$.
		
		\item[(ii)] $\| \delta g_1 + g_2\|_{X} \le \gamma$ for some $\gamma\le \gamma(C_2)$ small enough.
	\end{itemize}
	Condition (ii) is easy to check. For every $\omega\in E$ we have
	\begin{equation*}
		\|\delta g_1 + g_2\|_X \le \delta \|g_1\|_X + \|g_2\|_X \le \delta^{\frac 1 2} + \delta^{\frac 18} \, ,
	\end{equation*}
	therefore it is enough to choose $\delta$ small enough.

	To prove (i), we need to use once more the fact that $X$ is a critical space and the heat flow estimate \eqref{eq: heat est}:
	\begin{align*}
		\| B(u,v)(\cdot, t)\|_{L^4
		} & \le \int_0^t \| e^{(t-s)\Delta} P_H \div(u\otimes v)\|_{L^4}\dif s
		\notag \\&
		\le C\int_0^t (t-s)^{-\frac{1}{2} - \frac{3}{2}\left(\frac{1}{2} - \frac{1}{4}\right)} \| u \otimes v\|_{L^2}\dif s
		\notag \\&
		\le C \| u \|_X \cdot \|v\|_X \int_0^t (t-s)^{-\frac{1}{2} - \frac{3}{2}\left(\frac{1}{2} - \frac{1}{4}\right)} s^{-\frac{1}{4}}\dif s
		\notag \\&
		\le C t^{-\frac 1 8} \| u \|_X \cdot \|v\|_X \, .
	\end{align*}
Therefore, the proof is complete.
\hfill $\square$

\section*{Appendix}
\appendix
This section contains several technical proofs
and a brief introduction
of the ``background'' profile $\bar{U}$
constructed in \cite{ABC21}.

\section{Unstable velocity profile}   \label{App-Vel-Profile}

The unstable profile $\bar{U}$ constructed in \cite{ABC21} is an axisymmetric velocity field without swirl, taking the shape of a vortex ring. Its cross-section is a slight perturbation of the unstable vortex constructed by Vishik in \cite{Vishik2018}.
The latter, is a smooth, radially symmetric vorticity profile
$\bar{\omega}=\bar{\omega}(\varrho)$,
which decays at infinity. Importantly,
the associated linearized 2D Euler operator
\begin{align*}
	-\boldsymbol{A} \omega : =\bar{u} \cdot \nabla \omega+u \cdot \nabla \bar{\omega}\, , \quad
	\bar{u} = \operatorname{BS}_{2d}(\bar{\omega})
\end{align*}
admits an eigenvalue $\lambda$ with $\operatorname{Re} \lambda>0$.

As demonstrated in \cite[Prop. 2.2, Cor. 2.3]{ABC21}, the truncated vortex $\bar u_R = \phi_R \bar{u}$, $\bar{\omega}_R = \operatorname{curl} \bar u_R$, where $\phi_R \in C_0^{\infty}(B_R)$ is a radially symmetric cutoff function such that $\phi_R \equiv 1$ in $B_{R/2}$, remains unstable when $R$ is sufficiently large. More precisely, the operator
\begin{align*}
	-\boldsymbol{L}_{R}(\omega)
	=
	\bar u_R \cdot \nabla \omega+u \cdot \nabla \bar \omega_R\, ,
\end{align*}
admits an eigenvalue $\lambda_R$ with positive real part such that $\lambda_R \to \lambda$ as $R\to \infty$.

%
%
%
%

\medskip

The three-dimensional vortex ring is formed by employing the compactly supported vortex $\bar u_R(x, y) = \bar u_R^x(x,y) e_x + \bar u_R^y(x,y) e_y$ as a cross-section for a three-dimensional axisymmetric velocity field.
More precisely, in cylinder coordinates $(r\cos \theta, r\sin \theta, z)\in \mathbb{R}^3$ we define
\begin{align*}
	\widetilde{u}(r, z) :=u_R^x(r-\ell,z) e_r + \bar u_R^y(r-\ell,z)e_z \, ,
\end{align*}
where $\ell\ge 2R$ is a big parameter. After correcting the divergence with a smooth perturbation $v_\ell$ that tends to zero in $C^k(B_{R})$ as $\ell\to\infty$ for all $k>0$, we get the final profile:
\[
\widetilde{u}_{\ell}: =\widetilde{u}+v_{\ell}, \quad
\widetilde{\omega}_{\ell}=\operatorname{curl}_{\ell} \widetilde{u}_{\ell}
:=-\partial_z \widetilde{u}_{\ell}^r+\partial_r \widetilde{u}_{\ell}^z.
\]
The associated linearized 3D axisymmetric-no-swirl Euler operator
then has the form
\[
-\boldsymbol{L}_{\ell} \omega:=\widetilde{u}_{\ell} \cdot \nabla \omega+u \cdot \nabla \widetilde{\omega}_{\ell}-\frac{\widetilde{u}_{\ell}^r}{r+\ell} \omega-\frac{u^r}{r+\ell} \widetilde{\omega}_{\ell}\, ,
\]
and admits an unstable eigenvalue
$\lambda_\ell  {\to} \lambda_R$ when $\ell$ is big enough, see \cite[Prop.2.6]{ABC21}.

We can finally fix $\ell$ big enough so that $\boldsymbol{L}_{\ell}$ is unstable and set $\bar U = \tilde u_\ell$.

\section{Approximate stochastic NSE}  \label{App-C}

In this section we prove the energy inequality \eqref{Energy-ineq-SDE} and
the energy bound \eqref{vn-CtL2-L2V-bdd}
for the finite-dimensional approximate stochastic NSE \eqref{equa-vn}.

We may regard $v= \sum_{i=1}^n v^i e_i \in \calh_n$
as a vector $x = (v^1, \cdots, v^n)$ in $\mathbb{R}^n$.
Then,
for any $x\in \bbr$, $t\geq 0$,
set $b_n(t,x) := (b^1_n(t,x), \cdots, b^n_n(t,x))$
with
\begin{align*}
   b_n^i(t,x):= _{\calv'}(- (-\Delta)^\alpha v - \div (v\otimes v) + f(t,w(t)), e_i)_{\calv},
  \ \ 1\leq i\leq n.
\end{align*}
For the diffusion coefficients,
we set $\sigma_n = (\sigma_{n,ij})$
with
\begin{align*}
    \sigma_{n,ij}
    := \sqrt{\lambda_j} (\epsilon_j, e_i)_{\calh},
    \ \ 1\leq i,j\leq n.
\end{align*}

Let $X_n:= (v_n^1,\cdots, v_n^n)$
and let $B_n(t) := (\beta_1(t), \cdots, \beta_n(t))$,
$t\geq 0$,
be the $n$-dimensional Brownian motion.
Then, equation \eqref{equa-vn} can be formulated as
\begin{align}  \label{equa-X-SDE}
   dX_n(t) = b_n(t,X_n(t)) dt + \sigma_n  dB_n(t).
\end{align}

For any $x,y \in \bbr^n$,
$|x|, |y|\leq R$,
let $u:= \sum_{i=1}^n x^i e_i$
and $v:= \sum_{i=1}^n y^i e_i$.
Then,
$\|u\|_\calh, \|v\|_\calh \leq R$.
We have
\begin{align*}
   (x-y, b_n(t,x) - b_n(t,y))
   =&  _{\calv'}( - (-\Delta)^\alpha (u-v) - \div ((u-v)\otimes u)  - \div (v\otimes (u-v)), u-v )_{\calv}  \\
   \leq& -\|u-v\|_\calv^2 + C(n) R\|u-v\|_\calh^2,
\end{align*}
where we used the identity
\begin{align*}
    _{\calv'}(\div (u-v)\otimes v, u-v) =0
\end{align*}
and the estimate
\begin{align*}
    _{\calv'}(\div  v\otimes (u-v), u-v)_{\calv}
    =&   ( (u-v) \cdot \na v, u-v)_{\calh}   \\
    \leq& \|\na v\|_{L^\infty} \|u-v\|_\calh^2
    \leq \|v\|_{\calu}  \|u-v\|_\calh^2
    \leq C(n) R  \|u-v\|_\calh^2,
\end{align*}
due to the embedding $\calu \hookrightarrow W^{1,\infty}$
and
\begin{align*}
    \|v\|_\calu^2
    \leq (\sum\limits_{i=1}^n \|e_i\|_\calu^2) |x|^2
    =:C^2(n) |x|^2.
\end{align*}

Moreover,
since $_{\calv'}(\div(u\otimes u), u)_{\calv} =0$,
\begin{align*}
   (x,b_n(t,x))
   =&  _{\calv'}( - (-\Delta)^\alpha u  - \div (u\otimes v)  + f(t,w(t)) , u )_{\calv}  \\
   =& _{\calv'}( - (-\Delta)^\alpha u   -f(t,w(t)), u)_{\calv} \\
   \leq& -\|u\|_{\dot{H}^\alpha}^2 + \|u\|_\calv \|f(t,w(t))\|_{\calv'} \\
   \leq&  (1+ \|f(t,w(t))\|_{\calv'}^2 )(1+\|u\|_{\calu}^2) \\
   \leq& (1+C^2(n) )(1+ \|f(t,w(t))\|_{\calv'}^2 ) (1+|x|^2),
\end{align*}
and
\begin{align*}
   \|\sigma_n\|^2
   = \sum\limits_{i=1}^n  \sum\limits_{j=1}^n
      \lambda_j  (\epsilon_j, e_j)_{\calh}^2
   \leq  \sum\limits_{j=1}^\infty \lambda_j<\infty.
\end{align*}
It follows that
\begin{align*}
   2(x,b_n(t,x)) +  \|\sigma_n\|^2
   \leq 2(1+C^2(n))(1+ \|f(t,w(t))\|_{\calv'}^2 + \|\lambda\|_{\ell^1} )
          (1+|x|^2).
\end{align*}

Thus,
in view of the $L^2$-integrability of the forcing term,
we can apply Theorem 3.1.1 of \cite{LR15}
to obtain that a unique $(\calf_t)$-adapted and continuous solution
$X_n$ to \eqref{equa-X-SDE},
which gives the unique $(\calf_t)$-adapted solution $v_n$ to \eqref{equa-vn},
satisfying $v_n\in C([0,T]; \calh_n)$, $\bbp$-a.s.

We are left to prove the uniform estimate
\eqref{vn-CtL2-L2V-bdd}.
To this end,
an application of the It\^{o} formula gives
\begin{align}
   \frac 12 \|v_n(t)\|_{\calh}^2
    =& \int_0^t  \ _{\calv'}( - (-\Delta)^\alpha v_n + \div (v_n\otimes v_n) + f(s,w(s)), v_n)_{\calv} \dif s
       + \frac12\|\sigma_n\|^2 t  \notag   \\
     &  + \int_0^t (v_n, dw^{(n)}(s))  \label{Ito-vn-equa} \\
    \leq& -\frac 12 \int_0^t \|v_n\|_\calv^2 \dif s
          + 2 \int_0^t \|f(s,w(s))\|_{\calv'}^2 \dif s
          + \int_0^t \|v_n\|_\calh^2 \dif s
          + \frac12\|\sigma_n\|^2 t   \notag \\
      & +  \int_0^t (v_n,  dw^{(n)}(s)).\label{Ito-vn-L2}
\end{align}
Taking expectation of both sides of \eqref{Ito-vn-equa} we get
\begin{align*}
		\bbe  \|v_n(t)\|_{\calh}^2 + 2 \bbe \int_0^t \|(-\Delta)^{\frac \alpha 2} v_n(s)\|_\calh^2 \dif s
		\leq & \|x_0\|_\calh^2 + 2 \bbe \int_0^t \ _{\calv'}(f(s,w(s)), v_n(s))_{\calv} \dif s  \\
		& + \sum_{j=1}^{\infty}\lambda_j\|\epsilon_j\|^2_{\calh} t,
\end{align*}
which yields \eqref{Energy-ineq-SDE}.

Moreover, by the Burkholder-Davis-Gundy inequality,
\begin{align*}
   \bbe \sup\limits_{s\in [0,t]}
   |\int_0^s (v_n, dw^{(n)}(s))|
   \leq& C \bbe ( \int_0^t \|v_n\|_\calh^2  \|\sigma_n\|^2 \dif s)^{\frac 12} \\
   \leq& \frac 14 \bbe  \sup\limits_{s\in [0,t]}  \|v_n(s)\|_\calh^2
        + 4 C^2 \|\sigma_n\|^2 t.
\end{align*}
Taking into account \eqref{Ito-vn-L2}
we obtain
\begin{align*}
   & \bbe \sup\limits_{s\in [0,t]}  \|v_n(s)\|_\calh^2
   +  \bbe \int_0^t \|v_n\|_\calv^2 \dif s  \\
    \leq&  C' ( \int_0^t \bbe \sup\limits_{r\in [0,s]} \|v_n\|_\calh^2 \dif s
          + \bbe \int_0^t \|f(s,w(s))\|_{\calv'}^2 \dif s
          + \sum_{j=1}^{\infty}\lambda_j\|\epsilon_j\|^2_{\calh} t),
\end{align*}
where $C'$ is independent of $n$.
Thus, an application of Gronwall's inequality yields \eqref{vn-CtL2-L2V-bdd}.

\subsection*{Acknowledgments}
	The authors thank Martina Hofmanov\'a, Rongchan Zhu and Xiangchan Zhu
	for pointing out the work \cite{HZZ23}.
	E. Bru\'{e} would like to express gratitude for the financial support received from Bocconi University.
	Y. Li thanks the support by NSFC (No.11831011, 12161141004).
	D. Zhang  thanks the supports by NSFC (No.12271352, 12322108, 12161141004)
	and Shanghai Rising-Star Program 21QA1404500.
	Y. Li and D. Zhang are also grateful for the supports by
	Institute of Modern Analysis--A Shanghai Frontier Research Center.

\end{document}